\documentclass{amsart}
\usepackage{amsmath,amsthm,amsfonts,amssymb,amscd,verbatim} 
\usepackage{graphics}

\usepackage[all]{xy} 
\input xy

\usepackage{enumitem} 
\newtheorem{theorem}{Theorem}[section]
\newtheorem{lemma}[theorem]{Lemma}
\newtheorem{prop}[theorem]{Proposition}
\newtheorem{proposition}[theorem]{Proposition}
\newtheorem{corollary}[theorem]{Corollary}
\newtheorem{problem}[theorem]{Problem} 
 
\theoremstyle{definition}
\newtheorem{definition}[theorem]{Definition}
\newtheorem{question}[theorem]{Question}
\newtheorem{example}[theorem]{Example}

\theoremstyle{remark}
\newtheorem{remark}[theorem]{Remark}
\newcommand{\bu}{\bullet}
\newcommand{\ri}{\mathcal U}
\newcommand{\za}{\textnormal{Cent}_{\mathbb R}(A)}
\newcommand{\zaa}{\textnormal{Cent}^+_{\mathbb R}(A)}
\newcommand{\tr}{\textnormal{tr}}
\numberwithin{equation}{section}

\begin{document}
\title{Path methods for strong shift equivalence
of positive matrices}
\author{Mike Boyle}
\address{Mike Boyle\\
Department of Mathematics \\
University of Maryland\\
College Park, MD 20742-4015, U.S.A.}
\email{mmb@math.umd.edu}
\author{K. H. Kim}
\address{K. H. Kim  (deceased)\\
Mathematics Research Group\\
Alabama State University, Montgomery, AL 36101-0271, U.S.A.\\
and Korean Academy of Science and Technology}
\email{khkim@alasu.edu}
\author{F. W. Roush}
\address{F. W. Roush\\
Mathematics Research Group\\
Alabama State University, Montgomery, AL 36101-0271, U.S.A.}
\email{roushf3@hotmail.com}

\thanks{The authors thank Brendan Berg and Sompong Chuysurichay, 
for careful readings which removed some errata and otherwise 
improved the writing, and  also thank 
Richard Brualdi, for  Example \ref{brualdiexample} .} 

\date{}

 

\begin{abstract} 
In the early 1990's, Kim and Roush developed path methods 
for establishing strong shift equivalence (SSE) of positive 
matrices over a dense subring $\mathcal U$ of $\mathbb R$. 
This paper gives a detailed, unified and generalized 
presentation of these path methods.  
New arguments which address arbitrary dense subrings $\mathcal U$ 
of $\mathbb R$ are 
used to show that for any dense subring $\mathcal U$ of $\mathbb R$, 
positive matrices over $\mathcal U$ which have just 
one nonzero eigenvalue and which are strong shift equivalent over $\mathcal U$ must 
be strong shift equivalent over $\mathcal U_+$. In addition, we show 
matrices on a path of positive shift equivalent real matrices 
are SSE over $\mathbb R_+$;  positive rational matrices 
which are SSE over $\mathbb R_+$ must be 
SSE over $\mathbb Q_+$; and for any  dense subring $\mathcal U$ of 
$\mathbb R$, within the set of positive matrices over $\mathcal U$ 
which are conjugate over $\mathcal U$ to a given matrix, there 
are only finitely many SSE-$\mathcal U_+$ classes. 
\end{abstract}
\maketitle
\tableofcontents

\section{Introduction}


The classification problem for shifts of finite type (SFTs)
 remains a central open problem 
for symbolic dynamics.  
In the foundational work of Williams \cite{Wi,Wi2}
over forty years ago, 
the problem was recast as the following question: 
when are two matrices strong shift equivalent (SSE) over 
$\mathbb Z_+$? 
(The definition of SSE for matrices over a semiring is recalled below 
in Definition \ref{ssedefn}.
In this paper, rings and semirings are always assumed to contain 
1.)
Since then, SSE (involving other semirings) has been used for 
classification of other symbolic dynamical systems: for example, 
SFTs with Markov measure \cite{MT}, using matrices over Laurent 
polynomials; SFTs with a finite group action 
\cite{FN},
using matrices over the integral group ring of a finite group
\cite{BS}; and  
sofic shifts \cite{BK1,HN,KR1}, using a more complicated ring. 
Matsumoto has extended the ideas of SSE to a 
classification  setting for arbitrary subshifts \cite{Mat1,Mat2}. 

The original, notoriously difficult question of Williams 
for $\mathbb Z_+$ remains unanswered, 
and this is also 
a barrier to understanding the other classifications. 
One probe into this problem is to consider SSE over $\mathcal U_+$
for primitive matrices over a dense subring $\mathcal U$ 
of $\mathbb R$. 



Over a series of papers \cite{KR2,KR3,KR4,KR5} ending in 1992, 
Kim and Roush introduced  path methods for 
the study of strong shift equivalence of positive matrices over 
the reals and certain subrings of it.
One highlight of this work 
was the following theorem. For $\mathcal U $ 
any subfield of $\mathbb R$,  if $A$ and $B$ are 
square matrices over $\mathcal U_+$ which have  
eventual rank 1 (all large powers have rank 1) 
and have  the same 
nonzero eigenvalue, then $A$ and 
$B$ are SSE over $\mathcal U_+$. 
(As in Remark \ref{oneremark},  
there are in a sense no general results 
for greater eventual rank.) 
%
We extend this  
theorem to arbitrary dense subrings of $\mathbb R$, 
under the additional necessary condition that $A$ and $B$ 
are SSE over the ring $\mathcal U$. 
An additional  condition cannot be avoided: 
for general $\mathcal U$, matrices  with eventual rank 1
and the same nonzero eigenvalue 
need not be even shift equivalent over $\mathcal U$ to a 
$1\times 1$ matrix (see  Remark 
\ref{dedekindremark}). 
Whether the assumption of SSE over $\mathcal U$ 
is equivalent to the more tractable 
condition of shift equivalence over 
$\mathcal U$ remains an open question. 
%
However, our proof requires the assumption of SSE, not just 
SE, over $\mathcal U$. 

The central result of the path 
methods development was a Path Theorem for $\mathbb R$:
  matrices on a path of positive 
conjugate (similar) matrices must be SSE over $\mathbb R_+$. 
In this paper, we prove a generalized Path Theorem (\ref{paththeorem}) 
which has application to arbitrary dense subrings of $\mathbb R$.
We also show   that  matrices on a path of positive  matrices shift 
equivalent over 
$\mathbb R$ must be SSE over $\mathbb R_+$. 
This is a consequence of a more technical statement, the 
Connection Theorem (\ref{connectiontheorem}), which relies 
in turn  on a result which 
is pure linear algebra (Theorem \ref{matrixconnection}). 

One indication of the power of the path method 
comes from the corollary  due to 
 Chuysurichay (Theorem \ref{finitesse}): the set of positive matrices 
in a given  conjugacy (similarity) 
class over $\mathbb R$  contains only finitely many 
SSE-$\mathbb R_+$ classes. 
This holds even though 
(as shown by  Chuysurichay) 
it may be impossible 
to connect matrices in the class with SSEs with 
uniformly bounded lags  
(see Remark 
\ref{unboundedlag}).
Using our Path Theorem, we generalize this finiteness result 
 to arbitrary dense 
subrings of $\mathbb R$ (Theorem \ref{finitesseforu}). 

Using  the Connection Theorem, we are able to show that 
positive real matrices SSE over $\mathbb R_+$ are SSE 
 over $\mathbb R_+$ through positive matrices (Theorem 
\ref{positivetheorem}).  
As a consequence, 
primitive positive trace matrices 
over  a subfield  $\mathcal U$ of  $\mathbb R$
which 
are SSE over 
 $\mathbb R_+$ must also be  SSE over 
 $\mathcal U_+$ (Theorem  \ref{field}). 

Altogether,  for positive matrices SSE over a dense subring 
$\mathcal U$ of $\mathbb R$, 
 the current paper reduces the gap between 
 SSE over $\mathbb R_+$ and 
 SSE over $\mathcal U_+$, and provides further evidence 
for the utility of investigating SSE 
of positive matrices over $\mathbb R_+$. 
This is a problem to which more standard mathematics 
(e.g. fiber bundles, linear algebra) can be applied, 
as seen in \cite{KR1,KR2,KR3,KR4,KR5} and the current paper.
So, we suggest splitting into three parts 
the problem of understanding when two positive matrices 
over a dense subring $\mathcal U$ 
of $\mathbb R$ are SSE over $\mathcal U_+$: 
\begin{enumerate} 
\item 
Assuming $A$ and $B$  SSE over $\mathbb R$, 
prove they are SSE over $\mathbb R_+$. 
\item 
Assuming $A$ and $B$ SSE over $\mathbb R_+$, 
determine whether they are SSE over $\mathcal U_+$. 
\item 
Understand the refinement of SE over $\mathcal U$ 
by SSE over $\mathcal U$. 
\end{enumerate} 

We now say a little about the organization 
 of the paper.  

In Section \ref{elemsec}, we explain the decomposition 
of SSE into row splittings, column splittings and 
diagonal refactorizations, and provide 
some basic  technical results  
essential for the sequel. The results of this section 
hold over quite general rings, and refine the basic 
Williams theory. 
 
In Section \ref{fromsse}, given  positive matrices 
$A,B$ which are SSE over $\mathcal U$, we 
produce positive matrices $A',B'$ which are conjugate 
over $\mathcal U$ such that 
$A$ is SSE-$\mathcal U_+$ to $A'$ and 
$B$ is SSE-$\mathcal U_+$ to $B'$. 
(This is the step for which we need matrices 
SSE-$\mathcal U$, not just SE-$\mathcal U$.) 

In Section \ref{centralizersec}, we study 
$\textnormal{Cent}_{\mathbb R}(A)$, the group of invertible 
real matrices which commute with a given $n\times n$ real 
matrix $A$, and its group of connected components, 
$\pi_0( \textnormal{Cent}_{\mathbb R}(A) )$.  
This group plays a key role in the formulation of 
obstructions to applying the Path Theorem 
to produce SSE-$\mathcal U_+$. 

In Section \ref{frompathssec}, we prove the Path Theorem \ref{paththeorem} 
and some consequences. In Section \ref{findingsec}, we prove the 
eventually rank 1 results. 
In Section \ref{connectionsec}, we prove the Connection Theorem. 
A large part of the proof is an independent result in linear 
algebra, which we relegate to Appendix \ref{connectionappendix}. 
In Section \ref{sseoverplus}, we 
prove in particular that rational matrices SSE over $\mathbb R_+$ 
must be SSE over $\mathbb Q_+$. This is some supporting evidence 
for the conjecture \cite[Conj. 5.1]{B}
that positive rational matrices 
shift equivalent over $\mathbb Q_+$ are SSE over $\mathbb Q_+$. 

This paper is  entirely devoted to matrices. 
For background on shifts of finite type and symbolic dynamics, 
see \cite{Ki,LM}.

\section{Elementary splitting and strong shift equivalence }\label{elemsec}

Bob Williams introduced shift equivalence and strong shift equivalence
in his  paper \cite{Wi}, which is the foundation of all 
future work on the topic. One of the fundamental contributions was 
a decomposition of an elementary strong shift equivalence using 
even more fundamental relations, splittings and amalgamations. 
 In \cite{Wi}, 
 Williams considered 
matrices over $\mathbb Z_+$ and $\{0,1\}$.  For our work with 
unital nondiscrete subrings of $\mathbb R$, we need  some 
refinements to this work.

\begin{definition} \label{ssedefn}
Let $\mathcal U$ be a subset of a semiring containing 
0 and 1 (additive and multiplicative identities). 
Matrices $A,B$ are 
{\it elementary strong shift equivalent over $\mathcal U$}
(ESSE-$\mathcal U$) if there exist matrices $R,S$ over $\mathcal U$ 
such that $A=RS$ and $B=SR$. 
 Matrices $A,B$ are {\it strong shift equivalent over $\mathcal U$}  
(SSE-$\mathcal U$) if there exist matrices $A_0, A_1, \dots , A_{\ell}$, 
and for $1\leq i \leq \ell$ matrices $R_i,S_i$ over $\mathcal U$ such 
that $A_{i-1}=R_iS_i$ and $A_i=S_iR_i$, with $A_0=A$ and $A_{\ell}=B$. 
In this case the string $(R_i,S_i)$, $1\leq i \leq \ell$, is 
{\it a strong shift equivalence of lag $\ell$} from $A$ to $B$. 
\end{definition} 

Although we do not use shift equivalence before Section \ref{connectionsec}, 
to clarify ideas we recall its basic features now. 

\begin{definition} \label{sedefn} 
Let $\mathcal U$ be a subset of a semiring containing 
0 and 1 (additive and multiplicative identities). 
Matrices $A,B$ are 
{\it  shift equivalent over $\mathcal U$}
(SE-$\mathcal U$) if there exist matrices $R,S$ over $\mathcal U$ 
and $\ell \in \mathbb N$ such that the following hold: 
\begin{align*} 
A^{\ell} =RS\ \qquad \qquad B^{\ell} &=SR \\ 
AR = RB  \qquad \qquad BS &=SA \ . 
\end{align*}
\end{definition} 

Always, SE-$\mathcal U$ implies SSE-$\mathcal U$. 
The converse is true if $\mathcal U$ is a Dedekind domain \cite{BH2}
(e.g., a field or $\mathbb Z$, \cite{E,Wi2}). 
For primitive matrices $A,B$ over a subring of $\mathbb R$:  
$A,B$ are SE-$\mathcal U$ if and only if  
 $A,B$ are SE-$\mathcal U_+$ . 
Over $\mathcal U$ a subfield of $\mathbb R$, matrices are shift equivalent 
if and only if the nonsingular parts of their Jordan forms are the same. 
There is a ``conceptual'' version of shift equivalence, in terms of 
isomorphism of associated dimension modules.

Williams asked whether the relatively tractable relation 
SE-$\mathbb Z_+$ implies 
SSE-$\mathbb Z_+$.   
Working in the framework of Wagoner's 
algebraic topological framework the classification 
problem \cite{Wa}, Kim and Roush gave examples of 
primitive matrices over $\mathbb Z$ 
which are SSE over $\mathbb Z$ (equivalently, shift 
equivalent over $\mathbb Z$) but not SSE over 
$\mathbb Z_+$ \cite{KR6}. There are also 
examples of positive matrices over a dense subring $\mathcal U$ 
of $\mathbb R$ which are SSE over $\mathcal U$ but are not 
SSE over $\mathcal U_+$ 
(see Remark \ref{z1overpremark}). A feature of Wagoner's 
framework is that it is built up out of elementary SSEs, 
not out of SE. We will see the same feature in the 
proof of Theorem \ref{ssetosim}. For a ring $\mathcal U$, 
 can be useful to 
take SSE-$\mathcal U$ as a hypothesis, and leave the 
question of whether SE-$\mathcal U$ implies 
SSE-$\mathcal U$  as a separate issue. 

We turn away now from shift equivalence, until Section 
\ref{connectionsec}. 
For a ring $\mathcal U$, 
\begin{definition} 
An {\it amalgamation matrix} is a matrix with entries from $\{0,1\}$  such that every row has exactly 
one 1 and every column has at least one 1. A {\it subdivision matrix} is the transpose of an amalgamation matrix. 
\end{definition} 

\begin{definition} 
An {\it elementary row splitting} is an 
elementary strong shift equivalence 
$UX=A,XU=C$ in which $U$ is a subdivision matrix. 
In this case, $C$ is an {\it elementary row splitting
of $A$}, and $A$ is an {\it elementary row amalgamation of $C$}. 
\end{definition} 

\begin{definition} 
An {\it elementary column splitting} is an 
elementary strong shift equivalence 
$XV=A,VX=C$ in which $V$ is an amalgamation matrix. 
In this case, $C$ is an {\it elementary column splitting
of $A$}, and $A$ is an {\it elementary column amalgamation of $C$}. 
\end{definition} 

 Given $a_1+a_2+a_3=a,\  b_1+b_2+b_3=b,\ c_1+c_2=c,\ d_1+d_2=d$,  
here is an elementary row splitting $C$ of $A$: 
\begin{align*} 
UX = 
\begin{pmatrix} 
1 & 1 & 1& 0 & 0 \\
0 & 0 & 0 & 1 &1
\end{pmatrix} 
\begin{pmatrix} 
a_1 & b_1 \\
a_2 & b_2 \\
a_3 & b_3 \\
c_1 & d_1 \\
c_2 & d_2
\end{pmatrix} 
&=
\begin{pmatrix} a &b \\ c & d 
\end{pmatrix} 
= A \\ 
XU = 
\begin{pmatrix} 
a_1 & b_1 \\
a_2 & b_2 \\
a_3 & b_3 \\
c_1 & d_1 \\
c_2 & d_2
\end{pmatrix} 
\begin{pmatrix} 
1 & 1 & 1& 0 & 0 \\
0 & 0 & 0 & 1 &1
\end{pmatrix} 
& =  
\begin{pmatrix} 
a_1 &a_1 &a_1 & b_1& b_1 \\
a_2 &a_2 &a_2 & b_2& b_2 \\
a_3 &a_3 &a_3 & b_3& b_3 \\
c_1 &c_1 &c_1 & d_1& d_1 \\
c_2 &c_2 &c_2 & d_2& d_2
\end{pmatrix} = C 
\end{align*}
For an elementary row splitting, rows of $A$ are split as sums 
of rows (as described by $X$), and then columns of $X$ are ``copied''
in such a way that indices of rows in $C$ with the same ``parent'' row in $A$ 
have equal columns in $C$. We say a row in $C$ 
in $A$ is {\it sitting above} its parent row.

Similarly, here is an example of an elementary column splitting. 
\begin{align*} 
YV= & 
\begin{pmatrix} 
-3.4&1.4&5 \\
\pi & -\pi &6
\end{pmatrix} 
\begin{pmatrix} 
1&0\\1&0\\0&1 
\end{pmatrix} 
= 
\begin{pmatrix} 
-2& 5 \\ 
0 & 6 
\end{pmatrix} 
= A \\ 
VY = & 
\begin{pmatrix} 
1&0\\1&0\\0&1 
\end{pmatrix} 
\begin{pmatrix} 
-3.4&1.4&5 \\
\pi & -\pi &6
\end{pmatrix} 
= 
\begin{pmatrix} 
-3.4&1.4&5 \\
-3.4&1.4&5 \\
\pi & -\pi &6
\end{pmatrix} 
\end{align*}
Here, columns 1 and 2 of $VY$ are sitting above column 1 of $A$ in 
the column splitting. 

\begin{definition} 
A matrix is {\it nondegenerate } if it has no zero row and 
it has no zero column. 
\end{definition} 

\begin{definition}  A {\it diagonal refactorization over a semiring $\mathcal U$}
 is an elementary 
strong shift equivalence over $\mathcal S$ of the form $A=DX$, $B=XD$, 
where $D$ is nondegenerate 
diagonal over $S$. In this case, $A$ is a 
{\it diagonal refactorization of $B$} 
(and vice versa). 
\end{definition}

We now recall the canonical factorization of a nondegenerate matrix 
introduced by Williams {\cite{Wi}).

Suppose $M$ is a nondegenerate matrix over a semiring $\mathcal U$ 
containing $\{0,1\}$, 
with rows indexed by the set $\mathcal I$ and columns indexed by the 
set $\mathcal J$. 
Let $\mathcal E$ be the set of pairs $(i,j)$ such that $M(i,j)\neq 0$.  
Let $U_M$ be the $\mathcal I \times \mathcal E$
subdivision matrix such that $U_M(i',(i,j))=1$ iff $i'=i$. 
Let $V_M$ be the $\mathcal E \times \mathcal J$ amalgamation matrix such that 
$V_M((i,j),j')=1$ iff $j=j'$. Let $D_M$ be the $\mathcal E \times \mathcal E$ 
diagonal matrix such that $D_M((i,j),(i,j))= M(i,j)$. 
Then 
\[
M = U_MD_MV_M \ .
\] 
Because $M$ is nondegenerate, $U_M$ and $V_M$ are defined (e.g., given $i$
there is at least one $j$ such that $D(i,j)\neq 0$, so row $i$ of $U_M$ has 
at least one 1), and $D$ has nonzero diagonal entries.   

There is a graphical interpretation of the factorization 
$M = U_MD_MV_M $. The set $\mathcal E$ can be viewed as the set of 
edges of a directed graph, in which there is an edge from $i$ to $j$
if $M(i,j)$ is nonzero. The matrices $U_M$ and $V_M$ attach (respectively)
initial and terminal vertices to edges, and $D_M$ records the entry of 
$M$ labeling the edge. 

\begin{definition}
We call the factorization $M = U_MD_MV_M$ above 
of a nondegenerate matrix $M$  
the {\it Williams factorization of } $M$. It is well defined up to 
the choice of ordering of indices used for $D_M$ (and thus $U_M$ and 
$V_M$). 
\end{definition} 
If $M$ is a nondegenerate matrix and $M=UDV$ with 
$U$ subdivision, $D$ nondegenerate diagonal and $V$ amalgamation, 
then $M=UDV$ must be the Williams factorization described above. 

We may avoid the complications of defining a  factorization 
for degenerate matrices, on account of the following proposition.

\begin{proposition} \label{nondegeneratesseprop} 
Suppose $\ri$ is a ring which is torsion free as an additive group. 
Suppose nondegenerate matrices $A$ and $B$ are SSE over $\ri$. 
Then they are SSE through a chain of ESSEs 
$A_{i-1} =R_iS_i, A_{i+1}=S_iR_i$ such that all the matrices 
$A_i, R_i,S_i $ are nondegenerate. 
\end{proposition} 

The proof of Proposition \ref{nondegeneratesseprop} is a digression, 
and we give it in Appendix \ref{nondegenerateappendix}.  


\begin{proposition} 
\label{decompprop}
Suppose $A=RS, B=SR$ 
is an elementary strong shift equivalence  
over a semiring $\mathcal U$ containing $\{0,1\}$; 
$\mathcal U$ has no zero divisors; and the matrices 
$A$, $B$ are nondegenerate.   
Then there are nondegenerate matrices $C_1,C_2,D$ 
over $\mathcal U$ such that 
$D$ is diagonal and 
\begin{enumerate} 
\item 
$C_1$ is an elementary row splitting of $A$ 
\item 
There is a matrix $X$ over $\mathcal U$ such that 
$DX=C_1$ and $XD=C_2$ \newline 
(so, $C_1$ is a diagonal refactorization of $C_2$) 
\item 
$C_2$ is an elementary column splitting of $B$. 
\end{enumerate} 
\end{proposition} 

\begin{proof} 

Using the Williams factorization above, we have 
\begin{align*}
A &= (U_RD_RV_R)\ ( U_SD_SV_S) \\ 
B &= (U_SD_SV_S)\ ( U_RD_RV_R) \ . 
\end{align*} 
Define 
\begin{align*}
C_1 &= (D_RV_RU_SD_SV_S)\  U_R  \ , \quad 
X_1 = D_RV_RU_SD_SV_S 
\\ 
C_2 &= 
V_R\ (U_SD_SV_SU_RD_R ) 
\ , \quad 
X_2 = U_SD_SV_SU_RD_R\ . 
\end{align*} 
Set $D=D_R$ and 
  $X= V_RU_SD_SV_SU_R$.  Then 
 $DX=C_1$ and $XD=C_2$, proving (2). Also, 
\begin{align*} 
A   &= 
 U_RX_1 
\qquad \qquad 
B = 
X_2 V_R
\\
C_1 &
=  X_1U_R 
\qquad \qquad 
C_2 
= V_RX_2\ . 
\end{align*}
This proves (1) and (3). It remains to prove the nondegeneracy 
claims.

The matrix $D=D_R$ is nondegenerate by construction.  
The matrix $X_2$ has no zero row, because $B=X_2V_R$ has 
no zero row. The matrix $(U_SD_SV_S)$ has no zero column because 
$A$ has no zero column. Because $U_R$ is a subdivision matrix, 
the matrix $(U_SD_SV_S)U_R$ then has no zero column. Because 
there are no zero divisors, the matrix 
$(U_SD_SV_S)U_RD_R=X_2$ has no zero column. Thus $X_2$ is nondegenerate. 
Because $V_R$ is an amalgamation matrix, the matrix 
$C_2=V_RX_2$ is also nondegenerate. Similarly, $C_1$ is nondegenerate. 
This proves the proposition. 
\end{proof}

\begin{remark} \label{diagramremark} 
If (for example) $\mathcal U$ is a subring of the reals, 
then in Proposition \ref{decompprop} 
the matrix $D^{-1}$ is defined 
over $\mathbb R$ and 
the matrices $D^{-1}C_1$ and $C_2D^{-1}$ have entries in $\mathcal U$. 
Then we can summarize the proposition with a diagram 
\begin{equation} \label{elsplitdiagram} 
\xymatrix{ 
       & C_1 \ar[rr]^-{(D, D^{-1}C_1)} \ar[dl]_{(X,U)}   &  & D^{-1}C_1D \ar[dr]^{(V,Y)}  &     \\ 
A      &                                          &  &                   &  B 
}
\end{equation}
in which an arrow labelled $(J,K)$ from $M$ to $N$ represents an elementary 
strong shift equivalence $M=JK$,  $N=KJ$; $U$ is a subdivision matrix; and 
$V$  is an amalgamation matrix. $C_1$ is an elementary row splitting of 
$A$ and $D^{-1}C_1D$ is an elementary column splitting of $B$. 
\end{remark} 

For matrices over $\{0,1\}$, the next lemma is well known 
(\cite{Pa}, 
\cite[Theorem 2.1.14]{Ki}),  and can be interpreted 
as a fiber product statement.

\begin{lemma}[Fiber Lemma]  \label{fiberlemma}
Suppose over 
a ring $\mathcal U$ there is an elementary row splitting of 
$A$ to a nondegenerate $C_1$ and an elementary column splitting of $A$ to 
a nondegenerate $C_2$. Then there is a 
nondegenerate matrix $F$ such that 
over $\mathcal U$ there is an elementary  
column splitting of $C_1$ to
$F$ and an elementary  row 
splitting of $C_2$ to $F$. 

If all entries of $C_1$ and $C_2$ are nonnegative, or  positive, entries 
in a nondiscrete unital subring $\mathcal U$ of $\mathbb R$, 
then all entries of $F$ can 
be chosen to have nonnegative, or positive, entries in $\mathcal U$. 
\end{lemma} 
\begin{proof} 
Let $\mathcal I$ denote the set indexing the rows and columns of $A$. 
For $s=1,2$ let $\mathcal I_s$ be the index set for the rows and 
columns of $C_s$. The index set for the rows and columns of $F$ will 
be the set 
\[
\mathcal V:= \{(i_1,i_2)\in \mathcal I_1 \times \mathcal I_2: 
\overline{i_1}=
\overline{i_2}=i \} 
\]
where $\overline{i_s}$ denotes the element of $\mathcal I$ associated 
to $i_s$ under the given elementary splitting of $A$ to $C_s$. 
For $i\in \mathcal I$, let $\mathcal V(i)=\{(i_1,i_2)\in \mathcal V: 
\overline{i_1}=\overline{i_2}\}$. 
Let $F_{<i,j>}$ denote the 
submatrix of $F$ with index set 
$\mathcal V(i)\times \mathcal V(j)$. 
Given $i\in \mathcal I$, we let $\mathcal I_s(i)$ denote the set of indices 
$i_s$ in $\mathcal I_s$ such that $\overline{i_s}=i$. 

We will define $F$ by defining 
$F_{<i,j>}$ for each $i,j$. 
So, consider now $i,j$ from $\mathcal I$.
For notational simplicity, suppose 
for the definition of $F_{<i,j>}$ that 
$\mathcal I_1(i)=\{1,\dots ,m\}$ and 
$\mathcal I_2(j)=\{1,\dots ,n\}$.
We will define an $m\times n$ matrix $M=M_{<i,j>}$ 
and then set 
$F_{<i,j>}((i_1,i_2),(j_1,j_2))= M(i_1,j_2)$. 

Let $a$ denote $A(i,j)$.   By the nature of row splitting, 
there is a vector $\alpha=(\alpha_1, \dots ,\alpha_m)$ over 
$\mathcal U$ such that $C_1(s,k)=\alpha_s$ for every $k\in 
I_1(j)$ and $1\leq s \leq m$. Likewise, 
there is a vector $\beta=(\beta_1, \dots ,\beta_n)$ over 
$\mathcal U$ such that $C_2(k,t)=\beta_t$ for every $k\in 
I_2(i)$  and $1\leq t \leq n$. Also, 
$\sum_s \alpha_s = a = \sum_t\beta_t$. 

We now arrange that the vector of row sums of $M$ is $\alpha$ and the vector of 
column sums of $M$ is $\beta$. 
(In the special case that $\mathcal U$ is a field, for 
$a\neq 0$ we could simply set 
$M(s,t)=(1/a) \alpha (s) \beta (t)$.) 
If $m=1$, we necessarily set $M(1,t)= \beta(t)$, $1\leq t \leq n$.  
If $n=1$, we likewise set  $M(s,1)= \alpha(s)$, $1\leq s \leq m$. 
If $m=n=1$, the two definitions coincide. 
If $m$ and $n$ are greater than $1$, 
pick an $(m-1)\times (n-1)$ matrix $N$ over $\mathcal U$
and for $1\leq s<m$ and $1\leq t <n$ define 
$M(s,t) =N(s,t)$. Then for $1\leq s < m$, 
define $M(s,n)$ so that the $s$th row sum is 
$\alpha (s)$, and 
for $1\leq t < n$ 
define the entries $M(m,t)$ so that the $t$th column sum is 
$\beta (j_t)$. These additional entries must lie in 
the ring $\mathcal U$. 
 Finally define  $M(m,n)$ so that the sum of 
the entries of $M$ is $a$. Necessarily $M(m,n)$ is in $\mathcal U$. 
The $m$th row sum of $M$ is $\alpha(m)$ because it equals 
$a$ minus the sum of the other row sums $\alpha(1), \dots , 
\alpha(m-1)$. Similarly the $n$th column sum of $M$ is $\beta(n)$. 

In the case that $C_1$ and $C_2$ have nonnegative real entries and 
$m> 1$ and $n>1$, if $a=0$ then set $M=0$. 
If $a>0$, then for 
notational convenience suppose $\alpha_m\beta_n >0$. 
Then choose $N$ above such that 
$N(s,t) =0$ whenever $\alpha(s)\beta(t)=0$ and otherwise 
\[
0 <  (1/a) \alpha(s)\beta(t)-N(s,t) < \epsilon 
\]
where $\epsilon $ is small enough to guarantee that 
$M(m,n)>0$. Then $M$ will be nonnegative, and $M$ will be positive 
if $\alpha$ and $\beta$ are positive. 
This finishes the definition of $M$ and $F$.

Now define a $\mathcal V \times \mathcal I_1$
amalgamation matrix $V$ and an 
$\mathcal I_1\times \mathcal V$
matrix $X$ by the rules 
\begin{align*} 
V\Big((i_1,i_2),i_3\Big)& =1 \quad \ \textnormal{iff }i_1=i_3 \\ 
X\Big(i_1,(j_1,j_2)\Big)& = M_{<i,j>}(i_1,j_2) \ , \quad 
\textnormal{where }(i,j)=(\overline{i_1},\overline{j_2})\ .  
\end{align*} 
Then 
\begin{align*} 
(XV)(i_1,j_1) & = 
\sum_{\{j_2:\  (j_1,j_2)\in \mathcal V\}} 
X(i_1,(j_1,j_2))V((j_1,j_2),j_1) \\ 
&=\sum_{
\{j_2: \ \overline{j_1}=\overline{j_2}\}}
M_{<i,j>}(i_1,j_2) 
\ , \quad 
\textnormal{where }(i,j)=(\overline{i_1},
\overline{j_2})\ , 
\\ 
&= \alpha(i_1) = C_1(i_1,j_1) \ . 
\end{align*} 
Similarly, 
\begin{align*}  
(VX)((i_1,i_2),(j_1,j_2)) 
&= X(i_1,(j_1,j_2)) \\ 
&= M_{<i,j>}(i_1,j_2) \ , \quad 
\textnormal{where }(i,j)=(\overline{i_1},\overline{j_2})\ ,   \\ 
&=F((i_1,i_2),(j_1,j_2)) \ . 
\end{align*} 
Thus $C_1=VX$ and $F=XV$, and $F$ is an elementary 
column splitting of $C_1$.

Likewise, $F$ is an elementary row splitting of $C_2$. Define an $\mathcal I_2 
\times \mathcal V$ subdivision matrix $U$ and a $\mathcal V \times \mathcal I_2$ 
matrix $Y$ by the rules 
\begin{align*}
U(i_3,(i_1,i_2))&= 1 \quad \textnormal{ iff } i_3=i_2 \\ 
Y((i_1,i_2),j_2) &= 
M_{<i,j>}(i_1,j_2) \ , \quad 
\textnormal{where }(i,j)=(\overline{i_1},\overline{j_2})\ . 
\end{align*}
Then $F=YU$ and $C=UY$, by a similar computation. 

Finally, suppose $C_1$ and $C_2$ are nondegenerate. 
Then $F$ has no zero column (being a row splitting of $C_2$) 
and $F$ has no zero row (being a column splitting of $C_1$), 
so $F$ is nondegenerate. 
\end{proof} 

\begin{lemma}\label{diagrefactoconj}
Suppose $\mathcal U $ is a unital ring, $A$ and $B$ are $n\times n$ 
matrices over $\mathcal U$, and there is a nondegenerate diagonal 
matrix $D$ and a matrix $C$ such that $A=DC$ and $B=CD$. 

Then there are matrices $A',B'$ over $\mathcal U$ such that 
$A'$ is an elementary row splitting of $A$, 
$B'$ is an elementary column splitting of $B$, 
and $A'$ is conjugate over $\mathcal U$ to $B'$. 
If $A$ and $B$ are nondegenerate and the ring $\mathcal U$ 
has no zero divisors, then the matrices $A',B'$ 
can be chosen nondegenerate. 
\end{lemma} 

\begin{proof} 
If $D=I_n$, we are done, so suppose not. For notational simplicity, suppose there is 
a positive integer $k$ such that $D(i,i) =1$ iff $i>k$. Suppose 
$k<n$.  Let $E$ denote the 
$k\times k$ upper left corner of $D$. Then in block form, for 
some matrices $C_i$ over $\mathcal U$ (with $1\leq i\leq 4$ and $C_1$ $k\times k$) 
we have 
\[
A = \begin{pmatrix} EC_1 & EC_2\\C_3&C_4 \end{pmatrix} \ , \qquad 
B = \begin{pmatrix} C_1E & C_2\\C_3E&C_4 \end{pmatrix} \ . 
\] 
An elementary row splitting of $A$ to an $(n+k)\times (n+k)$ matrix $A'$ 
is given by 
\begin{align*}
A & 
 = \begin{pmatrix} I_k & I_k& 0 \\0 & 0 & I_{n-k} \end{pmatrix} 
 \begin{pmatrix} C_1 & C_2\\(-I_k+E)C_1&(-I_k+E)C_2\\C_3 & C_4 \end{pmatrix}
=\begin{pmatrix} EC_1 & EC_2\\C_3&C_4 \end{pmatrix}  \\ 
A'
&= \begin{pmatrix} C_1 & C_2\\(-I_k+E)C_1&(-I_k+E)C_2\\C_3 & C_4 \end{pmatrix}
 \begin{pmatrix} I_k & I_k& 0 \\0 & 0 & I_{n-k} \end{pmatrix}  \\ 
&=\begin{pmatrix} C_1 & C_1 & 
C_2\\(-I_k+E)C_1&(-I_k+E)C_1&(-I_k+E)C_2\\C_3 &C_3 & C_4 \end{pmatrix} \ .
\end{align*} 
An elementary column splitting of $B$ to an $(n+k)\times (n+k)$ matrix $B'$ 
is given by 
\begin{align*}
B & = 
 \begin{pmatrix} C_1 & C_1(-I_k+E)& C_2\\ C_3 & C_3(-I_k+E) &C_4 \end{pmatrix} 
\begin{pmatrix} I_k & 0 \\ I_k& 0 \\ 0 & I_{n-k} \end{pmatrix} 
=\begin{pmatrix} C_1E & C_2\\C_3E&C_4 \end{pmatrix}  \\ 
B'&= 
\begin{pmatrix} I_k & 0 \\ I_k& 0 \\ 0 & I_{n-k} \end{pmatrix} 
 \begin{pmatrix} C_1 & C_1(-I_k+E)& C_2\\ 
C_3 & C_3(-I_k+E) &C_4 \end{pmatrix} \\ 
&=  \begin{pmatrix} C_1 & C_1(-I_k+E)& C_2\\ C_1 & C_1(-I_k+E)& C_2\\ 
 C_3 & C_3(-I_k+E) &C_4 \end{pmatrix} \ .
\end{align*} 
Define 
\[ 
W =\begin{pmatrix} 0 & I_k & 0 \\ I_k & E-2I_k & 0 \\ 0 & 0 & I_{n-k}
\end{pmatrix}\ , \textnormal{ with inverse }\  
W^{-1} =  
\begin{pmatrix} 2I_k -E & I_k & 0 \\ I_k & 0 & 0 \\ 0 & 0 & I_{n-k}
\end{pmatrix}\ . 
\]  
A computation shows $A'W = WB'$. If $A$ and $B$ are nondegenerate and 
$\mathcal U$ has no zero divisors,  
then the constructed matrices $A'$ and $B'$ are nondegenerate. 
This finishes the proof for the case $k<n$. If $k=n$, then  
simply remove block rows and columns through $C_4$ from the proof above, 
and repeat the proof with $D$ in place of $E$. 
\end{proof}

\section {From strong shift equivalence to conjugacy} \label{fromsse}

 Let $\mathcal U$ be a nondiscrete unital subring of 
$\mathbb R$. Two $n\times n$ matrices   $A$ and $B$ 
with entries in 
$\mathcal U$ are {\it  conjugate over $\mathcal U$}, 
or {\it similar over $\mathcal U$}, if there 
exists $W$ in $\textnormal{GL}(n,\mathcal U)$ such that 
$W^{-1}AW =B$. 

The purpose of this section is to prove 
the following theorem.

\begin{theorem} \label{ssetosim} 
Let $\mathcal U$ be 
a nondiscrete unital subring 
of $\mathbb{R}$. Suppose 
 $A,B$ are positive matrices over $\mathcal U$ 
and are strong shift equivalent over $\mathcal U$. 
Then $A$ and $ B$ are 
strong shift equivalent over $\mathcal U_+$ to positive matrices 
which are conjugate  over $\mathcal U$.

Moreover, the conjugating matrix can be chosen
to have positive determinant and to send positive 
eigenvectors to positive eigenvectors.
\end{theorem} 

We begin with the main lemma. We use the following notation: 
$I_k$ denotes the $k\times k$ identity matrix. 

\begin{lemma}[Splitting Lemma]
\label{positivitysplittinglemma}
Let  $\mathcal U$ be a nondiscrete unital subring of 
$\mathbb R$.  
Suppose the following: 
\begin{itemize} 
\item 
$A$ and $C$ are $n\times n$ matrices over $\mathcal U$
\item 
$W$ is a matrix in $\textnormal{GL}(n,\mathcal U)$ such that  
$W^{-1}AW=C $
\item 
$C'$ is obtained from $C$ by a finite sequence of 
  row 
splittings over $\mathcal U$. 
\end{itemize} 
Then the following hold. 
\begin{enumerate} 
\item 
There is a matrix $A'$ conjugate  over $\mathcal U$ 
to $C'$ such that $A'$ is obtained from $A$
by a finite sequence 
of row splittings  over $\mathcal U$; 
and such that,  if $A$ and 
$C'$ are nondegenerate,  
then  $A'$ is nondegenerate. 
\item 
If $A$ is a positive matrix, then there is a 
positive matrix $A^+$ over $\mathcal U$ 
such that $A^+$ is obtained 
 from $A$ by a finite sequence 
of row splittings of positive matrices over 
$\mathcal U_+$, and $A^+$ is conjugate over 
$\mathcal U$ to a matrix of the form 
$\left(\begin{smallmatrix} C' & 0 \\ 0& 0 
\end{smallmatrix}\right)$.  
\end{enumerate} 
The lemma statement is also true with ``row'' replaced by 
``column''. 
\end{lemma} 
 
{\begin{proof} 
Any row splitting to a larger matrix 
is a composition of row splittings which 
increase the matrix size by exactly one.  
So, we have some positive integer $\ell$ and 
a finite sequence of elementary row splittings 
of matrices $C_i $ to $C_{i+1}$,  $0\leq i < \ell$, 
with $C_{i+1}$ obtained by splitting one row of $C_i$ to two rows, 
and with $C=C_0$ and $C'=C_{\ell}$.  

{\it Proof of Claim (1)} 
We first consider the case that $\ell = 1$. 
For notational convenience, suppose row $n$ of $C$ is split 
into rows $n$ and $n+1$ of $C'$. For any matrix $B$, we let 
$B_{\textnormal{row}(i)}$ 
 denote its $i$th row. We have matrices 
\[
X' = \begin{pmatrix} C_{\textnormal{row}(1)}\\ \vdots 
\\C_{\textnormal{row}(n-1)} \\ s' \\t' 
\end{pmatrix} \quad , \quad 
U' = \begin{pmatrix} I_{n-1} & 0 & 0 \\ 0 & 1 & 1 
\end{pmatrix} 
\] 
such that $U'X'=C$ and $X'U'=C'$. Set $t = t'W^{-1}$ and 
set $s= A_{\textnormal{row}(n)}-t$ and define the $(n+1)\times n$ matrix 
\[
Y' = 
\begin{pmatrix} A_{\textnormal{row}(1)}\\ \vdots \\A_{\textnormal{row}(n-1)} \\ s \\t 
\end{pmatrix}\ . 
\] 
Then $U'Y'=A$. Define $A'= Y'U'$, an elementary row splitting of $A$. 
Let $E$ be the $(n+1)\times (n+1)$ matrix equal to $I_{n+1}$ except that  
$E(n,n+1)=-1$.  Then we have matrix equations (in block forms) 
\begin{align*} 
E^{-1}A'E 
&= 
\begin{pmatrix} A & 0 \\ t & 0 \end{pmatrix}  \\
E^{-1}C'E 
& = 
\begin{pmatrix} C & 0 \\ t' & 0 \end{pmatrix} \\ 
\begin{pmatrix} A & 0 \\ t & 0 \end{pmatrix}  
\begin{pmatrix} W & 0 \\ 0 & 1 \end{pmatrix} 
&= 
\begin{pmatrix} W & 0 \\ 0 & 1 \end{pmatrix} 
\begin{pmatrix} C & 0 \\ t' & 0 \end{pmatrix} \ . 
\end{align*}  
Therefore $A'$ is conjugate over $\mathcal U$ to $C'$. 

At the inductive step, going from $\ell -1$ to $\ell$, 
we apply the same argument to matrices $A'_{\ell -1}$ and 
$C'_{\ell -1}$ given by the induction hypothesis. 

Now suppose $A$ is nondegenerate.  Then no sequence
of row splittings of $A$ can produce a matrix with  a 
zero column. If $C'$ is nondegenerate, then we can 
choose all those row splittings $C_i$ to $C_{i+1}$, 
splitting some row as a sum $s_i + t_i$, such that 
$s_i\neq 0 \neq t_i$. Then the construction, splitting 
$A'_i$ to $A'_{i+1}$, never introduces a zero row, 
and in the end $A'$ will be nondegenerate. 
This completes the proof of (1). 
\newline \newline 
{\it Proof of Claim (2).} 
As in part (1), we first consider the case $\ell =1$. 
Let $s',t',s,t$ be as in part (1).  
Define  matrices 
\[ 
U = 
\begin{pmatrix} I_{n-1} & 0 & 0 & 0 \\ 0 & 1 & 1 & 1 
\end{pmatrix} \ , \quad 
X'' = 
\begin{pmatrix} C_{\textnormal{row}(1)} \\ \vdots \\ C_{\textnormal{row}(n-1)} \\ s' \\ t' \\ 0 
\end{pmatrix} \ , \quad 
Y'' = 
\begin{pmatrix} A_{\textnormal{row}(1)} \\ \vdots \\ A_{\textnormal{row}(n-1)} \\ s \\ t \\ 0 
\end{pmatrix} \ . 
\] 
Set $A''= Y''U$. Let $F$ be the $(n+2)\times (n+2) $ 
matrix equal to $I_{n+2}$ except that 
$F(n,n+1)=F(n,n+2)= -1$. Then 
\[
F^{-1}A''F = 
\begin{pmatrix} A & 0 & 0  \\ t & 0 & 0 \\ 0 & 0 & 0 
\end{pmatrix} \quad \textnormal { and  } \quad 
F^{-1}C''F = 
\begin{pmatrix} C & 0 & 0  \\ t' & 0 & 0 \\ 0 & 0 & 0 
\end{pmatrix} \ .  
\] 
The matrix $ F^{-1}C''F $ 
(and therefore $A''$) 
is conjugate over $\mathcal U$ to the 
$(n+2)\times (n+2)$ matrix 
$\left(\begin{smallmatrix} C' & 0 \\ 0& 0 
\end{smallmatrix}\right) $.  

It remains to conjugate $A''$ over $\mathcal U$ to a 
matrix 
$A^+$ which is the required row splitting of $A$. 
 For this we will pick a suitable invertible 
 $3\times 3$ matrix $M$ with all column sums equal to 1, define  
$W^+$ to be $\begin{pmatrix} I_{n-1} & 0 \\ 0 &M\end{pmatrix}$ 
and set $A^+$ equal to $W^+ A''(W^+)^{-1} $. 
Because $M$ has all column sums 1 (i.e. fixes the row vector with 
every entry 1), it follows that $M^{-1}$ has all column sums 1, and therefore 
$U(W^+)^{-1}=U$. Consequently, 
\[
A^+  =  W^+ Y''U(W^+)^{-1}  =  W^+ Y''U
= 
\begin{pmatrix} I_{n-1} & 0   \\  0 & M 
\end{pmatrix} Y''
U  \ . 
\]
The matrix $M$ will have the form 
\begin{equation}\label{Mdefn}
\begin{pmatrix} 
x_1 & x_1 + \epsilon_1 & z_1 \\
x_2 & x_2 + \epsilon_2 & z_2 \\
1-(x_1+x_2) & 1-(x_1+x_2)-(\epsilon_1+\epsilon_2) & 1-(z_1+z_2)
\end{pmatrix} 
\end{equation}
and therefore the bottom three rows of $W^+Y''$ will equal 
\begin{equation}\label{onethird}
M 
\begin{pmatrix} s \\ t\\ 0 
\end{pmatrix} 
= 
\begin{pmatrix} 
x_1(s+t) + \epsilon_1t \\ 
x_2(s+t) + \epsilon_2t \\ 
[1- (x_1+x_2)](s+t) - [\epsilon_1+\epsilon_2]t  
\end{pmatrix} \ . 
\end{equation}  
These three rows  sum to $s+t$, 
which is row $n$ of $A$. Thus $U(W^+Y'')=A$, and 
$A^+=(W^+Y'')U$ is an elementary 
row splitting of $A$. 

We now complete the definition of $M$. 
Given $\gamma>0$, pick positive numbers $x_1, x_2, \epsilon_1$ from 
$\mathcal U$ with $|x_1-1/3|, |x_2 - 1/3|$ and $\epsilon_1$ all smaller 
than $\gamma$.  Pick $K\in \mathbb N$ 
such that $K\epsilon_1 \leq 1< (K+1)\epsilon_1$ and set $\epsilon_2= 1-K\epsilon_1
< \epsilon_1$. For small $\gamma$, this guarantees that $A_+$ is positive 
(the rows in (\ref{onethird}) are approximately 
$(1/3)A_{\textnormal{row}(n)}$).   
 Define $z_1 = -1+x_1$ and $z_2 = K+x_2$. 
A computation shows 
\begin{align*}  
\det(M) &= \epsilon_1(z_2-x_2) - \epsilon_2 (z_1-x_1) \\ 
&= \epsilon_1(K) - \epsilon_2(-1) = 1 \ . 
\end{align*} 
Therefore $M\in \textnormal{SL}(3,\mathcal U)$, and $W^+$ gives a conjugacy of 
$A^+$ to $A''$ as required. 

Let $0_i$ denote the  $i\times i$ zero matrix. 
At the inductive step,
 we begin with a conjugacy of a positive matrix $(A^+)_{\ell -1}$ to a matrix 
with block form 
$
\left(\begin{smallmatrix} C_{\ell -1} & 0 \\ 0& 0 
\end{smallmatrix}\right) 
=C_{\ell -1}\oplus 0_{\ell -1}$, and a splitting of $C_{\ell -1}$ to $C_{\ell}$. 
The argument of the basic step produces a 
row splitting    of $(A^+)_{\ell -1}$
to a positive matrix $(A^+)_{\ell} $ over $\mathcal U$ 
and a conjugacy  over $\mathcal U$ of 
$(A^+)_{\ell} $ 
to $(C_{\ell}\oplus 0_{\ell -1})\oplus 0_1$, which equals 
$C_{\ell}\oplus 0_{\ell}$. 

The final claim of the lemma is clear by passing to transpose matrices. 
\end{proof}

\begin{remark} \label{ifunit} 
If the nondiscrete ring $\mathcal U$ is assumed to have a nontrivial 
unit, then in Lemma 
\ref{positivitysplittinglemma}, the matrix $A^+$ can be chosen to have 
size equal to the matrix $C$ (the extra zero blocks can be avoided). 
For this, in the proof at the stage of splitting the row 
$s'+t'$ of $C$, pick $a,b$ from $\mathcal U$ such that 
$a$ closely approximates $1/2$ and $b$ is a sufficiently small unit. 
In place of the matrices $U,X'',M$ in the proof use 
\[
U = \begin{pmatrix} I_{n-1} & 0&0\\ 0&  1 & 1 
\end{pmatrix} \ , \quad 
X''= 
\begin{pmatrix}
* \\ 
s' \\ 
t' 
\end{pmatrix} \ , \quad 
M= \begin{pmatrix}
             a &a-b\\
           1-a  &b+(1-a)
   \end{pmatrix}\  . 
\] 
Then $\det M =b$, so $M$ is invertible over $\mathcal U$,  
and if $a$ is chosen close to $1/2$ and $b$ is sufficiently small, 
the positivity constraints will be satisfied.  
\end{remark}

\begin{remark}[Matrices, module structures and splitting]
Suppose $\mathcal U$ is a  unital ring, and 
let $\mathcal U[t]$ denote the ring of polynomials with coefficients in 
$\mathcal U$. 
If $A$ is an  $n\times n$ matrix $A$ over $\mathcal U$, 
then the free $\mathcal U$ module $\mathcal U^n$ of row vectors 
is a right $\mathcal U[t]$ module $\mathcal M_A$, where the action 
of $t$ is by $v\mapsto vA$. 
Two $n\times n$ matrices over $\mathcal U$ 
are conjugate 
over $\mathcal U$ if and only if their $U[t]$ modules are
isomorphic. 

If a matrix $C'$ is obtained by an elementary row splitting from a 
matrix $C$, with associated subdivision matrix $U$, 
then there is an embedding of 
$\mathcal M_C$ into 
$\mathcal M_{C'}$, given by the rule $v\mapsto vU$. 
The conjugacy given by $W^{+}$ in Lemma \ref{positivitysplittinglemma} 
is constructed to 
extend the conjugacy of embedded copies of $\mathcal M_A$ and 
$\mathcal M_C$ obtained by lifting $W$. 
One finds an $A^+$ and $W^+$ by requiring further conditions 
on the vectors $e_n, e_{n+1}, e_{n+2}$. 

The module viewpoint 
may give 
arguments which are easier and more conceptual, 
or help one find implementing matrices.  
On the other hand, 
it can be useful
to have  matrix arguments  
which can be verified by direct matrix computation.

It is worth noting that with any string of row splittings from a 
matrix $C$ to a matrix $C'$, the module 
$\mathcal M_C$ embeds as a $\mathcal U[t]$ submodule of 
$\mathcal M_{C'}$, and such that as a free $\mathcal U$ module (forgetting 
the $t$ action) $\mathcal M_{C'}$ is the internal direct sum 
of the emebedded copy of $\mathcal M_C$ and another free $\mathcal U$ module. 

\end{remark}

We will also use the following easy  lemma. 

\begin{lemma} \label{laststep} 
Let $\mathcal U$ be a unital nondiscrete subring of $\mathbb R$. 
Suppose $C$ is a positive square matrix over $\mathcal U$ and
 $M$ is a square matrix over $\mathcal U$ of 
the (block) form 
$\left(\begin{smallmatrix} C & 0 \\ X& 0 
\end{smallmatrix}\right)$
or 
$\left(\begin{smallmatrix} C & X \\ 0& 0 
\end{smallmatrix}\right)$. 
Then there is a positive matrix over $\mathcal U$ which is conjugate over $\mathcal U$ 
to $M$ and which is SSE over $\mathcal U_+$ to $C$. 
\end{lemma} 
\begin{proof} 
Clearly it is sufficient to prove the lemma assuming  
$M=\left(\begin{smallmatrix} C & 0 \\ X& 0 
\end{smallmatrix}\right)$. 
Pick $\kappa > 0$ in $\mathcal U$ such that 
$X':=X+\kappa JC$ is positive, where $J$ denotes a 
matrix of appropriate size with every entry equal to $1$, 
and set 
$M'=\left(\begin{smallmatrix} C & 0 \\ X'& 0 
\end{smallmatrix}\right)$. 
Then $M'$ is SSE over $\mathcal U_+$ to $C$, and also conjugate over 
$\mathcal U$ to $M$, since 
\[
M' = 
\begin{pmatrix} C & 0 \\ X' & 0 \end{pmatrix} 
= 
\begin{pmatrix} I & 0 \\ \kappa J & I \end{pmatrix} 
\begin{pmatrix} C & 0 \\ X & 0 \end{pmatrix} 
\begin{pmatrix} I & 0 \\ -\kappa J & I \end{pmatrix} \ . 
\] 
Given $\epsilon $ in $\mathcal U$, define another matrix conjugate 
over $\mathcal U$ to $M$, 
\begin{align*} 
M_{\epsilon} := & 
\begin{pmatrix} I & -\epsilon J \\ 0 & I \end{pmatrix} 
\begin{pmatrix} C & 0 \\  X' & 0 \end{pmatrix} 
\begin{pmatrix} I & \epsilon J  \\ 0 & I \end{pmatrix} \\ 
=& 
\begin{pmatrix} C-\epsilon JX' & 0 \\  X' & 0 \end{pmatrix} 
\begin{pmatrix} I & \epsilon J  \\ 0 & I \end{pmatrix} 
= 
\begin{pmatrix} C-\epsilon JX' & (C-\epsilon JX')\epsilon J
 \\  
X' & X'(\epsilon J) \end{pmatrix} \ . 
\end{align*} 
Fix $\epsilon > 0$ in $\mathcal U$ sufficiently small to 
guarantee 
$C-\epsilon JX'$ is positive. Then 
$M_{\epsilon}$ is a positive matrix SSE over $\mathcal U_+$ to 
$M'$, and hence to $C$. 
\end{proof} 

We are now prepared to prove the main result of this section.

\begin{proof}[Proof of Theorem 
\ref{ssetosim}] 
By assumption, for some $\ell$  we have 
 matrices 
\newline 
$A=A_{(0)},A_{(1)},\dots A_{(\ell )}=B$
and for $0\leq k<\ell $ an ESSE over  $\mathcal U$, 
\begin{equation} \label{ithesse}
A_{(i)}=R_{(i)}S_{(i)} \quad , \quad A_{(i+1)}=S_{(i)}R_{(i)}  \ . 
\end{equation} 
By Proposition \ref{nondegeneratesseprop}, we may assume all the 
matrices $A_{(i)},R_{(i)},S_{(i)}$ are nondegenerate. 
For each $i$, we can then 
by Proposition 
\ref{decompprop} 
 associate to the ESSE (\ref{ithesse}) a diagram of splittings 
and a diagonal refactorization 
\begin{equation} \label{diagram1} 
\xymatrix{ 
       & X_i \ar[rr]^-{(D_i, D_i^{-1}X_i)} 
\ar[dl]   &  & Y_i\ar[dr]  &     \\ 
A_{(i-1)}      &                                          &  
&                   &  A_{(i)} 
}
\end{equation}
 as described in 
(\ref{elsplitdiagram}). By  Lemma \ref{diagrefactoconj}  
we can lift each diagonal refactorization by a row and a column 
splitting to nondegenerate matrices conjugate over $\mathcal U$, giving a diagram 
of three levels, 
\begin{equation}
\xymatrix@=10pt{ 
    &                        & \bullet  \ar[dl] \ar@{=}[r] & \bullet \ar[dr] &                & \\
    & \bullet        \ar[dl] &                             &               & \bullet \ar[dr] &  \\
A_{(i-1)} &                    &                             &               &                 &   A_{(i)}}
\end{equation} 
For visual clarity, we use the horizontal 
 ``=''  in diagrams to indicate conjugacy over $\mathcal U$ (not equality). 

We consider an initial diagram 
(of three levels) formed by taking the union of the $\ell$ diagrams above
 (one for each ESSE). 
Arrows point southwest for row amalgamations and southeast 
for column amalgamations. 
For visual clarity, we suppress matrix names and arrow tips.
 Here is the initial diagram for 
the case $\ell =3$.

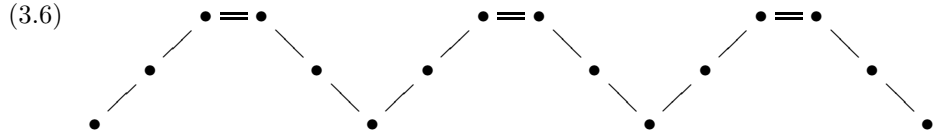
\begin{equation}
\xymatrix@=10pt{
     &                        & \bullet  \ar@{-}[dl] \ar@{=}[r] & \bullet \ar@{-}[dr] &                & 
    &                        & \bullet  \ar@{-}[dl] \ar@{=}[r] & \bullet \ar@{-}[dr] &                & 
    &                        & \bullet  \ar@{-}[dl] \ar@{=}[r] & \bullet \ar@{-}[dr] &                & \\
    & \bullet        \ar@{-}[dl] &                             &               & \bullet \ar@{-}[dr] &  
    & \bullet        \ar@{-}[dl] &                             &               & \bullet \ar@{-}[dr] &  
    & \bullet        \ar@{-}[dl] &                             &               & \bullet \ar@{-}[dr] &  \\
 \bullet &                    &                             &               &                 &   \bullet
         &                    &                             &               &                 &   \bullet 
          &                    &                             &               &                 &   \bullet 
}
\end{equation} 

We apply the Fiber Lemma \ref{fiberlemma} to construct a 
nondegenerate common column/row splitting for each
pair of  matrices in the diagram with a common row/column 
amalgamation, and iterate this move
as far as possible. For our case $\ell =3$, this produces 
the next diagram (with open circles 
and dotted lines reflecting additions to the diagram at this step).

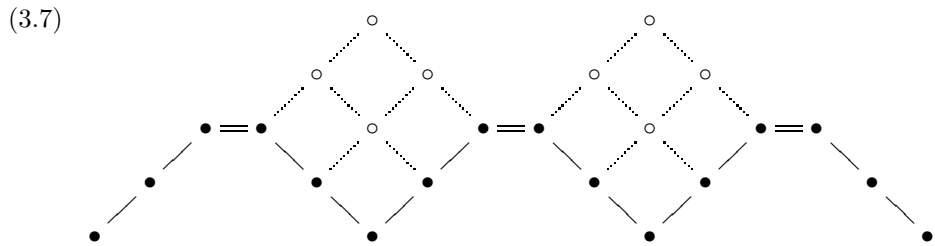
\begin{equation}
\xymatrix@=10pt{
    &                        &                            &                 &                              & 
 \circ \ar@{.}[dl]\ar@{.}[dr] &  &                 &                 &                                 & 
 \circ \ar@{.}[dl]\ar@{.}[dr] & &                 &                 &                & \\
    &                        &                            &                 & \circ \ar@{.}[dl]\ar@{.}[dr] & 
    & \circ \ar@{.}[dl]\ar@{.}[dr]  &                 &                 & \circ \ar@{.}[dl]\ar@{.}[dr] & 
    & \circ \ar@{.}[dl]\ar@{.}[dr]  &                 &                 &                & \\
    &                        & \bullet \ar@{-}[dl] \ar@{=}[r] & \bullet \ar@{-}[dr] &                & \circ \ar@{.}[dl]\ar@{.}[dr]
    &                        & \bullet  \ar@{-}[dl] \ar@{=}[r] & \bullet \ar@{-}[dr] &                &  \circ \ar@{.}[dl]\ar@{.}[dr]
    &                        & \bullet  \ar@{-}[dl] \ar@{=}[r] & \bullet \ar@{-}[dr] &                & \\
    & \bullet        \ar@{-}[dl] &                             &               & \bullet \ar@{-}[dr] &  
    & \bullet        \ar@{-}[dl] &                             &               & \bullet \ar@{-}[dr] &  
    & \bullet        \ar@{-}[dl] &                             &               & \bullet \ar@{-}[dr] &  \\
 \bullet &                    &                             &               &                 &   \bullet
         &                    &                             &               &                 &   \bullet 
          &                    &                             &               &                 &   \bullet 
}
\end{equation} 

Next, we apply part (1) of the Splitting Lemma 
\ref{positivitysplittinglemma}  to lift conjugacies of 
nondegenerate matrices by row or column splittings. 
Where there is a choice, for definiteness (only) we choose to lift by 
row splittings. For $\ell =3$ this produces 
the following diagram. A nonhorizontal arrow here arising from the Splitting 
Lemma represents the composition of several splittings through 
nondegenerate matrices. 
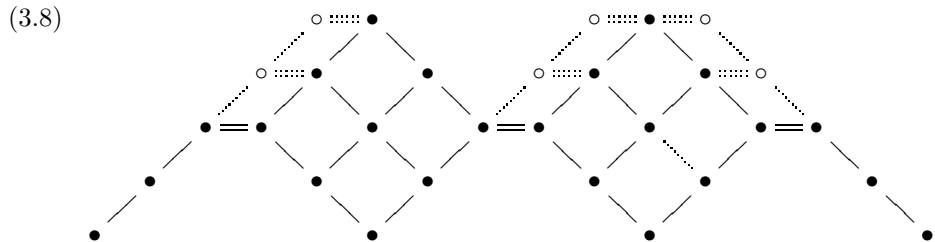
\begin{equation}
\xymatrix@=10pt{
    &                        &                            &                            & \circ \ar@{.}[dl] \ar@{:}[r]        & 
 \bullet \ar@{-}[dl]\ar@{-}[dr] &  &                 &  &\circ \ar@{.}[dl] \ar@{:}[r]                                    & 
 \bullet \ar@{-}[dl]\ar@{-}[dr]\ar@{:}[r] &\circ\ar@{.}[dr] &                 &                 &                & \\
    &                        &                            & \circ \ar@{.}[dl] \ar@{:}[r] & \bullet \ar@{-}[dl]\ar@{-}[dr] & 
    & \bullet \ar@{-}[dl]\ar@{-}[dr]  &                 & \circ \ar@{.}[dl] \ar@{:}[r] & \bullet \ar@{-}[dl]\ar@{-}[dr] & 
    & \bullet \ar@{-}[dl]\ar@{-}[dr] \ar@{:}[r] & \circ\ar@{.}[dr]                &                  &                & \\
    &                        & \bullet \ar@{-}[dl] \ar@{=}[r] & \bullet \ar@{-}[dr] &                & \bullet \ar@{-}[dl]\ar@{-}[dr]
    &                        & \bullet  \ar@{-}[dl] \ar@{=}[r] & \bullet \ar@{-}[dr] &                &  \bullet \ar@{-}[dl]\ar@{.}[dr]
    &                        & \bullet  \ar@{-}[dl] \ar@{=}[r] & \bullet \ar@{-}[dr] &                & \\
    & \bullet        \ar@{-}[dl] &                             &               & \bullet \ar@{-}[dr] &  
    & \bullet        \ar@{-}[dl] &                             &               & \bullet \ar@{-}[dr] &  
    & \bullet        \ar@{-}[dl] &                             &               & \bullet \ar@{-}[dr] &  \\
 \bullet &                    &                             &               &                 &   \bullet
         &                    &                             &               &                 &   \bullet 
          &                    &                             &               &                 &   \bullet 
}
\end{equation} 
 
We iterate the application of Splitting Lemma and Fiber Lemma
 until we arrive at a final diagram of $2\ell +1$ levels
whose top level consists of 
matrices which are all conjugate. For $\ell =3$, this happens at
 the next stage, and produces the diagram (\ref{finaldiagram}) below. 
\begin{equation}\label{finaldiagram}
  \xymatrix@=10pt{
    &                        &                            &                            &         & 
& \circ \ar@{.}[dl]\ar@{:}[r] &    \circ \ar@{.}[dl]\ar@{.}[dr]\ar@{:}[r]   &\circ \ar@{:}[dr]\ar@{:}[r]  &\circ    \ar@{:}[dr]  &   
&&    &                 &                & \\
    &                        &                            &                            &         & 
 \circ \ar@{.}[dl]\ar@{:}[r] &    \circ \ar@{.}[dl]\ar@{.}[dr] & &  \circ \ar@{.}[dl]\ar@{.}[dr]
\ar@{:}[r]   &\circ \ar@{:}[r] \ar@{:}[dr]  & \circ  \ar@{:}[dr] 
 &  &     &                 &                & \\
    &                        &                            &                            & \bullet  \ar@{-}[dl]  \ar@{=}[r]   & 
 \bullet \ar@{-}[dl]\ar@{-}[dr] &  &  \circ \ar@{.}[dl]\ar@{.}[dr]  &  &\bullet    \ar@{-}[dl] \ar@{=}[r]               & 
 \bullet \ar@{-}[dl]\ar@{-}[dr]\ar@{=}[r] &\bullet  \ar@{-}[dr]  &     &                 &                & \\
    &                        &          & \bullet  \ar@{-}[dl] \ar@{=}[r] & \bullet \ar@{-}[dl]\ar@{-}[dr] & 
    & \bullet \ar@{-}[dl]\ar@{-}[dr]  &         & \bullet  \ar@{-}[dl] \ar@{=}[r] & \bullet \ar@{-}[dl]\ar@{-}[dr] & 
    & \bullet \ar@{-}[dl]\ar@{-}[dr]\ar@{=}[r]  & \bullet  \ar@{-}[dr]                &                  &                & \\
    &                        & \bullet \ar@{-}[dl] \ar@{=}[r] & \bullet \ar@{-}[dr] &                & \bullet \ar@{-}[dl]\ar@{-}[dr]
    &                        & \bullet  \ar@{-}[dl]\ar@{=}[r] & \bullet \ar@{-}[dr] &                & \bullet \ar@{-}[dl]\ar@{-}[dr]
    &                        & \bullet  \ar@{-}[dl]\ar@{=}[r] & \bullet \ar@{-}[dr] &                & \\
    & \bullet        \ar@{-}[dl] &                             &               & \bullet \ar@{-}[dr] &  
    & \bullet        \ar@{-}[dl] &                             &               & \bullet \ar@{-}[dr] &  
    & \bullet        \ar@{-}[dl] &                             &               & \bullet \ar@{-}[dr] &  \\
 \bullet &                    &                             &               &                 &   \bullet
         &                    &                             &               &                 &   \bullet 
          &                    &                             &               &                 &   \bullet 
}
\end{equation} 

At the top of the left side of the final diagram is a nondegenerate matrix 
$C$ which is obtained from $A$ by a sequence of row splittings through 
nondegenerate matrices over $\mathcal U$. By Part (2) of the 
Splitting Lemma \ref{positivitysplittinglemma}, 
there is a  matrix $A^+$ SSE over $\mathcal U_+$ to $A$ and 
conjugate over  $\mathcal U$ to a matrix of  the form 
$\left(\begin{smallmatrix} C & 0 \\ 0& 0 
\end{smallmatrix}\right)$. Similarly, if $E$ is the matrix on the right 
side of the top level of the final diagram, then there is a positive 
matrix  $B^+$ which is SSE over $\mathcal U_+$ and conjugate over $\mathcal U$ 
to a matrix of the form 
$\left(\begin{smallmatrix} E & 0 \\ 0& 0 
\end{smallmatrix}\right)$. If $A^+$ and $B^+$ are not of the same size, 
then we may apply Lemma \ref{laststep}  to enlarge one of them, and assume 
they have the same size. Because $C$ and $E$ are conjugate over $\mathcal U$, 
it then follows that $A^+$ and $B^+$ are conjugate over $\mathcal U$. 

  This concludes the 
proof, apart from the ``moreover'' claim for the conjugating 
matrix. This is perhaps already clear from previous work, but 
we will give a self contained proof in the following lemma.  
\end{proof}

\begin{lemma} 
Suppose $\mathcal U$ is a nondiscrete unital subring of $\mathbb R$, 
and $A,B$ are $n\times n$ positive matrices conjugate over 
$\mathcal U$. Then there are positive matrices $A',B'$ SSE over 
$\mathcal U_+$ to $A,B$ respectively and a matrix $U$ invertible over 
$\mathcal U$ such that $U^{-1}A'U=B'$, $\det U > 0$, and 
$U$ sends positive eigenvectors to positive eigenvectors. 
\end{lemma} 

\begin{proof} 
We are given $U\in \textnormal{GL}(n,\mathcal U)$ such that 
$U^{-1}AU=B$. We may assume (if necessary after replacing 
$U$ with $-U$) that $U$ sends positive eigenvectors to 
positive eigenvectors. 
In the case $\det U < 0$, it would suffice to have some $W$ invertible 
over $\mathcal U$ such that 
$AW=WA$, $\det W < 0$ and $W$  
respects positive eigenvectors of $A$. (We could then replace 
$U$ with $WU$.) 

If no such $W$ exists, then for some small $\epsilon >0$ 
we will define an $(n+1)\times (n+1)$ matrix $A'$ 
as a row splitting of $A$, 
by splitting the first row $A_1$ of $A$ as 
$\epsilon A_1 + (1-\epsilon)A_1$. Here $\epsilon >0$, 
$\epsilon \in \mathcal U$ and $\epsilon$ is small enough 
that $ A'>0 $.  Let $E_{ij}(s)$ denote the $n+1\times n+1$ 
matrix equal to $I$ except that the $ij$ entry is $s$. 
Let $F=E_{12}(-\epsilon)E_{21}(1)$. Then 
\[
FA'F^{-1}= \begin{pmatrix} 0 & 0 \\ 0 & A \end{pmatrix} 
:= A'' \ . 
\] 
Because the matrix $K= \left(\begin{smallmatrix} -1 & 0 \\ 0 & I 
\end{smallmatrix}\right) $ 
has negative determinant, commutes with $A''$ and 
fixes eigenvectors for nonzero eigenvalues, the 
matrix $W=F^{-1}KF$ will have the same properties with respect to 
$A'=F^{-1}A''F$. 

Now define $B'$ from $B$ in the same way. The matrices $A',B'$ 
are positive, SSE over $\mathcal U_+$ to $A,B$ respectively, 
and conjugate by a matrix with positive determinant which sends 
positive eigenvectors to positive eigenvectors.  
\end{proof}

We prepare for 
 the last result of this section 
with the next lemma.

\begin{lemma} \label{zzzlemma} 
 Suppose $A,B,D,B'$ are  matrices 
over a field $\mathcal U$ such that $D$ is diagonal nonsingular,
 $A=D^{-1}BD$ and $B'$ is a row splitting of $B$. Then there is a 
row splitting $A'$ of $A$ and a diagonal matrix $D'$ over 
$\mathcal U$ such that 
 $A'=(D')^{-1}BD'$ .

The same statement is true if \lq\lq row\rq\rq is replaced with 
\lq\lq column\rq\rq .
\end{lemma}

\begin{proof}
Clearly it suffices to prove the row statement. Let $U$ be 
the subdivision matrix for the assumed row splitting,
$BU=UB' $ .  Define the diagonal matrix $D'$ in the obvious 
way, $D'(i,i)=D(\overline i, \overline i)$ where $\overline i$
satisfies $U(\overline i,i)=1$. 
Define 
 $A'$ to be 
 $(D')^{-1}BD'$. Then $U$ is also 
the subdivision matrix for a row splitting of 
$A$ to $A' $. For an example of this, take 
\begin{align*} 
A= 
\begin{pmatrix} a & b\frac{\delta_1}{\delta_2 }\\ 
c\frac{\delta_2}{\delta_1 } & d \end{pmatrix} 
\qquad \quad \quad 
B&= 
\begin{pmatrix} a & b \\ c & d \end{pmatrix} 
\quad \quad \qquad  
D= 
\begin{pmatrix} \delta_1 & 0 \\ 0 & \delta_2 \end{pmatrix} 
\\
A'=\begin{pmatrix} a_1 & a_1 & \frac{\delta_1}{\delta_2 }b_1 
\\ a_2 &a_2&\frac{\delta_1}{\delta_2 } b_2 \\ 
\frac{\delta_2}{\delta_1 }c &\frac{\delta_2}{\delta_1 }c & d \end{pmatrix} 
\quad \quad 
B'&= 
\begin{pmatrix} a_1 & a_1 & b_1 \\ a_2 &a_2 &b_2 \\ c &c & d \end{pmatrix} 
 \quad 
U= 
\begin{pmatrix} 1 & 1 &0 \\ 
0 & 0& 1 \end{pmatrix} \ . 
\end{align*}
 \end{proof}
If $\mathcal U$ is not a field, then the definition of $A'$  
in the proof of Lemma \ref{zzzlemma} above might not 
give a matrix over $ \mathcal U$.

We will need the following result from \cite{KR3}. 

\begin{theorem} \label{easystreet}
 Suppose $A,B$ are nondegenerate 
matrices SSE over $\mathcal U_+$, where $\mathcal U$ 
is a subfield of $\mathbb R$. Then there is a
 nondegenerate matrix $C$ over $\mathcal U_+$ and 
a nonsingular diagonal matrix $D$ over $\mathcal U_+$ 
such that $C$ is reached from $A$ by finitely many row splittings 
through  matrices over $\mathcal U_+$ and 
 $D^{-1}CD$ is reached from $B$ by finitely many column splittings 
through  matrices over $\mathcal U_+$.
 If $A$ is primitive, then $C$ is primitive. 
\end{theorem} 

\begin{proof}
The proof is a simplification of part of the proof of 
Theorem \ref{ssetosim}. We begin with a lag $\ell$  SSE from $A$ to $B$ 
through nondegenerate matrices over $\mathcal U_+$. As in 
(\ref{diagram1}), from each  elementary SSE we produce 
 a row splitting, diagonal refactorization and column splitting. 
For example, with $\ell = 4$ we get a diagram 
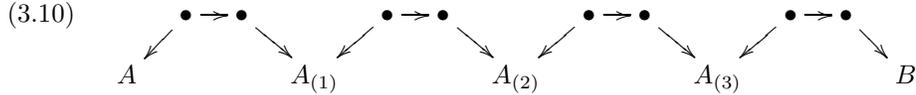
\begin{equation}
\xymatrix@=10pt{
    & \bullet  \ar[r]\ar[dl] & \bullet \ar[dr] &    &\bullet   \ar[r]\ar[dl] & \bullet  \ar[dr] &    & \bullet  \ar[r]\ar[dl]
 & \bullet  \ar[dr] &   & \bullet \ar[r]\ar[dl] & \bullet  \ar[dr] &      \\
A &    &  &A_{(1)}  &    &   &A_{(2)}   &  &  &A_{(3)}  &  &  & B                                         
}
\end{equation} 
A southwest pointing arrow  represents a row amalgamation. 
A southeast pointing arrow  represents a column amalgamation. 
A horizontal arrow represents a diagonal refactorization over $\mathcal U$. 

Above each $A_{(i)}$, 
we apply Lemma \ref{fiberlemma} to produce a common splitting: 
With this move, we have added another level to the top of the diagram: 
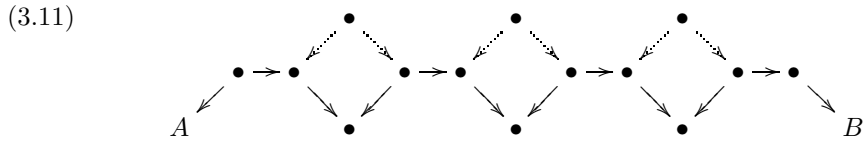
\begin{equation}
\xymatrix@=10pt{
 &&    & \bullet\ar@{.>}[dl]\ar@{.>}[dr] & & &\bullet\ar@{.>}[dl]\ar@{.>}[dr] &  &  &\bullet\ar@{.>}[dl]\ar@{.>}[dr]    &  &  &  \\
    & \bullet  \ar[r]\ar[dl] & \bullet \ar[dr] &    &\bullet   \ar[r]\ar[dl] & \bullet  \ar[dr] &    & \bullet  \ar[r]\ar[dl]
 & \bullet  \ar[dr] &   & \bullet \ar[r]\ar[dl] & \bullet  \ar[dr] &      \\
A &  &  &  \bullet &  &   &  \bullet  &  &  & \bullet  &  &   & B                                         
}
\end{equation}
Then we apply
Lemma \ref{zzzlemma}
to lift each diagonal refactorization on the old 
top row by a splitting to the new top row: 
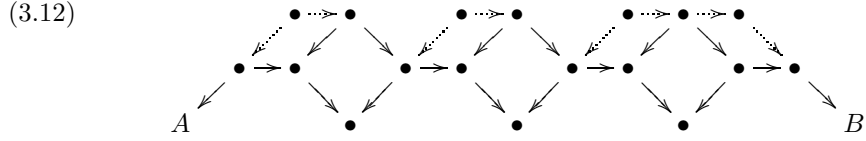
\begin{equation}
\xymatrix@=10pt{
 & &\bullet\ar@{.>}[dl]\ar@{.>}[r] & \bullet\ar[dl]\ar[dr] &  &\bullet\ar@{.>}[dl]\ar@{.>}[r] &\bullet\ar[dl]\ar[dr]
 &  &\bullet\ar@{.>}[dl]\ar@{.>}[r]  &\bullet\ar[dl]\ar[dr]\ar@{.>}[r]    & \bullet \ar@{.>}[dr]&    \\
    & \bullet  \ar[r]\ar[dl] & \bullet \ar[dr] &    &\bullet   \ar[r]\ar[dl] & \bullet  \ar[dr] &    & \bullet  \ar[r]\ar[dl]
 & \bullet  \ar[dr] &   & \bullet \ar[r]\ar[dl] & \bullet  \ar[dr] &      \\
A &  &  &  \bullet &  &   &  \bullet  &  &  & \bullet  &  &   & B                                         
}
\end{equation}
When there is a choice, for definiteness (only) we make the choice to lift with a 
row splitting.  

Iterating this pair of moves  $\ell -1$ additional times, 
we produce a diagram with $\ell +1$ horizontal levels.  
 For $\ell =4$, this is the following diagram, in which
 we insert some matrix 
names (whose generalizations to arbitrary $\ell$
should be clear) and 
for visual 
simplicity  suppress arrowheads and bullets.  
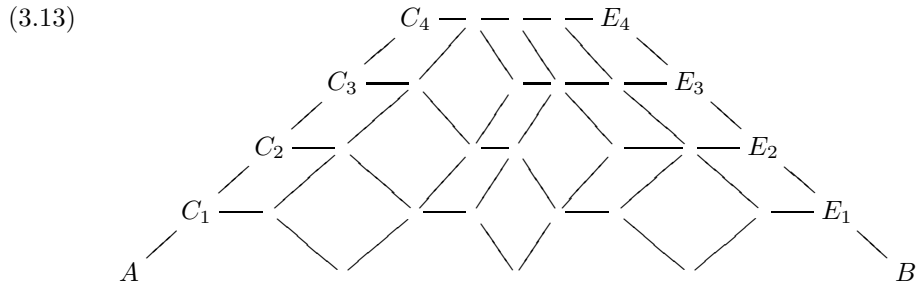
\begin{equation}
\xymatrix@=10pt{
&&&& C_4\ar@{-}[dl]\ar@{-}[r] & \ar@{-}[dl]\ar@{-}[dr] \ar@{-}[r]  &\ar@{-}[r]\ar@{-}[dr]   &  \ar@{-}[dr]\ar@{-}[r]  &  
E_4\ar@{-}[dr] &&  &  &  \\
&&&  C_3\ar@{-}[dl]\ar@{-}[r] & \ar@{-}[dl]\ar@{-}[dr] & & \ar@{-}[dl]\ar@{-}[r]  &\ar@{-}[dl]\ar@{-}[dr] \ar@{-}[r]       &  \ar@{-}[dr]\ar@{-}[r] &
E_3\ar@{-}[dr] &&      \\
&&  C_2\ar@{-}[dl]\ar@{-}[r] & \ar@{-}[dl]\ar@{-}[dr] & & \ar@{-}[dl]\ar@{-}[r]  &\ar@{-}[dl]\ar@{-}[dr]     &  &  \ar@{-}[dl]\ar@{-}[r] 
 &\ar@{-}[dl]\ar@{-}[dr] \ar@{-}[r]    &E_2\ar@{-}[dr]  &  &      \\
    & C_1 \ar@{-}[r]\ar@{-}[dl] &  \ar@{-}[dr] &    &  \ar@{-}[r]\ar@{-}[dl] &  \ar@{-}[dr] &    &  \ar@{-}[r]\ar@{-}[dl] &  \ar@{-}[dr] & 
  &  \ar@{-}[r]\ar@{-}[dl] & E_1 \ar@{-}[dr] &      \\
A &                 &           & &                 &           &  &                 &       
   & &                 &           & B                                         
}
\end{equation}
Define $A'=C_{\ell}$ and $B'=E_{\ell}$.
The diagonal matrix $D$ such that $D^{-1}A'D=B'$ is produced by composing 
the $\ell$ diagonal refactorizations on the top level of the diagram. 
 If $A$ is primitive, then $C$ is primitive, because $C$ is nondegenerate 
and SSE over $\mathbb R_+$ to $A$. 
\end{proof}


\section{The Centralizer} \label{centralizersec}
        
\begin{definition} Given 
an $n\times n$ matrix $A$ over $\mathbb R$,  we let 
$\za$ denote $\{B\in \textnormal{GL}(n,\mathbb R): AB=BA\}$, the centralizer 
of $A$ in $\textnormal{GL}(n,\mathbb R)$. 
 Let $\textnormal{GL}_+(n, \mathbb R)$ denote the connected component of the 
identity in $\textnormal{GL}(n, \mathbb R)$ (the matrices with positive determinant). 
Define
$\zaa $ to be $\za \cap \textnormal{GL}_+(n, \mathbb R)$,
\end{definition} 

This section provides some background and notation 
for $\za$, needed for the results to come on 
strong shift equivalence over nondiscrete unital subrings of $\mathbb R$. 

In the next lemma,  $\mathcal U$ could be for example a field, or it 
could be obtained 
from a real algebraic number ring by inverting all but finitely
many primes.

\begin{lemma}[Centralizer Lemma]\label{centralizerlemma}
 Suppose that $\mathcal U$ is a dense unital subring of 
$\mathbb R$ which contains an ideal $J\neq \mathcal U$ 
such that every  element of $\mathcal U$ outside $J$ is a 
unit in $\mathcal U$.  Suppose that $A $ is an $n\times n$
 matrix over $\mathcal U$. 

Then every connected component of $\textnormal{Cent}_{\mathbb R}(A)$ 
contains an element of $\textnormal{GL}(n,\mathcal U)$. 
\end{lemma} 

\begin{proof} 
Let $\mathbb F$ denote the field of fractions of $\mathcal U$. 
The set of (not necesarily invertible) 
real matrices which commute with $A$, as the solution 
set of $AX-XA=0$, is a real vector space $V_{\mathbb R}$ in 
which the matrices over $\mathbb F$ are an $\mathbb F$ 
vector space $V_{\mathbb F}$ of equal dimension.
Thus  $V_{\mathbb F}$  is dense in 
$V_{\mathbb R}$. 
It suffices to show $\textnormal{GL}(n,\mathcal U) \cap V_{\mathbb R}$ 
is dense in $ V_{\mathbb F}$ (in which case   
it is 
also dense in every connected component of 
$\textnormal{GL}(n,\mathbb R) \cap V_{\mathbb R}$, which is  
$\textnormal{Cent}_{\mathbb R}(A)$).  

If $J=\{0\}$, then $\mathcal U$ is a field. 
Then any element of 
$V_{\mathbb F}$ close to an element of 
$\textnormal{GL}(n,\mathbb R) \cap V_{\mathbb R}$ has  nonzero 
determinant and thus lies in 
$\textnormal{GL}(n,\mathcal U)$. 

So, let $J$ be a proper ideal 
of $\mathcal U$. Then $J$ is dense in $\mathbb R$. 
Suppose  $X\in V_{\mathbb F}$, 
and $\epsilon >0$.  Pick a nonzero $d_1\in \mathcal U$ 
such that $d_1X$ has all entries in $\mathcal U$. 
Pick $d_2\in \mathcal U$ such that $|d_1d_2-1|< \epsilon$. Let 
$Y=d_1d_2X$. Then $Y$ has entries in $\mathcal U$ and 
$||Y-X||< \epsilon ||X||$. 
Pick $c$ in $J$ such that $|1-c|<\epsilon$. 
Let $M= (1-c)I +cY$. Then $M\in V_{\mathbb F}$ and 
\[
||Y-M||
= ||(1-c)(Y-I) || \leq \epsilon ||Y-I||\ . 
\] 
Since  $\det(M)= \det(I+c(Y-I))
 \equiv 1 \mod J$, we also have 
$M\in \textnormal{GL}(n,\mathcal U)$. 
\end{proof} 

For completeness we recall an example from \cite{KR4}.

\begin{example} \label{badcentralizer} 
Let $U=\mathbb Z[1/p]$, where $p$ is a prime 
(e.g., 5)  which does not split in the algebraic number ring 
$Z[\sqrt 3]$.
Let $A$ be a nonzero $2\times 2$ matrix of the form 
$c_1I + c_2B$, with $c_1, c_2$ in $\mathcal U$, where 
$B= \left(\begin{smallmatrix} 0 & 3 \\1 & 0
\end{smallmatrix} \right)$. 

Then $\textnormal{GL}(2,\mathcal U)$ 
does not intersect every connected component of  
$\textnormal{Cent}_{\mathbb R}(A)$.  
If $C$ is the $3\times 3 $ matrix $B\oplus 1$, then  
$\textnormal{GL}(3,\mathcal U)$ 
does not intersect every connected component of  
$\textnormal{Cent}^+_{\mathbb R}(C)$.  
\end{example} 
\begin{proof} 
There is an isomorphism of fields from $\mathbb Q[\sqrt 3]$ to 
$\mathbb Q[A]$ induced by 
$\sqrt 3 \mapsto B$. 
A fundamental unit for the algebraic number 
ring $\mathbb Z[\sqrt 3]$ is  $2+\sqrt 3$, which has positive norm 1. 
If $p$ is an odd prime and $p>3$ and $3$ is not a 
square mod $p$ (for example $p=5$), 
then  $p$ does not split in $\mathbb Z[\sqrt 3]$  
 \cite[p.74]{Ma}.

 The
matrix $A$ has distinct real eigenvalues and can be diagonalized over the reals. When
diagonalized, its centralizer becomes all diagonal matrices and is all linear 
combinations over $\mathcal U$ 
of $I,A$. So that is also true when it is not diagonalized. The real centralizer will be a
direct sum of two copies of the reals and has four components. Some components have negative
determinant. The centralizer of $A$ 
within the $2\times 2$ matrices over $U$ consists of linear combinations of $I,A$
over $U$, and is isomorphic as a ring to the quadratic number ring $R=\mathbb Z[\sqrt 3]$ 
with the prime $p$ inverted. By assumption $p$ is prime in $R$, so any unit of 
$\mathcal U$ has the form $p^m(2+\sqrt 3)^n$ for some integers $m,n$. The norms of 
all units in $\mathcal U$ are positive, which translates to the determinants of 
all elements of $\textnormal{GL}(2,\mathcal U)$ in the centralizer of $A$ being positive. 
Therefore the negative determinant 
components of $\textnormal{Cent}_{\mathbb R}(A)$ will not intersect 
$\textnormal{GL}(2,\mathcal U)$. 

A matrix in the centralizer of  $C$ will act like a matrix in the centralizer 
of $A$ together with multiplication by a  scalar on the 
fixed direction. A negative-determinant component of 
$\textnormal{Cent}_{\mathbb R}(A)$, with multiplication by a negative 
number on the fixed direction, will yield a positive-determinant 
component  of $\textnormal{Cent}_{\mathbb R}(C)$ which does not intersect 
$\textnormal{GL}(3,\mathcal U)$. 
\end{proof}

\begin{definition} \label{gamma}
Given a square real matrix $A$, 
let $\mathcal J(A)$ denote the set of pairs 
$(\lambda, j)$ such that $ \lambda \in \mathbb R$, 
$j\in \mathbb N$ and the Jordan form of $A$ contains a 
$j\times j$ Jordan block for $\lambda$. Then define   
\begin{align*} 
\gamma (A) =& |\mathcal J(A)| 
\ .
\end{align*} 
For $A$ $n\times n$ over $\mathbb R$ 
and $(\lambda ,j)\in \mathcal J(A)$, define 
the vector space 
\[
V(A,\lambda, j) =
\{ x(A-\lambda I)^{j-1} \in \textnormal{ker}(A): 
x\in \mathbb R^n\setminus \textnormal{image}(A)\}\cup \{0\} \ .
\]
A matrix $B$ in $\za$ maps each  
$V(A,\lambda, j)$ to itself. Let $\sigma (\lambda ,j)$ be the sign of the
determinant of this map determined by $B$. 
\end{definition} 

Recall, $\pi_0(X)$ is the set of connected components of a topological 
space $X$. 
\begin{proposition} \label{gammabound}
Let $A$ be an $n\times n$ real matrix. Then with $\gamma $ as defined above, 
\[
|\pi_o(\za )| = 2^{\gamma (A)} \ . 
\] 
Two matrices lie in the same component if and only if 
they have the same sign $\sigma(\lambda ,j)$ for all
$(\lambda ,j)$ in $\mathcal J(A)$.
\end{proposition} 
\begin{proof} 
If $A$ is zero, the claim holds because 
$|\pi_o(\textnormal{GL}(n,\mathbb R )| = 2 \ $. Now suppose $A\neq 0$. 
$A$ is a  sum 
of commuting nonzero real matrices $A_{\lambda}$, where $\lambda $ denotes 
a root of $\chi_A$ with nonnegative imaginary part, and 
$A_{\lambda}-\lambda I$ is nilpotent if $ \lambda \in \mathbb R$,  
and $(A_{\lambda}-\lambda I)(A_{\lambda}-\overline{\lambda }I)$ 
is nilpotent if $\lambda \in \mathbb R \setminus \mathbb C$. The 
centralizer $\za$ is homeomorphic to the product 
$\prod_{\lambda}\textnormal{Cent}_{\mathbb R}(A_{\lambda})$. So it 
suffices to show the claim for each $A_{\lambda}$. 

If $\lambda$ is real, then 
$\textnormal{Cent}_{\mathbb R}(A_{\lambda})=
\textnormal{Cent}_{\mathbb R}(A_{\lambda}-\lambda I)$, 
and we can consider  
$\textnormal{Cent}_{\mathbb R}(M)$ for $M$ nilpotent. 
Let $C_j$ denote a matrix of the form $J\otimes I$, where 
$J$ is a  $j\times j$ Jordan block with zero diagonal. Then 
$\textnormal{Cent}_{\mathbb R}(C_j)$ has a banded form, e.g. 
for $C_3$ a $(3k\times 3k)\times (3k\times 3k)$ matrix,  
\[
C_3= \begin{pmatrix} 0 & I & 0 \\ 0 & 0 & I \\ 0 & 0 & 0 
\end{pmatrix}\ , \qquad 
\textnormal{Cent}_{\mathbb R}(C_3)= \Bigg\{ 
\begin{pmatrix} X & Y & Z \\ 0 & X & Y \\ 0 & 0 & X 
\end{pmatrix}\ : 
X\in \textnormal{GL}(k,\mathbb R)
\Bigg\}\ . 
\] 
Each  $\textnormal{Cent}_{\mathbb R}(C_j)$ has exactly 
two connected components, depending on the sign of 
the determinant of the repeated diagonal block, 
which is $\sigma(\lambda ,j)$. Up to 
conjugacy, the nilpotent matrix $M$ will be block 
diagonal of the form diag($M_1, M_2, \dots , M_j)$), i.e. 
\[ 
M= \begin{pmatrix} 
M_1 & 0   & 0     &  \dots  &  0 \\ 
0   & M_2 & 0     & \dots   & 0 \\ 
0   & 0   & \dots & \dots   & 0 \\ 
0   & 0   & 0     & \dots   & M_j
\end{pmatrix}\ 
\] 
where $M_i=C_{n(i)}$, $n(1)>n(2)> \dots > n(j)$ and 
$j=\gamma (M)$. 
$\textnormal{Cent}_{\mathbb R}(M)$ is a 
subset of the set of block upper triangular 
matrices such that an element of 
$\textnormal{Cent}_{\mathbb R}(M_i)$ 
occupies  the $i$th diagonal 
block. There is a homotopy from 
$\textnormal{Cent}_{\mathbb R}(M)$ to the set of its 
block diagonal matrices, which is homeomorphic to 
$\prod_j\textnormal{Cent}_{\mathbb R}(M_j)$, which 
has $2^j$ connected components. 

If $\lambda$ is complex and $A_{\lambda}$ is $(2k)\times (2k)$,
 then let $M$ be the $k\times k$ complex matrix 
which is the direct sum of the $\lambda$-Jordan blocks 
in the Jordan form of $A_{\lambda}$ (or $A$). Then 
$\textnormal{Cent}_{\mathbb R}(A_{\lambda})$ is homeomorphic to 
$\textnormal{Cent}_{\mathbb C}(M)$, 
the centralizer of $M$ in $\textnormal{GL}(k , \mathbb C)$. 
 The triangular structure 
described earlier applies to $\textnormal{Cent}_{\mathbb C}(M)$.
However, because $\textnormal{GL}(n,\mathbb C)$ is connected for every $n$, 
we have that $\textnormal{Cent}_{\mathbb C}(M)$ is connected. 
\end{proof}

\section{From paths of similar matrices to strong shift equivalence}
\label{frompathssec}

In this section, we will see how to pass from a path of positive conjugate 
matrices to a strong shift equivalence through positive matrices.   
(The problem of finding such a path we consider later.) 
For completeness, we begin with  a proof for the path lifting 
Proposition \ref{pathlift}, for which we make some preparation. 
Below, the particular choice of norm for $\mathbb R^n$ is 
unimportant. The next lemma was proved in \cite{KR3} with a 
citation to \cite{GR}; we include a proof for completeness. 

\begin{lemma} 
\label{closeconjugation}
Suppose $B$ is a matrix over $\mathbb R$ 
and $\epsilon > 0$. 
Then there exists $\delta >0 $ such that 
if  $||B-B'||<\delta$ and 
$B'$ is conjugate over $\mathbb R$ to $B$, 
then there exists $U$ in $\textnormal{GL}(n,\mathbb R)$ such that 
$U^{-1}BU = B'$ and $||U-I||<\epsilon$. 
\end{lemma} 

\begin{proof} 
We begin with a Claim: 
Suppose $\epsilon > 0$ and $M$ is an $n\times n$ matrix 
of rank $r$ over scalar field $\mathbb C$ or $\mathbb R$, 
and $u_1, \dots , u_{n-r}$ is a basis of $\textnormal{ker}(M)$. 
Then there is $\delta> 0$ such that for any $M'$ with 
$||M-M'||<\delta $ and 
$\textnormal{rank}(M')=r$, there is a basis 
$u'_1, \dots , u'_{n-r}$ of $\textnormal{ker}(M')$ with 
$||u_j-u'_j||<\epsilon ||u_j||$, for each $j$. 

{\it Proof of Claim.} Without loss of generality, suppose 
$0<r<n$. To set notation, we use row vectors
for  $\textnormal{ker}M=\{v:vM=0\}$. 
Let $\textnormal{proj}_W$ denote orthogonal projection onto $W$. Let 
$\textnormal{col} M$ denote the vector space generated by column 
vectors of $M$. 

Within the set of $n\times n$ matrices $M$ of rank $r$, 
the map $\textnormal{proj}_{\textnormal{col}M}$ varies continuously with 
$M$. (For $M'$ near a given $M$, the same $r$ linearly independent 
columns can be used to construct an orthonormal basis with the 
Gram-Schmidt algorithm.) So we may suppose $\delta$ is small enough 
that $||\textnormal{proj}_{\textnormal{col}M'}(v)||
\leq  ||\textnormal{proj}_{\textnormal{col}M}(v)|| + \epsilon ||v||$, for all $v$. 
Set $u'_j= \textnormal{proj}_{\textnormal{ker}M'}(u_j)$. Considering 
$(\textnormal{ker}M')^{\perp}=\{v^{tr}:v\in \textnormal{col}M'\}$, we have 
\begin{align*} 
||u_j-u'_j||\  &  =\  ||\textnormal{proj}_{(\textnormal{ker}M')^{\perp}}(u_j) || \ 
= \ ||\textnormal{proj}_{\textnormal{col}M'}(u^{\tr}_j)|| \\
&  \leq \ 
 ||\textnormal{proj}_{\textnormal{col}M}(u^{\tr}_j)||+ 
\epsilon  ||u^{\tr}_j|| \ 
= \  \epsilon  ||u^{\tr}_j|| \ . 
\end{align*} 
This proves the claim. 

Now suppose $\lambda $ is an eigenvalue of $B$ and 
$\mathcal J_{\lambda}= \{u_1, \dots , u_s\}$ is 
a Jordan basis for the restriction of $B-\lambda I$ to 
$\textnormal{ker} (B-\lambda I)^n$. 
($\cup_{\lambda}\mathcal J_{\lambda}$ is a Jordan basis for 
$B$.) 
Define
\[
\mathcal J_{\lambda}(t)= 
\{ u_i \in \mathcal J_{\lambda}: 
u_i(B-\lambda I)^t=0\textnormal{ and }
u_i \notin \textnormal{image}(B-\lambda I)
\}\ . 
\] 
(The number of vectors in
$\mathcal J_{\lambda}(t)$ equals the number of $t\times t$ 
Jordan $\lambda$-blocks  in the Jordan 
form of $B$.) 
For $B'$ close to $B$, 
let $\{u'_1, \dots , u'_s\}$ be the nearby basis of 
$\textnormal{ker} (B'-\lambda I)^n$, given by the Claim. 
With $B'$ close enough to $B$, 
\[
u_i \in \mathcal J_{\lambda}(t) 
\implies 
u'_i(B'-\lambda I)^{t-1}\neq 0 \ . 
\] 
Consider the $t$ in decreasing order, we then deduce 
from the conjugacy of $B$ and $B'$ that also 
\[
u_i \in \mathcal J_{\lambda}(t) 
\implies 
u'_i(B'-\lambda I)^{t}= 0 \ . 
\] 
For $\lambda $ real, those vectors $u_i, u'_i$ can be chosen 
to be real, and the map on 
$\mathcal J_{\lambda}$ defined by 
\[
u_i(B-\lambda I)^j \mapsto u'_i(B'-\lambda I)^j\ , \quad 
\textnormal{if } u_i\in 
\mathcal J_{\lambda}(t) \textnormal{ and } 0\leq j < t \ , 
\] determines a map $\ker(B-\lambda I)^n\to 
\ker(B'-\lambda I)^n$ which conjugates the restrictions of $B$ 
and $B'$ to these invariant subspaces. 
For $\lambda $ not real, say with positive imaginary part, 
pull back the complex conjugacy to define a map from a real 
Jordan form basis for $B$ for $\ker(B-\lambda I)^n(B-\bar\lambda I)^n$ 
to a corresponding nearby real Jordan basis for $B'$ 
for $\ker(B'-\lambda I)^n(B'-\bar\lambda I)^n$. 

The matrix $U $ in $\textnormal{GL}(n,\mathbb R)$ 
implementing these maps on invariant subspaces  
induces a conjugacy of $B$ and $B'$ and is close to the identity. 
\end{proof}

Below, we suppose $A$ is an $n\times n$ real matrix, 
$\textnormal{Cent}(A)= \{U\in \textnormal{GL}(n,\mathbb R): UA = AU\} $; 
$\textnormal{Conj}(A)= \{U^{-1}AU : U\in \textnormal{GL}(n,\mathbb R)\}$; 
$\gamma : U \mapsto  U^{-1}AU$; the topology of 
$\textnormal{Conj}(A)$ is by the metric induced by a matrix 
norm, and the image of $\pi$ has the quotient topology. 

\begin{prop} \label{fiberprop} 
The map $\phi$ which makes the following diagram commute 
\[
\xymatrix@=
10pt{
\textnormal{GL}(n,\mathbb R)
\ar^{\pi}[rd]  
\ar^{\gamma}[dd] 
&  \\ 
& \textnormal{GL}(n,\mathbb R)/\textnormal{Cent}(A) \ar_{\phi}[dl]   \\ 
\textnormal{Conj}(A) & 
}
\]
is a homeomorphism. 
\end{prop} 

\begin{proof} 
For $U,V$ in 
$\textnormal{GL}(n,\mathbb R)$, we have 
$\gamma(U)=\gamma(V)$ if and only if $VU^{-1}\in 
\textnormal{Cent}(A)$. So, $\phi$ is a well defined bijection. 
The map $\phi$ is continuous because $\gamma$ is continuous, 
$\pi$ is open and 
$\textnormal{GL}(n,\mathbb R)/\textnormal{Cent}(A)$ has the 
quotient topology. 

It remains to show that $\phi$ is an open map. This holds if 
$\gamma$ is an open map. Suppose $\mathcal V$ is 
an open 
subset of 
$ \textnormal{GL}(n,\mathbb R)$
and  $V \in \mathcal V$. Choose $\epsilon >0$ such that 
$ \mathcal V$ contains the open set 
$\{VU \in \textnormal{GL}(n,\mathbb R): ||U-I||< \epsilon \}$. 
Then $\gamma (\mathcal V)$ contains 
$\{ U^{-1}CU: ||U-I||< \epsilon \} $, which by   
Lemma \ref{closeconjugation} contains 
some neighborhood of 
$C$ in $\textnormal{Conj}(C)
=\textnormal{Conj}(A) $.   
  This shows the map $\gamma$ 
is open and finishes the proof. 
\end{proof} 

Above, 
$((\textnormal{GL}(n,\mathbb R), \pi ,  
\textnormal{GL}(n,\mathbb R)/\textnormal{Cent}(A)) $ 
is a principal bundle
\cite[Ch. 4.2]{Hus}. 
 The projection $\pi $ is locally trivial: 
for every $x$ in 
$\textnormal{GL}(n,\mathbb R)$, there is a neighborhood  $\mathcal U$ of $x$ 
and a neighborhood $\mathcal V$ of $\pi x$ and a homeomorphism $h: 
\mathcal U \to \mathcal V \times \textnormal{Cent}(A) $ such that 
on  $\mathcal U$, $\pi $ is equal to $h$ followed by 
 projection onto $\mathcal V$, $(v,c)\mapsto v$ .

\begin{prop} \label{pathlift} 
Suppose $(A_t)_{0\leq t\leq 1}$ is a path 
of conjugate $n\times n$ real matrices, $U\in 
\textnormal{GL}(n,\mathbb R)$ and $U^{-1}A_0U=A_0$. Then there 
is a path $(G_t)$ in $\textnormal{GL}(n,\mathbb R)$ such that 
$G_0=U$ and $G_t^{-1}A_0G_t=A_t$ for $0\leq t \leq 1$. 
\end{prop} 
\begin{proof} 
The map $\gamma $ has the topological properties of 
the principal bundle projection $\pi$ in  
Proposition \ref{fiberprop}.
The proposition translates to $\gamma$ the path lifting property 
which  $\pi$ enjoys on account of its local triviality as 
a projection. 
\end{proof} 

Let $\mathcal H(I, \zaa)$ denote the homotopy classes 
of paths in $\textnormal{GL}_+(n,\mathbb R)$ from 
the identity to $\zaa$ (by homotopy through paths with initial 
point $I$ and terminal points in $\zaa$).  
In a topological space $X$, $\pi_0(X)$ denotes the set of 
connected components and $\pi_1(x,X)$ denotes the fundamental group 
at basepoint $x$ in $X$. 

\begin{prop} 
Suppose $(A_t)_{0\leq t \leq 1}$ is a 
loop from $A$ to $A$ in $\textnormal{Conj}(A)$. 
Then there is a path $(G_t)$ in $\textnormal{SL}(n, \mathbb R)$ 
such that $G_0=I$ and  
$A_t=(G_t)^{-1}AG_t$, $0\leq t \leq 1$.  Moreover, the homotopy class of the 
loop $(A_t)$  determines both the element of 
$\mathcal H(I, \zaa)$ containing $(G_t)$ and the connected component of 
$\zaa$ which contains $G_1$. The induced maps 
\begin{align*} 
\pi_1(A,\textnormal{Conj}(A)) &\to 
\mathcal H(I, \zaa)\ \ ,  \\ 
\pi_1(A,\textnormal{Conj}(A)) &\to 
\pi_0(\zaa)
\end{align*}  
are bijections. 
\end{prop}
\begin{proof}
Proposition \ref{pathlift} 
explains the existence of the lift 
of $(A_t)$ to a path 
$(G_t)$ in $\textnormal{GL}(n,\mathbb R)$ beginning at $G_0=I$. 
By continuity, each $G_t$ has positive determinant,  
and we may replace $G_t$ with $(\det G_t)^{-1/n} G_t$ to put 
the conjugating path into $\textnormal{SL}(n,\mathbb R)$. It is straightforward 
to check the remaining claims about well defined induced bijections. 
\end{proof} 

\begin{definition} \label{centralizercondition} 
Suppose $\mathcal U$ is a nondiscrete unital subring of the reals. 
Given $A$  an $n\times n$ matrix over $\mathcal U$. We say 
the {\it  centralizer condition} holds for $(A, \mathcal U)$ if
every connected component of $\zaa $ 
has nonempty intersection with 
$\textnormal{GL}(n,\mathcal U)$.  We say that 
{\it 
$\mathcal U$ 
satisfies the 
centralizer condition} if the centralizer condition holds for 
$(A, \mathcal U)$ for every 
square matrix $A$ over $\mathcal U$. 
\end{definition} 
One equivalent statement of the 
 Centralizer Condition is that 
$\zaa $ is generated by 
$\mathcal U \cap \zaa  $ 
 and the connected component of the identity in 
$\zaa $. 
 
\begin{lemma} \label{factorize} 
Suppose $A$ is a positive $n\times n$ real matrix. 
There is an $\epsilon >0$ 
such that for $U\in \textnormal{SL}(n,\mathbb R)$ with $||U-I||< \epsilon $, 
$U$ can be written as a product of $m=(n+4)(n-1)$ basic elementary matrices,  
$U=E_1 \cdots E_{m}$, where each $E_k$ depends continuously on $U$. 

\end{lemma} 

\begin{proof} 
The given $U$ can be made upper triangular by 
$(n-1) + (n-2) + \dots +1=n(n-1)/2$ operations of 
adding multiples of rows successively to lower rows. For $U$ close to the identity, 
at each stage the diagonal terms will remain positive, and the multiples 
of row $i$ added to lower rows to zero out entries in column $i$ will depend 
continuously on $U$. Likewise, 
the same number of additions of 
lower rows to upper rows will diagonalize $U$.
  Finally, for $a\neq 0$, 
\[ 
\begin{pmatrix} 
a & 0 \\ 0 & 1/a 
\end{pmatrix} 
= 
\begin{pmatrix} 
1 & a(a-1)\\ 0 & 1 
\end{pmatrix} 
\begin{pmatrix} 
1 & 0 \\ 1/a & 1 
\end{pmatrix} 
\begin{pmatrix} 
1 & 1-a \\ 0 & 1
\end{pmatrix} 
\begin{pmatrix} 
1 & 0 \\ -1 & 1
\end{pmatrix} \ . 
\] 
So, we can  multiply a diagonal 
determinant 1 matrix by  
$4(n-1)$ elementary matrices 
to produce the identity. In total we have factored $U$ 
as a product of  $m=2[n(n-1)/2] + 4(n-1) = (n+4)(n-1) $ 
basic elementary matrices. We may fix the order of operations.  
 Then in each  
$E_k$, there is a single offdiagonal element which is allowed 
to be nonzero (or zero), and it varies continuously as a 
function of $U$. 
\end{proof} 

\begin{lemma} \label{positiveclose}  
Given  $0<\epsilon< \kappa $ and $n\in \mathbb N$, 
there is a $\delta>0$ such that 
for every $n\times n$ real matrix $M$ with all entries 
bounded below by $\epsilon $ and above by $\kappa$, 
for every $U\in \textnormal{SL}(n,\mathbb R)$ with $||U-I||<\delta$, 
the matrix $U^{-1}MU$ is positive and SSE over $\mathbb R_+$ 
to $M$. 
\end{lemma} 
\begin{proof} 
Pick $\delta >0$ 
small enough that for every $n\times n$ 
positive matrix $M$ with entries 
bounded below by $\epsilon $ and above by $K$, 
with $m=(n+4)(n-1)$ we have 
\begin{itemize} 
\item 
for $U\in \textnormal{SL}(n,\mathbb R)$ with $||U-I||< \delta$, there is a 
continuous factorization $U=E_1 \dots E_m$ of $U$ into 
basic elementary matrices, as in Lemma \ref{factorize}, and  
\item 
with $E_0=I$ and $V_i = E_0\cdots E_i$, for $0\leq i <m$ 
all of the matrices $V_i^{-1}MV_{i}$, $V_i^{-1}MV_{i}E_{i+1}$ and 
$E_{i+1}^{-1}V_i^{-1}MV_{i}$ are positive.  
\end{itemize} 
Given such a matrix $M$, 
set $B_i = V_i^{-1}MV_i$; then $B_0 = M$ and $B_m=U^{-1}MU$. 
For $1\leq i \leq m$,  $B_i$ is positive, and 
one 
of the pairs $(E_i, E_i^{-1}B_i)$, $(E_iB_i, E_i^{-1})$ will give an ESSE 
over $\mathbb R_+$ from $B_{i-1}$ to $B_{i}$. Thus $M$ and $U^{-1}MU$ are 
positive matrices which are SSE over $\mathbb R_+$. 
\end{proof} 

\begin{lemma} \label{diagonal} 
Suppose $\ri$ is 
a nondiscrete unital subring 
 of $\mathbb R$,
$B$ is an $n\times n$ 
 positive matrix over $\ri$ and   $d>0$.

Then there is $\epsilon > 0$ such that 
 $V\in \textnormal{GL}(n,\mathcal U)$ 
with 
$||V-dI||<\epsilon $ implies the matrix 
$V^{-1}BV$ is  positive and is  SSE over $\ri_+$ 
to $B$. 
\end{lemma} 
\begin{proof} 
If $||V-dI||<\epsilon $ 
with $\epsilon $  sufficiently small, then we may add 
small positive multiples of a row $i$ of $V$ to other rows to 
make all off diagonal entries of column $i$ positive and still small. 
Iterating, 
we may find  nonnegative elementary matrices $E_1, \dots , 
E_k$ over $\ri$, with $k\leq n(n-1)$, such that 
with $E=E_kE_{k-1}\cdots E_1$, the matrix 
$EV$ is a positive matrix in $\textnormal{GL}(n,\ri)$. 

Set $E_0=I$ and $B_0=B$. 
For $1\leq i \leq k$ set $B_i=E_iB_{i-1}E_i^{-1}$. 
For $\epsilon$ small, we may choose the matrices 
$E_i$ close enough to $I$ that all of the following matrices 
are also positive: 
$BE^{-1}$; $V^{-1}(BE^{-1})$; $B_{i-1}E^{-1}$ and 
$B_i=E_iB_{i-1}E_i^{-1}$, for $1\leq i\leq k$.

Then $V^{-1}BV= (V^{-1}BE^{-1})( EV) > 0$, and  
the pair $(V^{-1}BE^{-1}, EV)$ gives 
an  ESSE over $\ri_+$ from 
$V^{-1}BV$  to  $EBE^{-1}$. There is also 
an SSE over $\ri_+$ between $B=B_0$ and 
$EBE^{-1}=B_k$: 
for $1\leq i\leq k$, 
the pair $(B_{i-1}E_i^{-1},E_i)$ gives  
an  ESSE over $\ri_+$ from $B_{i-1}$ to $B_i$. 

\end{proof}

 \begin{definition} \label{throughpositivedefinition}
Let $\ri$ be a semiring in $\mathbb R$. 
Matrices $A,B$ are SSE over $\ri_+$,  
through positive $n\times n$ matrices,  if for 
some $\ell\in \mathbb N$ there are $n\times n$ 
positive matrices   
$A=A_0, A_1, \dots , A_{\ell}=B$ such that 
$A_{i-1} $ is ESSE over $\ri_+$ to $A_i$, $1\leq i \leq \ell$.  
\end{definition} 

In the next theorem, part (1)  was proved in \cite{KR3}. 
Parts (2) and (3) were 
 proved in \cite{KR4} 
under the condition 
that elements of $\textnormal{GL}(n,\mathcal U)$ are dense in $\textnormal{GL}(n,\mathbb R)$.  
Here we remove this 
condition by working with the special linear group and scalar matrices.

\begin{theorem}[Path Theorem] \label{paththeorem} 
Let $\mathcal U$ be a unital nondiscrete subring of $\mathbb R$, and  
suppose $(A_t)_{0\leq t\leq 1} $ is a path of positive, real 
$n\times n$ matrices, 
all in the same conjugacy class over $\mathbb R$,  
from $A=A_0$ to $ B=A_1$. Then the following hold 
\begin{enumerate} 
\item 
$A$ and $B$ are SSE over $\mathbb R_+$, through positive 
$n\times n$ matrices.  
\item 
Suppose there is a path $(G_t)$ in $\textnormal{GL}(n,\mathbb R)$ 
such that $G_0=I$ and 
$G_t^{-1}AG_t=A_t$ for all $t$, and 
   there is a $W$ in $\textnormal{GL}(n, \mathcal U)$ such that 
$W^{-1}AW=B$ and $WG_1^{-1}$ is in the connected component 
of the identity in $\textnormal{Cent}_{\mathbb R}(A)$.  \\ 
Then $A$ and $B$ are SSE over  $\mathcal U_+$, 
 through positive $n\times n$ matrices.  
\item 
Suppose $A$ and $B$ are conjugate matrices over $\mathcal U$ such that every 
connected component of 
$\textnormal{Cent}_{\mathbb R} (A) $ contains a matrix from 
$\textnormal{GL}(n, \mathcal U)$.\\ 
Then 
 $A$ and $B$ are SSE over $\mathcal U_+$,  through positive 
 $n\times n$  matrices.  
\item
Suppose $A$ and $B$ are conjugate matrices over $\mathcal U$ and 
$\mathcal U$ is a field, or more generally 
contains an ideal $J\neq \mathcal U$ 
such that every  element of $\mathcal U$ outside $J$ is a 
unit in $\mathcal U$. \\
Then 
 $A$ and $B$ are SSE over $\mathcal U_+$,  through positive 
 $n\times n$  matrices.  
\end{enumerate} 
\end{theorem} 

\begin{proof}
(1) Proposition \ref{pathlift} gives us a 
 path $(G_t)$ in $\textnormal{SL}(n,\mathbb R)$ with $G_0=I$ and 
$(G_t^{-1}AG_t)=(A_t)$. 
The entries of the positive matrices 
$A_t$ are by compactness of the path uniformly bounded below  by some 
positive  $\epsilon $ and above by some $\kappa$. 
Let $\delta > 0$ be chosen as in 
Lemma \ref{positiveclose} for $\epsilon, \kappa$. 
Now by uniform continuity pick $k\in \mathbb N$ such that 
$ s\leq t\leq s+\frac 1k$ implies 
$||G_t(G_s^{-1})-I||<\delta$. 
By Lemma \ref{positiveclose}, $A_{i/k}$ is SSE over $\mathbb R_+$ to 
$A_{(i+1)/k}$, for $0\leq i < k$. Therefore $A=A_0$ and $B=A_1$ are SSE 
over $\mathbb R_+$. 

(2) 
Let $X_t$, $0\leq t \leq 1$, be a path from the identity to $WG_1^{-1}$ 
in the centralizer of $A$ in 
$\textnormal{GL}(n, \mathbb R)$. Then $(X_tG_t)$ is a path from the identity 
to $W$ in $\textnormal{GL}(n, \mathbb R)$, and $W^{-1}AW=B$. 
For all $t$, $(X_tG_t)^{-1}A(X_tG_t)=G_t^{-1}AG_t $. 
So we may assume $G_1 
\in \textnormal{GL}(n, \mathcal U)$. 

Next define the path $H_t$ in $\textnormal{SL}(n,\mathbb R)$ by 
$H_t= c_tG_t$, $0\leq t\leq 1$, where $c_t=(\det (G_t))^{-1/n}$. 
We have  $H_t^{-1}AH_t=G_t^{-1}AG_t$, $0\leq t \leq 1$.  
As in (1), from 
Lemma \ref{positiveclose} we get an SSE through  
 positive real matrices from $A$ to $B$. We denote these  
matrices in order as   $A=B_0, B_1, \dots  , B_{\ell} =A_1=B$. 
For $1\leq i \leq \ell$, we have from 
Lemma \ref{positiveclose} a conjugacy 
$B_i = E_i^{-1}B_{i-1}E_i$, with $E_i$ close enough to $I$ 
that $ E_i^{-1}B_{i-1}$ and 
$B_{i-1}E_i$ are positive, which guarantees that there is 
an ESSE from $B_{i-1} $ to $B_i$. 
We also have $E_1E_2\cdots E_{\ell}=H_1$. 

Now for $1\leq i \leq \ell$ we will  choose a 
basic elementary matrix $E'_i$ over $\ri$, set 
$B'_0=A$, and recursively 
define $B'_i = 
E_i^{-1}B_{i-1}E_i$, $1\leq i \leq \ell$. 
Define $H'_1 = E'_1E'_2\cdots E'_{\ell}$ and 
$A'_1= B'_{\ell}$.  
Set $V=(H'_1)^{-1}G_1\in\textnormal{GL}(n,\mathcal U)$. 
Then $V^{-1}A'_1V=B$. Let $d=[\textnormal{det}(G_1)]^{1/n}$. 
Then $H_1=(1/d)G_1$ and 
\begin{align*} 
V\ =\ & \Big((H'_1)^{-1}-(H_1)^{-1}\Big)G_1 + (H_1)^{-1}G_1 \\ 
=\ & \Big((H'_1)^{-1}-(H_1)^{-1}\Big)G_1 +dI \ . 
\end{align*} 

Now choose $\epsilon >0$ for $B$ as in the statement of 
Lemma \ref{diagonal}. 
We 
 choose the $E'_i$ sufficiently close to the $E_i$  
to guarantee 
\begin{itemize} 
\item 
$B'_i$ is positive and ESSE over $\mathcal U_+$ 
to $B'_{i-1}$, for $1\leq i \leq \ell$ , and  
\item 
$||V-dI||<\epsilon$ .  
\end{itemize}
We have $A'_1$ SSE over $\mathcal U_+$ through 
positive matrices to $A$. It remains now to 
show  
$A'_1$ is SSE over $\mathcal U_+$ through 
positive matrices to $B$. This now follows 
from Lemma \ref{diagonal}.

(3) Again find a path $G_t$ in $\textnormal{GL}(n, \mathbb R)$ such that 
$A_t=G_t^{-1}AG_t$. By assumption, there is a $Y$ in $\textnormal{GL}(n, \mathcal U)$ 
such that $Y^{-1}AY=B$. Therefore $YG_1^{-1}\in 
\textnormal{Cent}_{\mathbb R} (A) $. By assumption 
 there is a matrix $Q$ 
in $\textnormal{GL}(n, \mathcal U)$ such that $Q$ and $YG_1^{-1}$ 
are in the same 
connected component of $ \textnormal{Cent}_{\mathbb R} (A) $. 
Let $W=Q^{-1}Y\in \textnormal{GL}(n,\mathcal U)$. 
Then $WG_1^{-1}$ is in the connected component 
of the identity in $ \textnormal{Cent}_{\mathbb R} (A) $, 
and $W^{-1}AW=B$. Therefore (3) follows from (2). 

(4) 
This claim follows from (3) and 
Lemma \ref{centralizerlemma}.
\end{proof}

\begin{remark}\label{finitelymanycomponents}
 For a given  $n\times n$ matrix 
$A$, the set of all positive matrices conjugate to $A$ 
has only finitely many connected components. This is 
an observation of Sompong  Chuysurichay 
\cite[Theorem 1.4.2]{Ch},  
made in the language of 
invariant tetrahedra (discussed in the appendix \ref{tetrahedrasec}).
It holds because  the set of matrices conjugate to a given matrix 
can be defined by finitely many inequalities in finitely 
many variables,
and  a semialgebraic set has only finitely many connected 
components \cite[Theorem 2.4.4]{BCR98}.
Chuysurichay
\cite[Introduction]{Ch} 
 pointed out the following corollary of this 
fact and Theorem \ref{paththeorem}(1) (which was proved in 
\cite{KR2}). We record this fact as the following theorem.  
\end{remark} 

\begin{theorem}[\cite{Ch,KR2}]\label{finitesse}  Suppose $A$ is a positive $n\times n$ 
matrix. The collection of positive $n\times n$ matrices 
conjugate over $\mathbb R$ to $A$ contains  only finitely many 
SSE-$\mathbb R_+$ classes. 
\end{theorem} 

\begin{remark} \label{unboundedlag} 
Note, the set of matrices  of a given size which are 
 SSE-$\mathbb R_+$  to a given matrix is not a priori 
semialgebraic when the lag is unbounded. Indeed, 
in contrast to Corollary \ref{finitesse}, 
Chuysurichay 
gave an example
\cite[Theorem 1.9.1]{Ch} 
 of a connected 
component $\mathcal C$ in a conjugacy class of  
 positive $2\times 2$ real matrices 
such that the lag of the SSE over $\mathbb R_+$, guaranteed to 
exist between any two  matrices in $\mathcal C$  
by Theorem \ref{paththeorem}(1) above
 (which was proved in \cite{KR2}),  cannot be uniformly bounded 
in $\mathcal C$. (The unboundedness of the lag 
arises for a component of positive conjugate matrices 
whenever there is a matrix on its boundary with more 
than one irreducible component.) 

The fact that the components method produces the finiteness 
result (\ref{finitesse}) despite the possibility of unbounded 
lag is an indication of the power of the method. 
\end{remark}

\begin{remark}\label{z1overpremark}
There are examples \cite[Appendix E]{B} 
of primes $p$ in $\mathbb Z$ and 
primitive matrices over 
$\mathcal U= \mathbb Z[1/p]$ which are SE over 
$\mathcal U_+$ (and hence SSE over $\mathcal U$, since 
$\mathcal U$ is a principal ideal domain) 
but are not SSE over $\mathcal U_+$. 
(We do not know whether these examples are SSE over $\mathbb Q_+$ 
or $\mathbb R_+$.) 
 There are no positive 
matrix 
examples known, for any nondiscrete unital subring $\mathcal U$ 
of $\mathbb R$, of matrices which are SSE over $\mathcal 
 U$ but not 
SSE over $\mathcal U_+$. 
The examples \cite[Appendix E]{B}, 
based on the work over $\mathbb Z$ in 
\cite{KR6}, are matrices with zero trace, and the 
general method relies in a fundamental way on 
 the existence 
of certain matrix powers having zero trace.

Unfortunately, if 
 $p$ is a prime in $\mathbb Z$, then 
the ring $\mathbb Z[1/p]$ does not satisfy 
the ideal hypothesis
of Theorem \ref{paththeorem}(4), and the 
Centralizer Condition \ref{centralizercondition} is not satisfied by 
$\mathbb Z[1/p]$  (Example \ref{badcentralizer}). 
Therefore Theorem \ref{paththeorem} does not 
rule out the possibility that for some $p$ there are positive 
matrices SSE over 
$\mathbb Z[1/p]$ 
which are not SSE over 
$(\mathbb Z[1/p])_+$, even in the case the matrices are connected by 
a path of positive conjugate matrices. 
\end{remark} 

The rest of this section is devoted to generalizing Theorem 
\ref{finitesse} to arbitrary dense subrings of $\mathbb R$. 
To prepare, we need more definitions. 
Let $\mathcal U$ be a dense subring of $\mathbb R$. Suppose $A$ and $B$ 
are matrices over $\mathcal U$;  
$W \in  
\textnormal{GL}(n,\mathcal U)$; and 
$W^{-1}AW=B$. 
Given a path $\mathcal P=(A_t)_{0\leq t \leq 1}$ 
of positive conjugate matrices from $A$ to $B$, 
let $G$ be a matrix such that there is a path 
$(G_{t})$ in $\textnormal{GL}(n,R)$  such that 
$G_{t}^{-1}AG_{t}=A_t$, $0\leq t \leq 1$, 
with 
$G_{0}=I$ 
 and $G_{1}=G$. 
Let 
$\pi_0^{\mathcal U}(\za)$ denote the subgroup of 
$\pi_0(\za)$ consisting of those connected components which contain 
a matrix with all entries in $\mathcal U$.
Define 
$\pi_0(\mathcal P,W)$ to be the connected component of 
 $\pi_0(\za)$ containing $WG^{-1}$. This component is uniquely 
determined by $\mathcal P$ and $W$. Finally,  
let $\overline{\pi_{0,\mathcal U}}(\mathcal P,W)$ be the coset of 
$\pi_0^{\mathcal U}(\za)$
 in 
$\pi_0(\za)$ which contains $WG^{-1}$. 
(We remark as an aside that the coset space 
$\pi_0(\za)/
\pi_0^{\mathcal U}(\za)$ 
is a group, because the group  
$\pi_0(\za)$ is abelian, because all its elements 
have order two.)

\begin{lemma} 
Let $\mathcal U$ be a dense subring of $\mathbb R$. 
Suppose $A,A_1,A_2$ are positive matrices over 
$\mathcal U$ and for $i=1,2$ that 
\begin{itemize}
\item
  $W_i$  is a matrix in $  
\textnormal{GL}(n,\mathcal U)$ such that  
$(W_i)^{-1}AW_i=A_i$ 
\item 
$\mathcal P_i$ is a path of positive conjugate 
matrices from $A$ to $A_i$. 
\end{itemize}
Suppose 
$\overline{\pi_{0,\mathcal U}}(\mathcal P_1,W_1)=
\overline{\pi_{0,\mathcal U}}(\mathcal P_2,W_2)$. 

Then $A_1$ and $A_2$ are SSE over $\mathcal U_+$. 
\end{lemma}

\begin{proof} 
For each path $\mathcal P_i$, 
let $G_i$ be as  in the preceding definitions, 
with $W_iG_i^{-1}$ in the coset 
$\pi_0(\mathcal P_i,W)$.
We get a path 
$\mathcal P$ of positive conjugate matrices from 
$A_1$ to $A_2$ by composing the reversal of $\mathcal P_1$ 
with  $\mathcal P_2$. Let $W=W_1^{-1}W_2$; then 
$W^{-1}A_1W=A_2$.
For the path 
$\mathcal P$,  we  have  $G=G_1^{-1}G_2$. 
We compute 
\begin{align*} 
WG^{-1} &= W_1^{-1}W_2G_2^{-1}G_1 
 = \Big( W_1^{-1}
  (G_1W_1^{-1}) 
  W_1\Big)  
\Big(
W_1^{-1}(W_2G_2^{-1}) 
  W_1\Big)  
\\
&
= W_1^{-1}C  
  W_1
\end{align*}
where $C=
  (W_1G_1^{-1})^{-1} 
(W_2G_2^{-1})$. 
There is a matrix $V$ over $\mathcal U$ 
which lies in the connected component of 
$\za$ containing $C$, and therefore 
the connected component of 
$ \textnormal{Cent}_{\mathbb R} (A_{1}) $
containing  
$WG^{-1}$ contains the matrix 
$V'=W_1^{-1}C  
  W_1$ from 
$\textnormal{GL}(n,\mathcal U)$. 
Now $((V')^{-1}W)G$ is in the connected component of the 
identity 
in $ \textnormal{Cent}_{\mathbb R} (A_{1}) $
and 
$(V')^{-1}W\in \textnormal{GL}(n,\mathcal U)$. 
It follows from 
the Path Theorem \ref{paththeorem}(2) that $A_1$ and $A_2$ 
are SSE over $\mathcal U_+$.
\end{proof}

 The number 
$\gamma (A)$ below was defined in Definition (\ref{gamma}). 

\begin{theorem} \label{finitesseforu}
Let $\mathcal U$ be a dense subring of $\mathbb R$. 
Suppose 
$\mathcal C$ is a path connected set of positive, 
conjugate  $n\times n$ matrices containing a matrix $A$ 
over $\mathcal U$. 
Then the number of distinct SSE-$\mathcal U_+$ classes 
of matrices which contain a matrix in $\mathcal C$ 
which is conjugate to $A$ over $\mathcal U$  is finite and cannot exceed 
\[
|\pi_0(\za)|/|\pi_0^{\mathcal U}(\za)|\ 
\]
which is not greater than $2^{\gamma(A)-1}$.
 
Consequently, the set of positive matrices 
conjugate over $\mathcal U$ to $A$ intersects 
only finitely many SSE-$\mathcal U_+$ classes. 
\end{theorem} 

\begin{proof} 
The upper bound by the displayed ratio follows from 
the lemma and the pigeonhole principle. The bound 
$2^{\gamma(A)-1}$ follows from 
Proposition \ref{gammabound} and the observation that 
$|\pi_0^{\mathcal U}(\za)|\geq 2$ (since 
$-I\in  \textnormal{Cent}_{\mathbb R} (A) $).
The final claim follows from the lemma and the 
fact that the set of positive matrices conjugate 
over $\mathbb R$ to $A$ contains only finitely 
many connected components. 
\end{proof} 

\section{Finding positive paths: the case of one nonzero eigenvalue} 
\label{findingsec}

\begin{definition} 
A matrix $A$ is {\it eventually rank m} if it is square and 
$\textnormal{rank}(A^k)=m$ for all large $k$.  That $A$ has eventual rank 
1 means that its 
characteristic polynomial has the form $\chi_A(t)=t^m(t-\lambda )$ 
with $\lambda$ nonzero. 
\end{definition} 

\begin{lemma} \label{positivepath1} 
Suppose $A$ and $B$ are positive $n\times n$ 
real matrices with spectral radius $\lambda $. 
Let $\ell_A , r_A$ be  positive left, right 
eigenvectors of $A$. Suppose 
 there is 
$U\in \textnormal{GL}(n,\mathbb R)$ with positive determinant 
such that 
$U^{-1}AU=B$ and 
the eigenvector $\ell_AU$ of $B$ 
is positive. 

 Then there is a path $\{U_t\}_{0\leq t\leq 1}$ 
in $\textnormal{GL}(n, \mathbb R )$ with $U_0=I$ and $U_1=U$ 
such that for $0\leq t \leq 1$,
the 
vectors 
$\ell_AU_t$ and $U_t^{-1}r_A$ 
are positive eigenvectors 
for  eigenvalue $\lambda$ 
for the matrix 
$A_t=U_t^{-1}AU_t$.

\end{lemma} 

\begin{proof} 
After passing to $(1/\lambda)A $ and $(1/\lambda)B$, 
without loss of generality we can suppose $\lambda=1$. 
Let $D$ be the diagonal matrix such that $D(i,i)=r_A(i)$. 
Then $D^{-1}AD$ is stochastic (the right eigenvector has 
every entry 1). Let  $D_t= (1-t)I + tD$. Then  
$\{ D_t^{-1}AD_t\}_{0\leq t\leq 1}$  is a path of positive 
matrices from $A$ 
to a positive stochastic matrix. The same argument  holds for $B$, 
so without loss of generality 
 we may suppose $A$ and $B$ are stochastic, with positive 
right eigenvector $r=r_A$ having every entry 1.

Because the subspace of row vectors 
$W=\{v\in \mathbb R^n: vr=0\}$ is the annihilator of $r$, 
the matrix $U$ maps $W$ to $W$. 
From the assumptions, if $\mathcal B$ is a basis of $W$, then 
the matrix representing the restriction to $W$ of $U$ with respect to 
$\mathcal B$ must have positive determinant. Because there is a path 
from the identity to this matrix in 
$\textnormal{SL}(n-1,\mathbb R)$, 
 there is a path $\{T_t\}_{0\leq t\leq 1}$ of invertible  
linear transformations
$T_t:W\to W $
such that $T_0=I$ and 
$T_1=U|_W$. 

Now we determine the required path 
of matrices, $\{U_{t}\}_{0\leq t\leq 1}$,
by specifying the corresponding linear transformations. 
For $w\in W$, set $wU_t=T_t(w)$. 
Also require 
$\ell_AU_t =(1-t)\ell_A+t\ell_AU:=\ell_t$. 
Then $\ell_t>0$ for all $t$. Because $W$ contains no positive vector 
and $W$ has codimension one, 
the matrices $U_t$ are well defined and invertible. 
The vectors 
$\ell_AU_t $ and $U_t^{-1}r_A:=r_t$ are  eigenvectors of $A_t$ 
for the eigenvalue 1. If $w\in W$, then there is a $w'$ in 
$W$ such that $w'U_t=w$, and therefore 
$wr_t= w' U_tU_t^{-1}r_A=w'r_A=0$. Since $W$ has codimension 1, 
there must be a constant $c_t$ such that $r_t=c_tr$. 
Because $0<\ell_Ar = \ell_aU_tU_t^{-1}r=\ell_tcr=c\ell_tr $, we 
conclude $c>0$.  Consequently, both $\ell_t$ and $r_t$ are positive, as required. 
Clearly, 
 $U_0=I$ 
and $U_1=U$.
\end{proof}

For the next lemma, we note that if $M$ is a 
 nilpotent real matrix and $c\neq 0$, then 
$M$ is conjugate to $cM$. For a concrete example, 
\[
\begin{pmatrix} 1 & 0 & 0 \\ 0 & c^{-1} & 0\\ 0 & 0 & c^{-2} 
\end{pmatrix} 
\begin{pmatrix} 0 & 1 & 0 \\ 0 & 0 & 1\\ 0 & 0 & 0
\end{pmatrix} 
\begin{pmatrix} 1 & 0 & 0 \\ 0 & c & 0\\ 0 & 0 & c^2 
\end{pmatrix} 
= c \begin{pmatrix} 0 & 1 & 0 \\ 0 & 0 & 1\\ 0 & 0 & 0
\end{pmatrix}  \ . 
\] 

\begin{lemma}\label{positivepath}
Suppose $A$ and $B$ are positive eventually rank one 
 matrices with  nonzero eigenvalue  
$\lambda $, and there is a path 
$(U_t)_{0\leq t\leq 1}$ in 
$\textnormal{GL}(n,\mathbb R)$ such that 
$U_0=I$, $U_1^{-1}AU_1=B$ and 
for each $A_t=U_t^{-1}AU_t$, the 
left and right eigenvectors of $A_t$ are positive. 

Then there is a path 
$(V_t)_{0\leq t\leq 1}$ in 
$\textnormal{GL}(n,\mathbb R)$ such that $V_0=I$,  $V_1=U_1$ and each 
matrix $V_t^{-1}AV_t$ is positive. 
\end{lemma} 

\begin{proof} Without loss of generality, suppose $\lambda =1 $. 
Let $\ell_t$ and $r_t$ be the left and right
positive  eigenvectors of $A_t$, 
normalized so that $\ell_t r_t= (1)$. Let $P_t=r_t\ell_t$. 
Let $Q_t$ be the nilpotent matrix such that $A_t=P_t+Q_t$. 
Then $P_t>0$, $P_tQ_t=Q_tA_t=0$, 
$U_t^{-1}P_0U_t=P_t$ and 
$U_t^{-1}Q_0U_t=Q_t$. 

Along the path, the entries of the $P_t$ 
have a positive lower bound $m$ and the absolute values of 
entries of the  $Q_t$ have a positive upper bound $M$. 
Choose a positive $\epsilon < m/M$. Then we have a path 
of positive conjugate matrices 
$P_t+ \epsilon Q_t$  from $P_0+\epsilon Q_0$ to 
$P_1 + \epsilon Q_1$. 
Taking $s$ from $\epsilon $ to 1, 
we get a path of positive conjugate matrices from $P+\epsilon Q_0$ to 
$P+Q_0=A$, and likewise from 
$P_1+\epsilon Q_1$ to 
$P_1+Q_1=B$. Composing paths, we get  a path of positive 
conjugate matrices from $A$ and $B$. Reparametrizing, we get 
the path $(V_t)_{0\leq t\leq 1}$ such that $V_0=I$ 
and $V_1=U_1$. 
\end{proof} 

The next result was proved in \cite{KR2} for 
the case $\mathcal U= \mathbb Q$ or $\mathbb R$. 
The positive matrix path construction below 
is a  matrix version of the 
invariant tetrahedra argument in \cite{KR2}. 
We describe the approach from \cite{KR2} of ``positive invariant 
tetrahedra'' in Appendix 
\ref{tetrahedrasec}.

\begin{theorem}\label{oneeigenvalue} 
Suppose $\mathcal U$ is a nondiscrete unital 
subring of $\mathbb R$, and $A$ and $B$ are nonnegative 
eventually rank one matrices
 which are SSE over 
 $\mathcal U$. 

 Then $A$ and $B$ are SSE over $\mathcal U_+$. 
\end{theorem} 

\begin{proof}
After passing to matrices SSE over $\mathcal U_+$, 
we may assume that $A$ and $B$ are 
primitive (using the eventually rank one assumption),  
and then positive (by 
 Proposition \ref{primtopos}). By Theorem \ref{ssetosim}, 
we may assume also we have $U\in \textnormal{GL}(n,\mathcal U)$ 
such that $AU=UB$, $\det U>0$ and $U$ sends a positive 
eigenvector of $A$ to a positive eigenvector of $B$. 
By Lemmas \ref{positivepath1} and \ref{positivepath}, 
there is a path $( U_t)_{0\leq t\leq 1}$
in $\textnormal{GL}(n,\mathbb R)$ such that 
$U_0=I, U_1=U$ and each $U_t^{-1}AU_t$ is positive. 
By Theorem \ref{paththeorem}(2), it follows that 
$A$ and $B$ are SSE over $\mathcal U_+$. 
\end{proof} 

\begin{remark} \label{oneremark} 
The one eigenvalue result above looks better in contrast 
to the lack of other general results. 
For every subring $\mathcal U$ of $\mathbb R$, 
for every primitive matrix $A$ over $\mathcal U$, 
it is unknown whether there exists an 
algorithm which given $B$ primitive and 
SSE over $\mathcal U$ to $A$ decides whether 
$B$ is SSE over $\mathcal U_+$ to $A$. 
\end{remark}

Theorem \ref{oneeigenvalue} is not a complete solution 
to the problem of classifying eventually rank one 
positive matrices 
over $\mathcal U$ (for an arbitrary dense subring of $\mathbb R$).
It is complete with regard to addressing positivity, but  
 we do not understand in general 
 how SSE refines SE 
over $\mathcal U$. Especially, 

\begin{problem}\label{rankonesseproblem} 
Suppose 
$\mathcal U$ is a nondiscrete unital subring of $\mathbb R$, 
 and $A,B$ are eventually rank one matrices which 
are shift equivalent  over the ring 
$\mathcal U$. Must they be strong shift equivalent over
$\mathcal U$?
\end{problem}

However, we are able to handle some classes of rings, as 
follows. 

\begin{theorem} \label{dedekindtheorem}
Suppose 
$\mathcal U$ is a nondiscrete unital subring of $\mathbb R$, 
 and $A$ is a nonnegative 
eventually rank one matrix over $\mathcal U$, with nonzero 
eigenvalue $\lambda$. Then the following hold. 
\begin{enumerate} 
\item 
If 
$\mathcal U$ is a Dedekind domain and $A$ is shift equivalent 
over $\mathcal U$ 
to the matrix $\begin{pmatrix}\lambda\end{pmatrix}$, then 
$A$ is SSE over $\mathcal U_+$ to $\begin{pmatrix}\lambda\end{pmatrix}$. 
\item 
If $\mathcal U$ is a principal ideal domain (e.g., 
a field), then $A$ is SSE over $\mathcal U_+$ to 
$\begin{pmatrix}\lambda\end{pmatrix}$. 
\end{enumerate} 
\end{theorem} 

\begin{proof} 
(1) Over a Dedekind domain $\mathcal U$, SE-$\mathcal U$ implies 
SSE-$\mathcal U$ \cite[Prop. 2.4]{BH2}, 
so Theorem \ref{oneeigenvalue} applies. 

(2) Over the principal ideal domain $\mathcal U$, 
$A$ is SSE-$\mathcal U$ to a nonsingular matrix 
\cite{E,Wi2}. 
This matrix can only be 
$\begin{pmatrix}\lambda\end{pmatrix}$, 
so again  
Theorem \ref{oneeigenvalue} applies. 
\end{proof} 

\begin{remark} \label{dedekindremark}
Example 2.2  in \cite{BH2} provides a
$2\times 2$ positive matrix 
$A'$ with eigenvalues $0$ and $1$ over the 
Dedekind domain $\mathcal U=\mathbb Z[\sqrt{15}]$ 
which is not shift equivalent over $\mathcal U$ 
 to a nonsingular matrix. 

Whenever a 
Dedekind domain $\mathcal U$  
is not a principal ideal domain, there will be matrices 
over $\mathcal U$ which are not SSE-$\mathcal U$ to 
a nonsingular matrix \cite{BH2}.
\end{remark}

\begin{remark} The proofs above 
easily adapt to prove the result stated next, 
which is one version of the  ``positive models'' result 
 in \cite{KR4}. 
\end{remark} 

\begin{theorem} \label{positivemodels}
Suppose $A$ and $B$ are $n\times n$ positive real matrices, 
and there are matrices $P,Q_0,Q_1,U$ such that the following hold. 
\begin{itemize} 
\item $P$ is a positive matrix  

\item $A$ and $B$ are internal direct sums,  
$A=P+Q_0$ and $B=P+Q_1$,  with the matrices 
$Q_0,Q_1$  nilpotent 

\item There is $U\in \textnormal{SL}(n,\mathbb R) $ such that 
$UP=PU$ and $UQ_0=Q_1U$. 
\end{itemize} 

Then there is a path of positive matrices $A_t=U_t^{-1}AU_t$ 
from $A=A_0$ to $B=A_1$, such that $U_0=I$ and $U_1=U$. 
\end{theorem} 

Theorem \ref{positivemodels} looks like a powerful tool, 
but so far it has not led to a general result. 

\begin{problem} \label{positive modelproblem}
Suppose $A$ is a positive real matrix.
Must there exist a positive matrix $P$, with 
$\textnormal{rank}(P)=\textnormal{rank}(P^2)$,  and a nilpotent matrix 
$Q$,  such that $PQ=QP=0$ and 
$A$ is strong shift equivalent over $\mathbb R_+$ 
to $P+Q$? 
\end{problem}

\section{The Connection Theorem} \label{connectionsec}

Below, $||C||_{\textnormal{max}}$ denotes the maximum absolute value of an 
entry of $C$. 

\begin{definition}\label{defnnbhd} 
 For an $n\times n$ real matrix $A$,  and $\epsilon >0$,  
$\mathcal N_{\epsilon}(A)$ denotes the set of $n\times n$ 
matrices $B$ such that  $||B-A||_{\textnormal{max}}< \epsilon $. 
\end{definition}

\begin{definition} For an $n\times n$ real matrix $A$ and $\epsilon >0$,  
$\mathcal N_{\epsilon}^{\textnormal{SE}}(A)$ denotes the set of $n\times n$ 
matrices $B$ which are shift equivalent over $\mathbb R$ to $A$ and 
satisfy $||B-A||_{\textnormal{max}}< \epsilon $. 
\end{definition}

\begin{theorem}[Connection Theorem]\label{connectiontheorem}
Suppose $A$ is an $n\times n$ positive matrix. Then there 
is a $\delta >0$ such that for any $B,C$ in 
$\mathcal N_{\delta}^{\textnormal{SE}}(A)$ and $m\geq n^2/2$,  there are row 
splittings of $B,C$ to positive conjugate matrices $B',C'$ 
such that there exists a path of positive conjugate matrices 
from $B'$ to $C'$, and therefore 
the matrices $B,C$ are 
SSE-$\mathbb R_+$, through positive matrices 
not larger 
than $(n^2/2)\times (n^2/2)$. 

Moreover, if $B$ and $C$ have their 
entries in a nondiscrete subring $\mathcal U$ of $\mathbb R$, then 
the splittings to $B'$ and $C'$ can be done through matrices over 
$\mathcal U_+$. If in addition $\mathcal U$ is a field, then 
the matrices $B,C$ are 
SSE-$\mathcal U_+$, through positive matrices
not larger 
than $(n^2/2)\times (n^2/2)$. 
\end{theorem} 

In the Connection Theorem, $A=B$ is allowed. 
Before proving the theorem, we record some 
immediate consequences. 

\begin{corollary} 
If $A$ is a 
 positive $n\times n$ matrix and 
$\dim (\ker (A))\geq 1$, then  
$A$ is  SSE over $\mathbb R_+$ to a positive  
$n\times n$ matrix $B$ such that $\dim (\ker (B))=1$. 
\end{corollary} 

\begin{proof} 
Given $\dim (\ker (A))>1$, 
there are positive 
matrices shift equivalent to $A$ 
which are arbitrarily close to $A$ 
such that $\dim (\ker (A))=1$, as one can 
see by replacing each superdiagonal zero in the Jordan form 
of the nilpotent part of $A$ with $\epsilon$. 
Therefore Theorem \ref{sepaththeorem} applies to 
prove the corollary. 
\end{proof}  
 
The Path Theorem \ref{paththeorem} produced SSE's over $\mathbb R_+$ from 
paths of positive matrices which are conjugate. The following 
  consequence of the Connection Theorem 
shows we only need those matrices to be shift equivalent. 

\begin{theorem}\label{sepaththeorem} Suppose $(A_t), 0\leq t \leq 1$, 
is a path of 
positive shift equivalent $n\times n$ matrices. Then $A_0$ and $A_1$ are 
SSE over $\mathbb R_+$. 
\end{theorem} 

\begin{proof}[Proof of Theorem \ref{sepaththeorem}] 
It follows from compactness that for $0\leq t\leq 1$, the Connection 
Theorem 
holds for $A_t$ in place of $A$, for a uniform $\epsilon $ 
(independent of $t$). Consequently $A_0$ and $A_1$ are SSE over $\mathbb R_+$. 
\end{proof}

The rest of this section is devoted to 
the proof of the Connection Theorem, which relies also on 
Theorem \ref{matrixconnection}.  The consequences of the Connection Theorem 
in later sections can be read independent of the proof of 
Theorem \ref{matrixconnection}. 
 
We  prepare for the proof with two lemmas. 
The idea behind Lemma \ref{jordanlemma}, apart from the 
generality of $\mathcal U$, 
can be found in \cite{KR5} and \cite[Lemma 1]{KR2}.
 $J_k$ denotes the standard  $k\times k$ Jordan block  matrix 
(zero except for entries 1 in positions $(i,i+1)$, $1\leq i < k$). 
$J_0(A)$ denotes the nilpotent part of the Jordan form of a matrix $A$.

\begin{lemma} \label{jordanlemma} 
Suppose $\mathcal U$ is a nondiscrete subring of $\mathbb R$, 
$n\in \mathbb N$, $A$ and $B$ are positive $n\times n$ 
matrices over $\mathcal U$, $||A-B||<\epsilon$, 
$0<\alpha < 1$ and the $i$th rows of $A$ and $B$ are denoted 
$w_A$ and $w_B$. 
Let $A'$ be the 
 $(n+1)\times (n+1)$ matrix obtained by a splitting 
of its $i$th row corresponding 
to $w = \alpha w + (1-\alpha )w$. Suppose $B$ is an $n\times n$ 
positive matrix. 
Then there is an $(n+1)\times (n+1)$ positive matrix $B'$, 
obtained by a splitting of its $i$th row,
 such that 
the following hold. 
\begin{enumerate} 
\item 
$||B'-A'|| <  \epsilon $ 
\item 
 $J_0(B')$ can be chosen to be either 
(i) $J_0(B) \oplus [0]$ or, (ii) for 
any $k$ such that $J_0(A) $ has a $k\times k$ Jordan block, 
$J_0(B')$ can be obtained  
from $J_0(B) $ by replacing a 
$k\times k$ Jordan block 
with a $(k+1)\times (k+1)$ Jordan block. 
\item 
If $B$ has entries in $\mathcal U$, then 
the splitting to $B'$ can be done over $\mathcal U$. 
\end{enumerate} 
\end{lemma} 

\begin{proof}
 Without loss of generality, we may assume 
$\alpha \in \mathcal U$. 
For (i), use the splitting of  row $i$ by $\alpha w + (1-\alpha )w$. 
For (ii), we have by assumption that there is a vector $v$ 
such that $vB^k=0, vB^{k-1} \neq 0$ and $v$ is not in the image of 
$A$. Let $\mathbb F$ denote the field of fractions of $\mathcal U$. 
The kernel of $B^k$ contains a dense subset of vectors from 
$\mathbb F^n$. Pick $v'$ in $\ker(B^k)\cap \mathcal U^n$ with 
$||u-v'||$ small enough that $v'B^{k-1}\neq 0$ and 
$v'$ is not in the image of $B$. Pick $\beta \neq 0$ from $\mathcal U$ 
such that $\beta v'\in \mathcal U^n$. 
Pick $\gamma >0$ in $\mathcal U$ arbitrarily small, and small enough 
that for $s:= \alpha w_B + \gamma \beta v$ we have $s>0$ and 
$w_B-s>0$. Now form $B'$ by splitting row $i$ of $B$ according 
to  $w_B=s + (w_B -s)$. Then $B'$ is conjugate to 
$ \left(\begin{smallmatrix} B & 0 \\s & 0
\end{smallmatrix} \right)$ and to 
 $ \left(\begin{smallmatrix} B & 0 \\\gamma \beta v' & 0
\end{smallmatrix} \right)$. This last matrix has the required 
Jordan form. For small $\gamma$, we have $||B'-A'||< \epsilon$. 
\end{proof}

For the next lemma we establish some notation. 
For a square matrix $M$, we let $G_M,H_M$ denote
the unique matrices $G,H$ such that $M=G+H$,
$GH=HG=0$, $H$ is nilpotent and $\textnormal{rank}(G)=
\textnormal{rank}(G^2)$. We also use $U_M,N_M,F_M$ 
to denote matrices such that $U^{-1}_MMU_M=F_M\oplus N_M$, where 
$F_M$ is nonsingular and $N_M$ is nilpotent, 
and for concreteness $N_M$ is in Jordan form.  
In this case, 
$U^{-1}_MG_MU_M=F_M\oplus 0$ and 
$U^{-1}_MH_MU_M=0\oplus N_M$, where $0$ denotes a zero 
matrix of appropriate size. 
The norm used below is the max norm. 

Finally,   we define a notion critical for our proof 
of the Connection Theorem. 

\begin{definition} \label{localconnectednessdefinition}
Suppose $N$ is an $n\times n$ nilpotent matrix and 
$\mathcal C$ is a 
 conjugacy class of  $n\times n$ 
nilpotent matrices. We say $\mathcal C$ is 
{\it locally connected at} $N$ if  for every $\epsilon >0$ 
there exists 
$\delta >0$ such that any two matrices in 
$\mathcal C \cap \mathcal N_{\delta}(N)$ are connected 
by a path in $\mathcal C \cap\mathcal N_{\epsilon}(N)$.
\end{definition}

\begin{lemma} \label{reduction} Suppose $\epsilon >0$ and $A,B$ are positive 
$n\times n$ matrices such that $G_A$ and $G_B$ are conjugate,  
and the conjugacy class $\mathcal C$ of 
$N_B$ is locally connected at $N_A$. 

Then there is a $\delta >0$ such that any two 
matrices in $\mathcal N_{\delta}(A)$ which are conjugate to 
$B$ are connected by a 
path of positive conjugate matrices 
in $\mathcal N_{\epsilon}(A)$. 
\end{lemma} 

\begin{proof} 
Fix a matrices $U,F_A,N_A$ such that $U^{-1}AU=F_A \oplus N_A$. 
The idea of the proof is the following. 
For $\delta $ small enough, 
given two matrices conjugate to $B$ inside 
$\mathcal N_{\delta}(A)$, we show there are paths from them in 
  $\mathcal N_{\epsilon}(A)$ to matrices
$C_1,C_2$  which are 
conjugated by $U$ to matrices 
$C'_1=\left(
\begin{smallmatrix} F_A & 0 \\ 0 & N_1 
\end{smallmatrix}
\right) $ 
and 
$C'_2=\left(
\begin{smallmatrix} F_A & 0 \\ 0 & N_2 
\end{smallmatrix}
\right) $ 
such that there is a path of conjugate 
matrices from $N_1$ to $N_2$ which induces 
a path from $C'_1$ to $C'_2$ which 
$U^{-1}$ conjugates to the desired path 
from $C_1$ to $C_2$. We spell out quantifiers for this next, 
for a matrix $C$ conjugate to $B$.

Take $\epsilon $ to be smaller than $||A||$. 
Pick $\epsilon_1>0$ such that 
$||N'-N_A||<\epsilon_1$ implies $||U(F_A\oplus N')U^{-1}-A||< \epsilon$ . 
Pick $\epsilon_2>0$ such that if conjugate nilpotent matrices 
$N_1,N_2$ are in $\mathcal N_{\epsilon_2}(N_A)$, then there is a path 
of conjugate matrices in $\mathcal N_{\epsilon_2}(N_A)$ from 
$N_1$ to $N_2$. Pick $\epsilon_3>0$ such that  
$||X-A||<\epsilon_3$ implies $||U^{-1}(X-A)U||<\epsilon_2$ .  

Finally, pick $\delta >0$ such that if $C$ is conjugate to 
$B$ and  $||A-C||<\delta$, then the conjugate matrices 
$A^n$ and $C^n$ are sufficiently close that there is a 
$V$ in $\textnormal{SL}(n,\mathbb R)$ such that $V^{-1}A^nV=
C^n$ (which means 
$V^{-1}(G_A)^nV=
(G_C)^n$)  and $||V-I||$ is sufficiently small that the following 
hold: 
\begin{enumerate} 
\item 
$V^{-1}G_AV=G_C$ . 
\item 
$||V^{-1}CV-A||<\epsilon_3$ . 
\item 
There is a path $(V_t)$ in 
$\textnormal{SL}(n,\mathbb R)$ from $I=V_0$ to
$V=V_1$ remaining sufficiently close 
to $I$ that 
$V_t^{-1}CV_t\in \mathcal N_{\epsilon}(A)$, $\ 0\leq t\leq 1$ . 
\end{enumerate} 

Because $V$ maps $\textnormal{ker}(C^n)$ onto $\textnormal{ker}(A^n)$, 
the matrix 
$U^{-1}(V^{-1}BV)U$ has the form 
$\left(
\begin{smallmatrix} F_A & 0 \\ 0 & N' 
\end{smallmatrix}
\right) $, with 
\begin{align*} 
\begin{pmatrix} 0 & 0 \\ 0 & N' 
\end{pmatrix} 
- 
\begin{pmatrix} 0 & 0 \\ 0 & N_A 
\end{pmatrix} 
\ &= \ 
\begin{pmatrix} F_A & 0 \\ 0 & N' 
\end{pmatrix} 
-
\begin{pmatrix} F_A & 0 \\ 0 & N' 
\end{pmatrix} 
\\ 
&= \
 U^{-1}(V^{-1}BV-A)U
\ . 
\end{align*} 
Since $|| V^{-1}BV - A||<\epsilon_3$, we have 
\[
||N'-N_A||= 
\Bigg|\Bigg|
\begin{pmatrix} 0 & 0 \\ 0 & N' 
\end{pmatrix} 
- 
\begin{pmatrix} 0 & 0 \\ 0 & N_A 
\end{pmatrix} 
\Bigg|\Bigg|
< \epsilon_2
\] 
which shows  our $\delta $ is small enough 
to establish the conclusion of the lemma. 
\end{proof}

\begin{proof}
[Proof of the Connection Theorem]
If $B$ and $A$ are conjugate, then the theorem follows 
from the Path Theorem. So we assume $B$ and $A$ are 
not conjugate, which implies $n\geq 3$. 

The strategy of the proof is to take a row splitting of 
$A$ to a suitable positive matrix $A'$ chosen independent of $B$ and $C$; 
pick a suitable class $\mathcal C$ of nilpotent matrices 
which is locally connected at $N_{A'}$; and then for 
 $\delta >0  $ taken from Lemma \ref{reduction}, 
perform  row splittings of $B$ and $C$ 
to matrices $B'$ and $C'$ in $\mathcal C\cap \mathcal N_{\delta}(A')$.

To begin the proof, we choose 
$A'$ and a sequence of matrices 
$A_i$, 
$n\leq i\leq m$, such that  $A=A_n$, $A'=A_m$ 
and for $n\leq i < m$, the matrix $A_{i+1}$ 
is obtained by splitting a row of $A_i$ 
into two proportional rows. 
 For example, if $m=sn+p$ with 
$s\geq 1 $ and $1\leq p <n$, then we could split 
each of the first $p$ rows into $s$ equal rows and split 
each of the remaining rows into $s-1$ equal rows. 
Then $\textnormal{rank}(A')=
\textnormal{rank}(A)$, and  $A'$ is conjugate to 
$A\oplus 0_{m-n}$.

Using notations from Definition \ref{betasnotation}, 
we define  
 $h=\max\{h(A),h(B),h(C)\}$ and 
$\beta=\max\{\beta(A),\beta(B),\beta(C)\}$. 
(If  $B$ is in 
$\mathcal N_{\epsilon}^{\textnormal{SE}}(A)$ and $\epsilon $ 
is sufficiently small, then $h(B)\geq h(A)$ must 
hold.) 
Let $r$ be the rank of $F_A$. For 
any $C$ shift equivalent to $A$ over $\mathbb R$, 
we have $F_C$ conjugate to $F_A$ and 
$\textnormal{rank}(F_C)=r$. So, if $C$ is $m\times m$ 
and shift equivalent to $A$, then 
$C$ is conjugate to $F_A\oplus N_C$, and the 
nilpotent matrix $N_C$ is $(m-r)\times (m-r)$. We will 
define $\mathcal C$ to 
be the conjugacy class of $(m-r)\times (m-r)$
matrices which contains 
the matrix $(J_h)^{\beta}\oplus 0_q$, 
where $q$ is $ m - r-\beta h$. We need to check 
this makes sense ($q\geq 0$) and also that 
$q>0$. Because  $n\geq 3 $ and 
\[
(\beta )(h) \leq  
\Big(\frac{n-r}2\Big)(n-r-1)
= \frac{n^2}2 -\frac r2(3n-1) \ , 
\]
we have 
\begin{align*} 
(\beta )(h) +r 
&\leq 
\frac{n^2}2 -\frac r2(3n-3)\\
&\leq 
\frac{n^2}2 -\frac 12(3(3)-1) = 
\frac{n^2}2 -3r 
< \frac{n^2}2 \leq m \ ,
\end{align*} 
and therefore $ q > 0 $. 
Clearly 
\begin{enumerate} 
\item 
$\mathcal C$ has a zero Jordan block (because $q>0$). 
\item 
$h(\mathcal C) = h \geq h(A) = h(A')$ . 
\item 
$\beta^{\textnormal{top}}(\mathcal C)
= \beta \geq \beta (A) = \beta (A')$ . 
\end{enumerate} 
It follows from Theorem \ref{matrixconnection} 
that $\mathcal C$ is locally connected at 
$N_{A'}$. Therefore, we can specify any $\epsilon >0$ 
and for that $\epsilon $ pick $\delta > 0$ in the statement 
of Lemma \ref{reduction}. This will be the 
$\delta$ in the statement of the Connection Theorem. 

Now we describe the splitting of $B$ to $B'$ 
(the argument to split $C$ to $C'$ is the same), 
such that $N_{B'}\in \mathcal C$. 
Starting with $B_n=B$ and $||B-A||<\delta$, 
for $n\leq i <m$ we inductively appeal to 
Lemma \ref{jordanlemma} to split $B_i$ to 
$B_{i+1}$, with $||A_{i+1}-B_{i+1}||<\delta $. 
We use condition (2) of
Lemma \ref{jordanlemma} at each stage as 
follows, applying the first listed criterion 
for which $B_i$ satisfies the required 
condition. 
\begin{enumerate} 
\item 
If $N_{B_i}$ has no zero Jordan block, 
then  $N_{B_{i+1}}=N_{B_i}\oplus 0_1$. 
\item 
If $N_{B_i}$ has fewer than $\beta $ 
Jordan blocks which are $2\times 2$ or larger, 
then 
$N_{B_{i+1}}$ is $ N_{B_i}$
with  a zero block 
replaced by $2\times 2$ Jordan block. 
\item 
If 
$N_{B_i}$ has 
 a $k\times k$ Jordan block of 
 with $2\leq k<h$, 
then for a maximum such $k$, 
$N_{B_{i+1}}$ is $ N_{B_i}$
with  a Jordan $k$-block 
replaced by a $(k+1)$-block. 
\item 
If $N_{B_i}$ has $\beta $ 
 Jordan blocks which are $h\times h$, 
then $N_{B_{i+1}}=N_{B_i}\oplus 0_1$. 
\end{enumerate} 
Clearly $B'$ has the required form, and there is a 
path of positive conjugate matrices from $B'$ to $C'$. 
 By the Path Theorem (\ref{paththeorem}),
$B'$ and $C'$ are SSE over $\mathbb R_+$, through 
positive $m\times m$ matrices. 

The ``Moreover'' condition of keeping splittings over 
$\mathcal U$ can be achieved by condition (3) 
of Lemma \ref{jordanlemma}. 
If $\mathcal U$ is a field, then the conjugacy of $B'$ and $C'$ over 
$\mathbb R$ implies their conjugacy over 
$\mathcal U$, and (again using that $\mathcal U$ 
is a field) by the Path Theorem we have an SSE-$\mathcal U_+$ 
through positive matrices from $B'$ to $C'$, and 
hence also from $B$ to $C$. This finishes the proof. 
\end{proof}

\section{From SSE over $\mathbb R_+$ 
to SSE over $\mathbb U_+$} 
\label{sseoverplus}

\begin{theorem} \label{positivetheorem}
Suppose $A,B$ are positive matrices over $\mathbb R$ which 
are SSE-$\mathbb R_+$. Then 
$A,B$ are SSE over $\mathbb R_+$ through positive matrices: 
there are positive 
matrices  $A=A_0,A_1, \dots , A_{\ell}=B$, with 
$A_i$ elementary SSE over $\mathbb R_+$ to $A_{i+1}$, $0\leq i < \ell$. 
\end{theorem}

\begin{proof} 
Appealing to Theorem \ref{easystreet},
choose a primitive matrix $C$ over $\mathbb R_+$ 
 and 
a nonsingular diagonal matrix $D$ over $\mathbb R_+$ 
such that $C$ is reached from $A$ by finitely many row splittings 
through  primitive matrices  and 
 $D^{-1}CD$ is reached from $B$ by finitely many column splittings 
through primitive  matrices. 
Then choose a construction, 
by the procedure described  
 in Appendix \ref{booleanappendix}, 
of a positive matrix 
 $\widetilde C$ SSE over $\mathbb R_+$ to $C$. 
Appealing to the Connection 
Theorem, 
choose $\delta >0$ such that 
matrices shift equivalent to 
$\widetilde C$
and 
$\delta $ close to  
$\widetilde C$ are SSE to each other through positive matrices. 
Appealing to Lemma \ref{pospos}, 
pick $\nu >0$ such that for any positive matrix 
$M$  with $||M-C||
< \nu$ there is a strong shift equivalence over 
$\mathbb R_+$ through positive matrices from $M$ to a positive matrix $\widetilde M$ 
 such that $||\widetilde M - \widetilde C||<\delta$. 

Now, perform row splittings from $A$ through positive matrices 
 to a positive matrix $A^*$ such that 
$||A^*-C||<\nu$. This is done simply by approximating 
the string of splittings from $A$ to $C$ over $\mathbb R_+$ by 
a string of splittings 
 from $A$ to $C$  through positive matrices.
By composition, we have an SSE over $\mathbb R_+$ from $A$ to a matrix 
$\widetilde A$ within $\delta$ of $\widetilde C$.  

The argument to obtain
 an SSE over $\mathbb R_+$ through positive matrices from $B$ to a matrix 
$\widetilde B$ within $\delta$ of $\widetilde C$ is similar. 
 We obtain a positive matrix $B^*$ 
near $D^{-1}CD$ by approximating 
the given column splittings from $B$ to $D^{-1}CD$. 
There is an elementary SSE  over $\mathbb R_+$  from 
$B^*$ to the positive matrix $D^{-1}B^*D:=B^{**}$. 
We take $B^*$ close enough to 
$D^{-1}CD$
to guarantee $||B^{**}-C||<\nu$. Then we apply 
 Lemma \ref{pospos} again to obtain the SSE through positive matrices from 
$B^{**}$ to the desired $\widetilde B$ near $\widetilde C$.
By the Connection Theorem, 
$\widetilde A$ and $\widetilde B$ are SSE over $\mathbb R_+$ through 
positive matrices. By composing the assembled SSEs, the theorem is proved.
\end{proof} 

Theorem 8.2 below was proved in \cite{KR5} for $\mathcal U=\mathbb Q$ 
under the additional assumption that 
$A$ and $B$ are SSE over $ \mathbb R_+$ through positive matrices.

\begin{theorem} \label{field}
Let $\mathcal U$ be a subfield of $\mathbb R$.
Suppose $A,B$ are positive matrices over $\mathcal U$ 
which  are SSE over $\mathbb R_+$. Then 
$A$ and $B$ are SSE over $ \mathcal U_+$, through positive matrices. 
\end{theorem} 

\begin{proof} 
We examine the proof of Theorem \ref{positivetheorem} 
and check that the SSEs constructed in the various steps 
can be taken through positive matrices over $\mathcal U$.

The splittings to $A^*$ and  $B^*$ can be done over  $\mathcal U$.

Approximate $D$ by a diagonal matrix $D'$ over  $\mathcal U_+$.
In place of $B^{**}$, use $(D')^{-1}B^*D':=B'$. 
Because $\mathcal U$ is a field, the positive matrix $B'$ 
has its  entries  in
$\mathcal U$ and is ESSE over $\mathcal U_+$ 
to 
$B^*$ (by the matrices 
$(D')^{-1}$ and $B^*D'$). 

The matrices 
$\widetilde A$ 
and  $\widetilde B$ are constructed 
from $A^*$ and $B'$ (matrices over 
$\mathcal U$) by appeal to 
Proposition \ref{primtopos}, and therefore they can be 
taken over  $\mathcal U$. Lemma \ref{pospos} allows the 
approximating SSE through positive matrices to be taken over $\mathcal U$.
 Because $\mathcal U$ is a field, the 
Connection Theorem then 
gives an SSE through positive matrices over $\mathcal U$ 
from 
$\widetilde A$ 
to  $\widetilde B$.

This completes the proof.
\end{proof}

\begin{theorem}\label{finitessefield}  Suppose $A$ is a positive $n\times n$ 
matrix. The collection of positive $n\times n$ matrices 
conjugate over $\mathbb R$ to $A$ contains  only finitely many 
SSE-$\mathbb R_+$ classes. 
\end{theorem} 
\begin{proof}
This follows immediately from Theorem \ref{finitesse} and Theorem \ref{field}.
\end{proof}

\begin{problem}\label{sserefinementproblem}
Let $\mathcal U$ be a dense subring of 
$\mathbb R$ Suppose  $A$ and $B$ are positive matrices which are strong shift 
equivalent over $\mathcal U$ and also over $\mathbb R_+$. 
Must they be strong shift equivalent over $\mathcal U_+$? 
\end{problem}

\appendix 
\section{Making SSE nondegenerate} \label{nondegenerateappendix} 

The purpose of the appendix is to prove
Proposition \ref{nondegeneratesseprop}, which we now restate. 
\newline \newline 
{\bf Proposition \ref{nondegeneratesseprop}} 
{\it 
Suppose $\ri$ is a ring which is torsion free as an additive group. 
Suppose nondegenerate matrices $A$ and $B$ are SSE over $\ri$. 
Then they are SSE through a chain of ESSEs 
$A_{i-1} =R_iS_i, A_{i+1}=S_iR_i$ such that all the matrices 
$A_i, R_i,S_i $ are nondegenerate. } 
 
Recall, a matrix is nondegenerate if it has no zero row and no zero 
column. Below, by a nonzero matrix we mean a matrix which is not 
the zero matrix.

We will prove  Proposition 
 \ref{nondegeneratesseprop} after proving three lemmas. 

\begin{lemma} \label{onezerotonozero}
Suppose $\ri$ is a ring and there are matrices 
$A,B,C,R,S,R',S'$ 
over $\ri$ 
satisfying the following conditions
$A = RS\ , \   B=SR\ ;\ \ 
B = R'S'\ , \   C=S'R'   ;\ \ 
A \neq 0\ ,  \ \  B = 0 \ ,  \ \  C \neq 0 \  
$.  
Then there are nonzero matrices $A=A_0, A_1, \dots , A_6=C$  such 
that $A_{i-1} $ is ESSE over $\ri$ to $A_i$, $1\leq i \leq 6$. 
\end{lemma} 

\begin{proof} 
In block form, define 
\begin{alignat*}{2} 
A_1 &=  \begin{pmatrix} RS & 0 \\ S & 0 \end{pmatrix} 
& 
A_4 &=  \begin{pmatrix} 0 & 0 \\ S' & 0 \end{pmatrix} 
= \begin{pmatrix} R'S' & 0 \\ S' & 0 \end{pmatrix} \\ 
A_2 &=  \begin{pmatrix} 0 & 0 \\ S & SR \end{pmatrix} 
=  \begin{pmatrix} 0 & 0 \\ S & 0 \end{pmatrix} \qquad 
&
A_5 &=  \begin{pmatrix} 0 & 0 \\ S' & S'R'  \end{pmatrix} \\ 
A_3 &=  \begin{pmatrix} 
0 & 0 & 0 \\ S & 0 & 0 \\ 0 & S' & 0 \end{pmatrix} 
&
A_6 & =  \begin{pmatrix} S'R' \end{pmatrix} = B \ . 
\end{alignat*}
The matrices $S$ and $S'$ cannot be zero (because $A$ and $C$ 
are not zero), so the $A_i$ are not zero. An ESSE from 
$A\to A_1$ is given by 
\[ 
A  = 
  \begin{pmatrix} I & 0  \end{pmatrix} 
 \begin{pmatrix} RS  \\ S  \end{pmatrix} 
\ , \qquad 
A_1 =  
 \begin{pmatrix} RS  \\ S  \end{pmatrix} \\ 
  \begin{pmatrix} I & 0  \end{pmatrix}  \ . 
\] 
There are ESSEs $A_2\to A_3, A_3\to A_4, A_5\to A_6$ of this 
type or a transpose type. 
An ESSE $A_1\to A_2$ is given by 
\begin{align*}  
A_1  & = 
  \begin{pmatrix} RS  & 0  \\ S & 0 \end{pmatrix} 
 =  \ 
  \begin{pmatrix} RS  & RSR \\  S  & SR  \end{pmatrix} 
  \begin{pmatrix} I  & -R \\ 0  & I  \end{pmatrix}  \\
A_2  & = 
  \begin{pmatrix} 0  & 0 \\ S & SR  \end{pmatrix} 
 =  
  \begin{pmatrix} I  & -R \\ 0  & I  \end{pmatrix} 
  \begin{pmatrix} RS  & RSR \\  S  & SR  \end{pmatrix} 
\end{align*} 
The remaining ESSE $A_4\to A_5$ is of the same type. 
\end{proof} 

\begin{lemma} \label{longzerotoonezero} 
Suppose $\mathcal U$ is a unital semiring and $A$ is ESSE over 
$\mathcal U$ to $0_m$, the $m\times m$ zero matrix. 
Then $A$ is ESSE over $\mathcal U$ to $0_{m+k}$, for 
all $k$ in $\mathbb N$. 
\end{lemma} 
\begin{proof} 
We are given $A=RS,\  0_M = SR$. Then 
$A= 
\left(
\begin{smallmatrix} R & 0 
\end{smallmatrix}
\right) 
\left(
\begin{smallmatrix} S \\ 0 
\end{smallmatrix}
\right) 
$ 
and 
$0_{m+k}= 
\left(
\begin{smallmatrix} S \\ 0 
\end{smallmatrix}
\right) 
\left(
\begin{smallmatrix} R & 0 
\end{smallmatrix}
\right) 
$ 
where $0$ denotes a zero block of the necessary size. 
\end{proof} 

\begin{lemma} \label{nozerotonondegenerate} 
Let $\mathcal U$ be a unital ring which is torsion free 
as an additive group. Suppose $A$ is an $n\times n$ matrix 
over $\mathcal U$ which is not the zero matrix. 
Then there is $V$ in $\textnormal{SL}(n,\mathbb Z)$ such that 
$V^{-1}AV$ is nondegenerate. 
\end{lemma} 
\begin{proof} We can assume $n>1$.  
For example, suppose row 1 of $A$ is nonzero and 
row $n$ of $A$ is zero. Given $M\in \mathbb N$, 
let $E$ be the basic elementary matrix such that 
$E(n,1)=M$ and set  $C= EAE^{-1}$. Then 
\begin{align*} 
C (i,1) &= A(i,1) - MA(i,n) \quad 
\textnormal{ if } i<n \\ 
C(n,1) & = MA(1,1)-M^2A(1,n) \\ 
C(n,j) & = MA(1,j) 
\textnormal{ if } j>1 \\ 
C(i,j) & = A(i,j) 
\textnormal{ if } i<n 
\textnormal{ and } j>1 \ .  
\end{align*} 
Appealing to the torsion free assumption, 
choose $M$ such that 
\begin{align*} 
1\leq i < n \textnormal{ and }  
A(i,1)\neq 0 
& \implies 
A(i,1) \neq MA(i,n) 
\ . 
\end{align*} 
Then $A(i,j)\neq 0 $ implies $C(i,j)\neq 0$. 
In addition, row $n$ of $C$ is not zero, as follows.
If there exists $j>1$ with $A(1,j)\neq 0$, then 
$C(n,j)\neq 0$; otherwise, $A(1,1)$ is the 
only nonzero entry of row 1 of $A$, and 
$C(n,1) = MA(1,1) \neq 0$. 

Iterating this move as needed, with 
other indices $(i,j)$ in place of $(1,n)$, 
and interchanging the role of column and row as needed, 
we produce $V\in \textnormal{SL}(n,\mathbb Z)$ such that $V^{-1}AV$ 
is nondegenerate. 

\end{proof} 

\begin{remark} 
We are not concerned in this paper with finding the 
sharpest version of Proposition \ref{nondegeneratesseprop}.
However, we note that Lemma \ref{nozerotonondegenerate} would be false if 
the ``torsion free'' assumption were simply dropped. 
Over the field $\mathbb Z/2$, let 
$A= 
\left(
\begin{smallmatrix} 1 & 0 \\ 0 & 0 
\end{smallmatrix}
\right)$ and 
$B= \left(
\begin{smallmatrix} 1 & 1 \\ 1 & 1 
\end{smallmatrix}
\right)
$. 
Then $A\neq 0$ but $A$ is not conjugate over $\mathbb Z/2$ to 
a nondegenerate matrix, because $B$ is the only rank one 
nondegenerate $2\times 2$ matrix over $\mathbb Z/2$, and 
$A^2 \neq 0 = B^2$. 
\end{remark} 

We are now ready to prove Proposition 
\ref{nondegeneratesseprop}.

\begin{proof}[Proof of Proposition \ref{nondegeneratesseprop}] 
We are given some string of ESSEs over $\mathcal U$ from $A=A_0$ to
$B=A_{\ell}$, 
\[
A=A_0 \to A_1 \to A_2 \to \cdots \to A_{\ell -1} \to A_{\ell}=B  
\] 
with matrices $R_i,S_i$ over $\mathcal U$ such that 
$A_{i-1}=R_iS_i, \ A_{i}= S_iR_i$, for $1\leq i \leq \ell$. 

Suppose for some $i$ and some $k>2$ that $A_i$ and $A_{i+k}$ are 
not zero, but $A_j$ is a zero matrix for $i<j<i+k$. 
By Lemma \ref{longzerotoonezero}, 
there is a zero matrix $Z$ ESSE to $A_i$ and to $A_{i+k}$. 
We replace the ESSEs 
$A_i \to A_{i+1} \to \cdots \to A_{i+k}$ 
with ESSEs $A_i\to Z \to A_{i+k}$. After iterating this 
move as necessary, we may assume that $A_i=0 $
implies $A_{i-1}\neq 0$ and $A_{i+1}\neq 0$.  

Then, by Lemma \ref{onezerotonozero}, if $A_i=0$, 
we may replace the ESSEs $A_{i-1} \to A_i\to A_{i+1}$ 
with a string of ESSEs from $A_{i-1}$ to $A_{i+1}$  through nonzero 
matrices. After iterating as needed, we may assume every $A_i$
is not zero. 

If $0< i< \ell$ and 
$U^{-1}A_iU=A'_i$, then we can replace 
\begin{equation}
\xymatrix@=
10pt{
A_{i-1} 
\ar^{(R_i,S_i)}[rrrrr] &&&&& A_i 
\ar^{(R_{i+1},S_{i+1})}[rrrrrrr]&&&&&&&  A_{i+1}  
}
\end{equation}
with 
\begin{equation}
\xymatrix@=
10pt{
A_{i-1} 
\ar^{(R_iU,U^{-1}S_i)}[rrrrr] &&&&& A'_i 
\ar^{(U^{-1}R_{i+1},S_{i+1}U^{-1})}[rrrrrrr]&&&&&&&  A_{i+1}  
}\ . 
\end{equation}
Thus by repeated application of this move, with 
$U^{-1}A_iU=A'_i$ nondegenerate by Lemma \ref{nozerotonondegenerate}, 
we can pass to an SSE through nondegenerate matrices as required. 
\end{proof}

\section{Boolean matrices and positivity } \label{booleanappendix} 

Let $\mathcal U$ be a nondiscrete unital subring of $\mathbb R$. 

We will include in this section a proof 
of the result of \cite{KR0} that every 
primitive matrix over $\mathcal U$ is SSE over $\mathcal U_+$ to a 
positive matrix. 
As in  \cite{KR0},this is done by proving the result for Boolean 
matrices and then carrying it over. 
We can then prove the approximation result Lemma 
\ref{pospos}, which we need  in Section 
\ref{sseoverplus}.

Boolean matrices are matrices with entries in the 
Boolean semiring $\mathcal B= \{ 0,1 \}$, in which $1+1=1$.  
The usual row and column splitting and amalgamations can be 
used to produce SSEs over $\mathcal B$. In 
particular, if a row $i$ of $A$ is less than or equal 
to row $j$ of $A$, then adding column $j$ of $A$ to column $i$ 
produces a matrix $B$ SSE over $\mathcal B$ to $A$; for the 
corresponding  elementary matrix $E$, we have  
$A=EA$ and $B=AE$.
An example,  
assuming row 1 of $A$ is less than or equal to row 2, is 
\begin{align*}
A= \begin{pmatrix} 
a_1 & a_2 & a_3 \\
b_1 & b_2 & b_3 \\
c_1 & c_2 & c_3 
\end{pmatrix} 
& = EA = 
\begin{pmatrix} 
1 & 0 & 0 \\ 
1 & 1 &0 \\ 0 & 0 & 1 
\end{pmatrix} 
\begin{pmatrix} 
a_1 & a_2 & a_3 \\
b_1 & b_2 & b_3 \\
c_1 & c_2 & c_3 
\end{pmatrix} \\ 
B =
\begin{pmatrix} 
a_1+a_2  &a_2  & a_3 \\
b_1+b_2  &b_2  & b_3 \\
c_1+c_2  &c_2  & c_3 
\end{pmatrix} 
&= AE = 
\begin{pmatrix} 
a_1 & a_2 & a_3 \\
b_1 & b_2 & b_3 \\
c_1 & c_2 & c_3 
\end{pmatrix} 
\begin{pmatrix} 
1 & 0 & 0 \\ 
1 & 1 &0 \\ 0 & 0 & 1 
\end{pmatrix} \ . 
\end{align*}
If $A$ is the Boolean image of a matrix $A'$ over $\mathcal U_+$, 
then there are $E',B'$ over   $\mathcal U_+$ 
with Boolean images  $E,B$ such that  
$A'=E'A'$ and $B'=A'E'$. Here, $E'$ is an elementary matrix 
whose off diagonal entry  can be chosen arbitrarily close to 
zero, and $B'$ is conjugate 
over $\mathcal U$ to $A'$. In 
the example (using the letter entries in $A$ above to denote 
entries of $A'$, for simplicity), we have 
\begin{align*}
B' &=(E')^{-1} A'E' =
\begin{pmatrix} 
1 & 0 & 0 \\ 
-\epsilon & 1 &0 \\ 0 & 0 & 1 
\end{pmatrix} 
 \begin{pmatrix} 
a_1 & a_2 & a_3 \\
b_1 & b_2 & b_3 \\
c_1 & c_2 & c_3 
\end{pmatrix} 
\begin{pmatrix} 
1 & 0 & 0 \\ 
\epsilon & 1 &0 \\ 0 & 0 & 1 
\end{pmatrix} 
\\
&=
\begin{pmatrix} 
a_1 & a_2 & a_3 \\
b_1 -\epsilon a_1 & b_2 -\epsilon a_2& b_3 -\epsilon a_3\\
c_1 & c_2 & c_3 
\end{pmatrix} 
\begin{pmatrix} 
1 & 0 & 0 \\ 
\epsilon & 1 &0 \\ 0 & 0 & 1 
\end{pmatrix} 
\\
&=
\begin{pmatrix} 
a_1+\epsilon a_2 & a_2 & a_3 \\
b_1 -\epsilon a_1+\epsilon b_2-\epsilon^2a_2 \ 
& b_2 -\epsilon a_2 \  & b_3 -\epsilon a_3\\
c_1 +\epsilon c_2& c_2 & c_3 
\end{pmatrix} \ .
 \end{align*}
For any sufficiently small and positive $\epsilon$ 
from $\mathcal U$, we have 
$(E')^{-1} A'\geq 0$, and therefore an ESSE over $\mathcal U_+$ 
between $A'$ and 
$(E')^{-1} A'E'$. 

The next result is proved in \cite{KR0} and we take that proof.

\begin{proposition} \label{booleanprop}
Suppose $A$ is a  primitive Boolean matrix with positive 
trace. Then $A$ is SSE over $\mathcal B$ to 
$[1]$.
\end{proposition} 

\begin{proof} 
$A$ is the adjacency matrix of a directed graph. 
 Take a closed walk through the graph  which passes through
every vertex at least once. 
Suppose the walk passes through vertex $i$ $n_i$ times. 
Define a matrix $V$ which has, for each $i$, $n_i$ copies 
of row $i$ of $A$. 
(Over $\mathcal B$, a row copying is an example of a row splitting.) 
Let $U$ be the subdivision matrix such that 
$UV=A$, and set $VA=B$, SSE over $\mathcal B$ to 
$A$.  Then there is a closed walk through the graph of 
$B$ which hits every vertex exactly once. Without loss of generality, 
then, suppose $B$ is $m\times m$ and 
$B(1,1)=1$ and $P\leq B$, where $P$ is 
 the  matrix with positive entries at 
$(1,1),(1,2),(2,3), \ldots ,(m,1)$. 

Next, 
define an ESSE from $B$ to a $2m\times 2m$ matrix $C$, by 
\begin{align*} 
C = 
\begin{pmatrix} 
P & B \\ 
P & B 
\end{pmatrix} 
= &
\begin{pmatrix} 
I_m \\ I_{m}  
\end{pmatrix} 
\begin{pmatrix} 
P&B
\end{pmatrix} \\ 
B  = &
 \begin{pmatrix} 
P & B 
\end{pmatrix}
\begin{pmatrix} 
I_m \\ I_{m}
\end{pmatrix} 
\end{align*} 
An example  with $C$ $10 \times 10$ is  
\[ \label{tenbyten}
C =
\begin{pmatrix} 
P & B \\ P & B 
\end{pmatrix} 
= 
\begin{pmatrix} 
1 & 1 & 0 & 0 & 0 &   \ \ 
1  & 1 & \bu & \bu & \bu \\ 
0 & 0 & 1 & 0 & 0   &   \ \ 
\bu  & \bu & 1 & \bu & \bu \\ 
0 & 0 & 0 & 1 & 0   &   \ \ 
\bu  & \bu & \bu & 1 & \bu \\ 
0 & 0 & 0 & 0 & 1     &   \ \ 
\bu & \bu & \bu & \bu & 1 \\ 
1 & 0 & 0 & 0 & 0   &   \ \ 
1 & \bu & \bu & \bu & \bu \\ 
                                       \\ 
1 & 1 & 0 & 0 & 0   &   \ \ 
1 & 1 & \bu & \bu & \bu \\ 
0 & 0 & 1 & 0 & 0   &   \ \ 
\bu  & \bu & 1 & \bu & \bu \\ 
0 & 0 & 0 & 1 & 0    &   \ \ 
\bu & \bu & \bu & 1 & \bu \\ 
0 & 0 & 0 & 0 & 1   &   \ \ 
 \bu  & \bu & \bu & \bu & 1 \\ 
1 & 0 & 0 & 0 & 0   &   \ \ 
1  & \bu & \bu & \bu & \bu 
\end{pmatrix}    
\]
in which a bullet denotes an entry which could be 0 or 1, depending on $A$. 

Note, column 1 of $C$ is greater than or equal to column 2. 
So, we may add row 1 of $C$  to row 2 (to produce an SSE matrix). 
Now in the order $i=2,3, ..., m-1$, add row $i$ to row $i+1$. 
At the point row $i$ is added, column $i$ will be greater than or 
equal to column $i+1$, so the addition will 
give an ESSE.  After these moves, row $m$ has every entry 1. 
 In order, for $i=m, m-1, \dots , 2,1$, add column 
$i$ to every other column. At the point column $i$ is added, 
row $i$ will be all 1's, so SSE is respected.  Because every row 
has an entry 1 in one of the columns $1,2,3,\dots , m$, at the conclusion 
of this $C$ will be transformed to a matrix with every entry 1. 
Such a matrix equals $vv^{tr}$, where $v$ is a column vector with every 
entry 1, and then $v^{tr}v=[1]$. 
\end{proof} 

The next result extends a result in \cite{KR3}, with essentially
the same proof. Let $\omega(n)$ denote the maximum size of a minimal 
length closed walk which hits all vertices 
in a strongly connected directed graph with exactly $n$ vertices. 
Clearly, $\omega (n)\leq n^2$ by composition of shortest paths 
$i\to i+1$ to get $1\to 2 \to\cdots \to n \to 1$.
On the other hand, there is an example which shows that up to a 
modest multiplicative factor, in general one can't do better. 
We thank Richard Brualdi for this example. 

\begin{example}\label{brualdiexample}  For $n=2k\geq 4$, 
 consider  
the directed graph on vertices $\{1,2, \dots , 2k\}$ 
for which the set of nonzero entries of the adjacency matrix 
is the union of the following  sets:  
\begin{align*}
\{ (1,j):1\leq j \leq k\} \ , \ \ \ 
& \{ (i,k+1):1\leq j \leq k\} \ ,  \\ 
\{ (i,i+1):k\leq i <2k\}\ , \ \ \ 
&\{ (2k,1)\} \ .
\end{align*}  
For this directed graph, 
$\omega (2k)\geq (k-1)(k+2)$ and therefore 
$\omega (n)\geq \frac 14 n^2$. 
 \end{example} 

\begin{proposition} \label{primtopos} 
Suppose $\mathcal U$ is a unital nondiscrete subring 
of $\mathbb R$ and $A$ is an $n\times n$  primitive 
matrix over $U_+$ with positive trace. 
Then $A$ is SSE over $U_+$
to a positive matrix $\widetilde A$ 
which is not larger than $2n^2\times 2n^2$. 
\end{proposition} 

\begin{proof} 

The 
 matrix operations used in the proof 
of Proposition \ref{booleanprop} give   elementary SSEs 
over $\mathcal B$, 
which can be mimicked over $U_+$ by matrices with the same 
zero/nonzero pattern  
to give the SSE over $\mathcal U_+$ to a positive matrix 
$\widetilde A$.  
The matrix $\widetilde A$ is not larger than 
$2\omega(n)\times 2\omega(n)$, and 
$2\omega(n)\leq 2n^2 $. 
\end{proof} 

\begin{lemma}\label{pospos} 
Suppose $\mathcal U$ is a unital nondiscrete subring 
of $\mathbb R$, $A$ is an $n\times n$  primitive 
real matrix positive trace and 
$\widetilde A$ is a positive matrix 
SSE over  $\mathbb R$ to $A$ and 
constructed from $A$ using the algorithm of the proof 
of Proposition \ref{primtopos}. 
Suppose $\delta >0$. 
Then there is $\nu >0$ such that the following holds. 
If $A'$ is a positive matrix over $\mathcal U$ and $||A'-A||< \nu$, 
then $A'$ is SSE over $\mathcal U_+$, through positive matrices over 
 $U$, to a matrix $\widetilde{A'}$
such that  $||\widetilde{A'}-\widetilde{A}||< \delta$ .
\end{lemma} 

\begin{proof} 
The construction of 
 $\widetilde{A}$ in 
proposition \ref{primtopos} proceeds by mimicking over $\mathbb R_+$ the 
Boolean construction in the proof of 
Proposition \ref{booleanprop}. There are three steps in the construction of 
$\widetilde A$.
\begin{enumerate}
\item
Mimicking the row splittings of the Boolean image of $A$ to a Boolean
matrix $B$, there are row splittings 
over $\mathbb R_+$ of $A$ to some matrix $B^*$.
\item 
Mimicking the  splitting of the $m\times m$ Boolean $B$ to the 
$2m\times 2m$ Boolean $C$,
there is  a real 
$2m\times 2m$ matrix  $C^*$ split from $B^*$.
\item
Mimicking the Boolean SSE from $C$ to a 
matrix with no zero entry, 
there is a string of elementary matrices 
$E_t$ 
which produces 
an SSE from $C^*$ to $\widetilde A$. 
Let $\ell$ be the lag of this SSE, and use a notation 
$C^*=M_0$,
$\widetilde A=M_{\ell}$, 
$M_t=(E_t)^{-1}M_{t-1}E_t$. For 
$1\leq t \leq \ell$,  
 there is an $\epsilon_t>0$ 
which without loss of generality  we assume 
is in $\mathcal U$, such that
\begin{itemize} 
\item 
if   $1\leq t \leq m-1$,
then $(E_t)^{-1}\geq 0$ and  $M_{t-1}E_t$ is obtained 
from  $M_{t-1}$  
by  subtracting $\epsilon_t$ times column  $t+1$ 
from column $t$.
\item 
if $m\leq t\leq \ell$, then $E_t\geq 0$ and 
there are $i,j$ such that 
$(E_t)^{-1}M_{t-1}$ is obtained 
from  $M_{t-1}$, in which row $j$ has no zero entry,   
by  subtracting $\epsilon_t$ times 
 row $i$ 
from row $j$.
\end{itemize}
\end{enumerate}
First consider step 3. 
Suppose $M'_0$ is a 
 positive $2m\times 2m$ matrix. 
We recursively define 
$M'_t=(E_t)^{-1}M'_{t-1}E_t$, for $t=1,2,\dots , \ell$. 
Then define 
 $\delta'$ to be the minimum of $\delta$ 
and the smallest positive entry in a matrix of the form  
$M_t$, 
$E_t$, 
$(E_t)^{-1}M_{t-1}$ or 
$M_{t-1}E_t$, 
 $1\leq t\leq \ell$.
Suppose the following hold:
\begin{enumerate}
\item[] (i)
 $M'_0$ is  close enough to $M_0$ that 
for $1\leq t\leq \ell$,  
if $G$ is a matrix in one of the four forms above, 
and $G'$ is defined by replacing 
$M_t$ with $M'_t$ wherever it appears in the definition of $G$, 
then $||G'-G||_{\textnormal{max}}<\delta'$. 
\item[] (ii)
For $1\leq  t< m$ and $1\leq i \leq 2m$,  if 
$C(i,t)=C(i,t+1)=0$, then 
$M'_0(i,t+1)< (1/\epsilon_t)M'_0(i,t)$.
\end{enumerate} 
We claim that the matrices $M'_t$ are then positive, 
and for $1\leq t\leq \ell$, there is an ESSE over $\mathcal U_+$ 
from $M'_{t-1}$ to $M'_{t}$. 
For $1\leq t <m-1$, the conditions 
(i) and (ii) imply that 
$M'_{t-1}E_t\geq 0$ and 
$M'_t>0$, and the matrices 
 $(E_t)^{-1}$ and  $M'_{t-1}E_t$ give the required ESSE 
over $\mathcal U_+$.
For $t\geq m$, condition (i) 
implies the matrices 
$(E_t)^{-1}M'_{t-1}$ will be nonnegative, 
and again we get the ESSE over $\mathcal U_+$. 
It also follows from (i) that 
$||M_{\ell}'-\widetilde A||_{\textnormal{max}}<\delta$.

To finish, we 
first note  that by taking $\nu$ sufficiently small we can approximate 
the splittings from $A$ to $B^*$ in the first step
arbitrarily closely  by splittings 
from $A'$ to a matrix $B'^*$ through 
positive matrices over $\mathcal U$; and for the second 
step, 
given a positive matrix over $\mathcal U$ close to $B$, 
we can split $B'^*$  to a positive matrix 
$M'_0=
\left(
\begin{smallmatrix} P' & B'  \\ P' & B' 
\end{smallmatrix}
\right)$
over $\mathcal U$ close to $C^*$ and also satisfying the 
inequalities listed in condition (ii). 
\end{proof}

\section{Positive invariant tetrahedra}\label{tetrahedrasec}

Suppose $A$ and $B$ are positive  $n\times n$ matrices over 
$\mathbb R$, and are conjugate over $\mathbb R$.  
In this section we describe the 
approach introduced in \cite{KR2} 
for finding a positive path $A_t$  of conjugate matrices 
from $A$ to $B$. 
Without loss of generality, we suppose that $A$ and $B$ 
have spectral radius 1.  

First we give some terminology from \cite{KR2}. 
A {\it positive tetrahedron} is 
an $n$-tuple $\tau$ of vectors 
in an $n-1$ dimensional 
real vector space
such that the origin 
is contained in the interior of 
its  convex hull $C(\tau )$. With respect to a given 
linear endomorphism $T$, an 
{\it invariant positive tetrahedron} is a positive 
tetrahedron $\tau$ such that $C(\tau )$ is mapped into 
its interior by $T$. 

Now, suppose that $(A_t)=(G_t^{-1}AG_t)$ is a path of positive 
matrices from $A=A_0$ to $B=A_1$. We can deform such a path 
to a path of positive stochastic matrices, so we assume now 
that the $A_t$ are positive and stochastic (so, letting $r$ 
denote 
the column vector with every entry 1, we have $A_tr=r$).  
Let $e_i$ denote the row vector which is the $i$th canonical 
basis vector. Let $\ell^t$ be the stochastic left eigenvector 
of $A_t$ and set $v_i^t= e_i-\ell^t$, the projection 
of $e_i$ along $\mathbb Rr$ to the $A_t$ invariant subspace 
$W$ of vectors whose entries sum to zero. The tuple 
$(v_1^t, \dots ,v_n^t)$ has convex hull which contains the origin 
in its interior. Set $w_i^t=v_iG_t^{-1}$. 
Then $\tau_t = (w_1^t, \dots ,w_n^t)$ is a positive invariant 
tetrahedron with respect to the linear transformation 
$T:W\to W$ defined by $w\mapsto wA$. 

The path $(A_t)$ gives rise to the path $\tau_t$. 
Given $t$, the matrix $A_t$ is recovered from the action of $A$ on 
$\tau_t$, as follows. For each $i$, the vector 
$w_i^tA$ is a unique convex combination of the
$w_j^t$ (which are the extreme points of 
$C(\tau_t)$), and the coefficients for this convex combination are 
provided by row $i$ of the matrix $A_t$, as follows:  
\begin{align*} 
 w_i^tA & = (v_i^tG_t^{-1})(G_tA_tG_t^{-1}) 
 = v_i^tA_tG_t^{-1} 
= \sum_jA_{ij}^tv_j^tG_t^{-1} \\ 
& = \sum_jA_{ij}^tw_j^t\ . 
\end{align*} 

Conversely, starting from  a path of positive invariant 
tetrahedra from $\tau_0$ to $\tau_1$, we have 
a path $(A_t)$ of positive stochastic matrices, 
with the $A_t$ defined as above. 
Given $t$, there is a unique matrix $G_t$ such that 
$v_i^0G_t = v_i^t$ for $1\leq i \leq n$ and $G_tr=r$, 
and for this matrix we have $A_t=G_t^{-1} AG_t$. 

Now, to find  a path of positive invariant tetrahedra, 
one passes (for example, see Lemmas 
\ref{positivepath1}, 
\ref{positivepath}) to considering a path of matrices 
$A_t=G_t^{-1}AG_t$ with $A_tr=r$ for every $t$. 
As before, define the vectors $v_i^t$ and 
$w_i^t$ to get a path of positive 
tetrahedra $\tau_t$. Now, if $A_t$ is not 
positive, then  $\tau_t$ 
will  not be an invariant positive tetrahedron. The 
problem of deforming the path $(A_t)$ to a 
path of conjugate positive matrices is replaced  
with the problem of deforming the path $(\tau_t) $ 
to a path of invariant positive tetrahedra. So, one 
is led to study the set of connected components of 
invariant positive tetrahedra for a positive matrix. 

There is more information about these components in 
the thesis \cite{Ch} of Chuysurichay.

\section{A local connectedness condition for nilpotent matrices}
\label{connectionappendix} 

Recall,  $||C||_{\textnormal{max}}$ denotes the maximum absolute value of an 
entry of $C$. This is the norm we use through this appendix. 
For convenient reference, we repeat two definitions. 
\\ \\ 
{\bf Definition \ref{defnnbhd}.}
 For an $n\times n$ real matrix $A$,  and $\epsilon >0$,  
$\mathcal N_{\epsilon}(A)$ denotes the set of $n\times n$ 
matrices $B$ such that  $||B-A||_{\textnormal{max}}< \epsilon $. 
\\ \\ 
{\bf Definition \ref{localconnectednessdefinition}.} 
Suppose $N$ is an $n\times n$ nilpotent matrix and 
$\mathcal C$ is a 
 conjugacy class of  $n\times n$ 
nilpotent matrices. We say $\mathcal C$ is 
{\it locally connected at} $N$ if  for every $\epsilon >0$ 
there exists 
$\delta >0$ such that any two matrices in 
$\mathcal C \cap \mathcal N_{\delta}(N)$ are connected 
by a path in $\mathcal C \cap\mathcal N_{\epsilon}(N)$.
\\ \\ 
We introduce 
some  notation. 
$I_t$ is the $t\times t$ identity matrix, 
$0_t$ is the $t\times t$ zero matrix, 
and  
$e_i$ denotes a zero-one row vector whose 
only nonzero entry is in coordinate $i$. 
The direct sum of square matrices $A,B$ is 
the matrix 
$
\left(
\begin{smallmatrix} A & 0 \\ 0 & B 
\end{smallmatrix}
\right)
$. 
$J_n$ is the $n\times n$ 
Jordan block matrix: 
the $n\times n$ zero-one matrix $J$ 
 such that $J(i,j) =1$ iff 
$1\leq i < n$ and $j=i+1$. 
A matrix in Jordan form is a direct sum of Jordan blocks; 
the matrix  
has a zero 
Jordan block iff its kernel is not contained in its image. 

\begin{definition}\label{betasnotation} 
Suppose  $M$ is a  nilpotent matrix in Jordan form. 
Then 
\begin{itemize} 
\item 
$h(M)$ is the maximum size of a Jordan block summand of $M$.  
\item 
$\beta(M)$ is the number of Jordan block summands of $M$ of size 
at least $2\times 2$. 
\item 
$\beta^{\textnormal{top}}(M)$ is the number of Jordan block summands of $M$ 
of size $h(M)\times h(M)$. 
\end{itemize} 
For $N$ nilpotent in a conjugacy class $C$, we define 
$h(N)$ and $h({\mathcal C})$ to be $h(M)$ for any $M$ in Jordan 
form conjugate to $N$. Similarly for $\beta$ and 
$\beta^{\textnormal{top}}$. 
\end{definition} 

\begin{theorem} \label{matrixconnection}
Suppose $N$ is a nilpotent $n\times n$ matrix and  $\mathcal C$ is a conjugacy class of nilpotent 
$n\times n$ matrices, such that the following hold: 
\begin{enumerate} 
\item The Jordan form of a matrix in $\mathcal C$ has a 
zero block.  
\item  $h(\mathcal C)\geq h(N)$. 
\item 
$\beta^{\textnormal{top}}(\mathcal C)\geq \beta(N)$. 
\end{enumerate} 
Then $\mathcal C$ is locally connected at $N$. 
\end{theorem} 

The necessity of condition (2) in the statement is clear. 
Without condition (1), $N$ can be the limit of matrices from 
different connected components of $\mathcal C$, as happens 
with $t> 0$ in the following example: 
\[
N= 
\begin{pmatrix} 
0 & 1 & 0 & 0 \\ 0 & 0 & 0 & 0 \\ 0 & 0 & 0 & 0 \\ 0 & 0 & 0 & 0 
\end{pmatrix} \ , \qquad 
M_t=
\begin{pmatrix} 
0 & 1 & 0 & 0 \\ 0 & 0 & t & 0 \\ 0 & 0 & 0 & t \\ 0 & 0 & 0 & 0 
\end{pmatrix} \ , \qquad 
P_t=
\begin{pmatrix} 
0 & 1 & 0 & 0 \\ 0 & 0 & t & 0 \\ 0 & 0 &  0& -t \\ 0 & 0 & 0 & 0 
\end{pmatrix} \ .  
\]

\begin{question} 
Does Theorem \ref{matrixconnection} remain true if the assumption 
(3) is removed? 
\end{question} 

Our partial result and the structure of the nilpotent matrices as 
a stratified space \cite{Ve1,Ve2,Wh} suggest the answer may be yes. 

It is clear that the theorem holds for $N$ if and only if 
it holds for some matrix conjugate to $N$. For the proof, 
we will make explicit constructions using a matrix of a 
specific form. We will formulate the constructive result 
 below as a technical lemma, for which we make some 
preparations.

Theorem \ref{matrixconnection} is true if $N\in \mathcal C$ 
(Lemma \ref{closeconjugation}) and it is vacuously true 
for $\mathcal C$  if $N$ is not 
a limit of matrices from $\mathcal C$. So, we assume from here 
 that $N\notin \mathcal C$ and 
$N$ is a limit of matrices from $\mathcal C$, which implies 
for $M$ in $\mathcal C$ that 
$\textnormal{rank}(M^k)\geq \textnormal{rank}(N^k)$, $0\leq k\leq n$. 
If $N=0$, then Theorem \ref{matrixconnection} can be proved 
quickly with the argument of Step 2 of Stage 4 below. So 
we also assume from here that $N\neq 0$, which means $\beta (N)\geq 1$. 
Set $\beta =\beta (N)$ and $h=   h(\mathcal C) $.

Given $k$ with $1\leq k <h$, we define 
the $h\times h$ matrix $N_k$ by the rule 
\begin{align*} 
N_k(i,j) \ &= \ 1 \quad \textnormal{if } 1\leq i \leq k \ \textnormal{ and } \ j=i+1 \\ 
&= \ 0 \quad \textnormal{otherwise } . 
\end{align*} 
The first $k$ rows of $N_k$ equal those of $J_h$ and the remaining rows 
of $N_k$ are zero. 

We also fix a list 
$k_1, \dots , k_{\beta}$ with $k_i\geq 2$ for each $i$, such that 
$N$ is conjugate to the direct sum of 
$J_{k_1} \oplus J_{k_2} \oplus \cdots \oplus J_{k_{\beta}}$ and a 
zero matrix. Then we fix the form we will use for our matrix $N$:  
\[
N \ = \ N_{k_1} \oplus N_{k_2} \oplus \cdots \oplus N_{k_{\beta}} \oplus 0_{n-\beta h}  
\]
and define some associated subsets of $\{1,2, \dots , n\}$,  
\begin{align*} 
\mathcal J \ &=\ \{i: \textnormal{row } i \textnormal{ of } N  \textnormal{ is nonzero} \}\\  
\mathcal K \ &=\ \{i: \textnormal{row } i \textnormal{ of } N  \textnormal{ is zero}\} \\ 
\mathcal T\ &=\{1+rh: 0\leq r < \beta\} \ . 
\end{align*} 

\begin{definition}\label{Jdefinition} Given 
$M \ = \ J_h \oplus J_h \oplus \cdots \oplus J_h$,  
let $J$ be an $(n-\beta h) \times (n-\beta h )$ matrix in Jordan form 
such that 
$
\left(
\begin{smallmatrix} M & 0 \\ 0 & J 
\end{smallmatrix}
\right)
$
is in $\mathcal C$ and has row $n$ and 
column $n$ zero. 
\end{definition} 
(The condition that the $n$th row and column of $C$ can 
be chosen zero is possible by the condition 
  (1) in Theorem \ref{matrixconnection}.) 
The set $\mathcal T$ indexes the rows of $M$ 
 through the 
top rows of its first $\beta $ Jordan blocks, each of which 
is  $J_h$. These are also 
the top rows of the diagonal blocks $N_{k_i}$ in $N$. 

Given $\epsilon >0$, define 
$M(\epsilon )$ to be the $n\times n$ 
matrix such that 
\begin{align*} 
M(\epsilon )(i,j) \ &=  N(i,j)\quad \ \ \  \textnormal{if } i\in \mathcal J \\
\ &= \ \epsilon M(i,j)\quad \textnormal{if } i\notin \mathcal J \ . 
\end{align*} 
Given $\delta > 0$, $\mathcal M_{\delta}$ denotes the set of 
$n\times n$ matrices 
$C$ such that the following hold: 
\begin{enumerate}
\item[] (i) 
$C$ is conjugate to $M$.
\item[] (ii) 
$||C-N||_{\textnormal{max}}< \delta$. 
\end{enumerate} 
We say $C\in \mathcal M^0_{\delta}$ if
$C\in \mathcal M_{\delta}$ and 
 in addition  
\begin{enumerate}  
\item[] (iii)  
If $i\in \mathcal J$, then row $i$ of $C$ equals row $i$ of $N$. 
\end{enumerate} 
We will use $\mathcal M$ and $\mathcal M^0$ to denote the union 
over $\delta >0$ of 
$ \mathcal M_{\delta}$ and 
$ \mathcal M^0_{\delta}$ 
 (respectively).

\begin{example} \label{blockexample} 
For the matrix arguments to follow, it may be helpful to 
have the block structure of an example in view. 
For this example, we take 
\begin{align*} 
h \ & =  5 \ , \ \beta =2 \ , \ n=12 \ , k_1= 2 \ , \ k_2=1\ , \\ 
M \ &=\ J_5 \oplus J_5\oplus 0_2 \ , \\ 
N\ &=\ N_2 \oplus N_1 \oplus 0_2 \ .
\end{align*} 
Now a matrix $C$ in $\mathcal M^0_{\delta }$ has a block structure: 
\begin{align*}\label{samplec} 
\setcounter{MaxMatrixCols}{14} 
C\ &= 
\begin{pmatrix} 
0&\mathbf 1 &0&0&0&& 0&0&0&0&0&& 0&0 \\
0&0&\mathbf 1 &0&0&& 0&0&0&0&0&& 0&0 \\ 
\bu &\bu &\bu &\bu &\bu & &\bu &\bu &\bu &\bu &\bu & &\bu &\bu \\ 
\bu &\bu &\bu &\bu &\bu & &\bu &\bu &\bu &\bu &\bu & &\bu &\bu \\ 
\bu &\bu &\bu &\bu &\bu & &\bu &\bu &\bu &\bu &\bu & &\bu &\bu \\ 
 & & & & &  & & & & &  & & & \\
0&0&0& 0&0& &0&\mathbf 1 &0&0&0& &0&0 \\
\bu &\bu &\bu &\bu &\bu & &\bu &\bu &\bu &\bu &\bu & &\bu &\bu \\ 
\bu &\bu &\bu &\bu &\bu & &\bu &\bu &\bu &\bu &\bu & &\bu &\bu \\ 
\bu &\bu &\bu &\bu &\bu & &\bu &\bu &\bu &\bu &\bu & &\bu &\bu \\ 
\bu &\bu &\bu &\bu &\bu & &\bu &\bu &\bu &\bu &\bu & &\bu &\bu \\ 
 & & & & &  & & & & &  & & & \\
\bu &\bu &\bu &\bu &\bu & &\bu &\bu &\bu &\bu &\bu & &\bu &\bu \\ 
\bu &\bu &\bu &\bu &\bu & &\bu &\bu &\bu &\bu &\bu & &\bu &\bu 
\end{pmatrix} 
\end{align*}
in which each $\bu$ has absolute value less than $\delta$.
If $C$ is only in $\mathcal M_{\delta}$, then the entries marked 
$0$ and $1$ above are only approximated to within $\delta$. 
Continuing the example, we have 
\begin{align*}
\setcounter{MaxMatrixCols}{14} 
M(\epsilon ) \ &= 
\begin{pmatrix} 
0&\mathbf 1 &0&0&0&& 0&0&0&0&0&& 0&0 \\
0&0&\mathbf 1 &0&0&& 0&0&0&0&0&& 0&0 \\ 
0&0&0&\boldsymbol\epsilon   &0&& 0&0&0&0&0&& 0&0 \\ 
0&0&0&0&\boldsymbol\epsilon   && 0&0&0&0&0&& 0&0 \\ 
0&0&0&0&0&& 0&0&0&0&0&& 0&0 \\ 
 & & & & & & & & & & &  & & \\
0&0&0&0&0&  & 0&\mathbf 1 &0&0&0            & &0&0 \\
0&0&0&0&0&  & 0&0&\boldsymbol\epsilon   &0&0& &0&0 \\ 
0&0&0&0&0&  & 0&0&0&\boldsymbol\epsilon   &0& &0&0 \\ 
0&0&0&0&0&  & 0&0&0&0&\boldsymbol\epsilon   & &0&0 \\ 
0&0&0&0&0&  & 0&0&0&0&0                     & &0&0 \\ 
 & & & & & & & & & & &  & & \\
0&0&0&0&0&& 0&0&0&0&0&& 0&0 \\ 
0&0&0&0&0&& 0&0&0&0&0&& 0&0 
\end{pmatrix} 
\end{align*} 
\end{example} 
The example is somewhat special in that the summand 
$0_2$ of $M$ could have been much more complicated. 
However, it turns out that this possible complication 
doesn't matter in the proof below until the last stage, 
where it is not a big problem. 

We are finally ready to state the technical lemma, 
from which Theorem \ref{matrixconnection} follows immediately. 

\begin{lemma} \label{blockconnectionlemma} 
Given $0< \gamma <1/49$, there exists $\delta >0$ such that 
for all $C$ in $\mathcal M_{\delta}$ and all $\epsilon $ 
such that $0<\epsilon < \delta$, there is a path in 
$\mathcal M_{\gamma}$ from $C$ to 
 $M(\epsilon )$. 
\end{lemma} 

\begin{proof} 

The path will be a concatenation of paths constructed in four  stages. 
Combining the estimates, given $0< \gamma < 1/49$, the lemma 
will hold for 
\[
\delta = \frac{ n^{1/2^n} \gamma^{2^n}}{4(2n+2)^n}\ . 
\]
We do not claim this estimate or the requirement $\gamma < 1/49$ are sharp.  

Below,  subscripted matrices $C$ in different 
stages are dummy variables not related to  subscripted matrices $C$
in other stages. 

{\bf Stage 1.} 
 Given $C$ in $\mathcal M_{\mu}$, we produce 
a path in $\mathcal M_{\kappa\mu}$ to a matrix $C_S$
 in $\mathcal M^0_{\kappa\mu}$, where 
$\kappa =[2(n+1)]^S$ and $S=\# \mathcal J<n$. 
For this stage,  
let  $i_1< i_2 < \dots < i_S$  denote the elements of $\mathcal J$. 
Set $C_0=C$. 
For $1\leq s\leq S$, given $C_{s-1}$, 
we will 
define inductively 
$C_s$ and 
a path $(\widetilde{C_t})_{0\leq t \leq 1}$  
from $C_{s-1}=\widetilde{C_0}$ to $C_{s}=\widetilde{C_1}$
such that there is  a $\kappa >0$, independent of 
$C$, such that the following hold whenever  $0<\mu < 1/2$.

\begin{enumerate} 
\item 
For $i=i_s$,  $\ e_i\widetilde{C_1} =e_{i+1}$ . 
\item 
For $i=i_t < i_s$, 
  $\ e_i\widetilde{C_1} =e_{i+1}$ . 
\item 
If $C_{s-1}\in \mathcal M_{\mu}$, then 
$\widetilde{C_t}\in \mathcal M_{\kappa \mu}$, $\ 0\leq t \leq 1$.  
\end{enumerate} 
(Then the rows $i_1,i_2,\dots , i_s$ of $C_s$ will equal 
the corresponding rows in  $N$.) 
The path will be a (renormalized) 
concatenation of two paths. The first 
path $(C'_t)_{0\leq t \leq 1}$  
moves the $(i_s,i_s+1)$ entry of $C_{s-1}$ to $1= N(i_s,i_s+1)$. 
Let $\eta= C_{s-1}(i_s,i_s+1)$. 
For $0\leq t\leq 1$, define 
$C'_t= D_tC_{s-1}D_t^{-1}$, where $D_t$ is diagonal and equal to $I$ except at 
$D_t(i_s + 1, i_s +1) = \eta^t$. (Note, given 
 $\mu < 1$, we have $\eta >0$.) 
For each $t$, the rows $i_1,i_2, ... ,i_{s-1}$ of $C'_t$ equal 
those of $N$, by the induction hypothesis. 
At the two types of entry where $C'_t(i,j)$ might not 
equal $_{s-1}(i,j)$, we have the following.   
\begin{itemize} 
\item 
If $j\neq i_s+1$ and $i=i_s+1$, 
then $C'_t(i,j)=\eta^tC_{s-1}(i,j)$ . 
\item 
If $i\neq i_s+1$ and $j=i_s+1$, then 
$C'(i,j)=\eta^{-t}C_{s-1}(i,j) $ .
\end{itemize} 
In both of these cases, if $(i,j)\neq (i_s,i_s+1)$, then 
$N(i,j)=0$ and $|C_{s-1}(i,j)|<\mu$. 
Also, because $1/2< \eta < 3/2$, 
for $|t|\leq 1$ we have $\eta^t < 2$, 
and consequently in both cases  
$|N(i,j)-C'_t(i,j)|= |C'_t(i,j)|<2\mu$. 
Lastly, as $t$ moves from 0 to 1 , 
$C'_t(i_s,i_s+1)$ moves monotonically from $\eta $ to 1 .  
We conclude $(C'_t)_{0\leq t \leq 1}$ is a path in 
$\mathcal M_{2\mu}$.

We now replace $C_{s-1}$ with $C'_1$, and for notational simplicity denote it 
as $C$.  
For $0\leq t\leq 1$, define an $n\times n$ matrix 
$V_t$ by setting 
\begin{align*} 
V_t(i_s+1,i) & = -tC(i_s,i) \quad \qquad \textnormal{if } i\neq i_s+1 \\ 
V_t(i,j) &= I(i,j) \ \  \qquad  \qquad \textnormal{otherwise }. 
\end{align*} 
Define $\widetilde{C_t}=V_t^{-1}CV_t$. 
Then $V_1$ acts to add multiples of column $i_{s}+1 $ of $C$ to other columns so that 
row $i_s$ of $CV_1$ equals $e_{i_s+1}$. 
The rows of $\widetilde{C_t}$ and $CV_t$ must be equal, except for row $i_s+1$. It 
follows that (1) and (2) hold. 
Also, 
\begin{align*} 
||N-V_t^{-1}CV_t || & \leq
||N-C || + 
||C-CV_t || + 
||CV_t -V_t^{-1}CV_t || \\ 
&\leq \mu + \mu^2 + (n-1)\mu < (n+1)\mu \ . 
\end{align*}  
So, combining this path together with the
diagonal conjugation path, property (3) holds with $\kappa = 2(n+1)$. 
We now pass from $C_0$ to $C_S$. This completes the proof 
for Stage 1.

{\bf Stage 2.} Given $C_0$ in 
$\mathcal M_{\mu}^0$, with $\mu>0$,
we produce 
a path in $\mathcal M_{\mu}^0$  to a matrix $M$  satisfying the following 
condition: 
\begin{enumerate}
\item[]
(iv) If $ i\in \mathcal T$, then  
$e_iM^{h-1}\neq 0$. 
\end{enumerate} 

So, suppose $C_0\in \mathcal M^0_{\mu}$, and  
let $\mathcal U=\{i: e_iC_0^{h-1}\neq 0\}$.
Suppose $\mathcal T \not\subset\mathcal U $. 
We have $\# \mathcal U \geq \# \mathcal T $, because 
$\beta^{\textnormal{top}}(C_0)\geq \beta (N) = \beta= \# \mathcal T$.
Therefore, we can choose an injection 
$\mathcal T \setminus (\mathcal T \cap \mathcal U)\to 
\mathcal U \setminus (\mathcal T \cap \mathcal U)$, $i\mapsto \xi (i)$. 
Let  $\{i_1, \dots , i_R\}$ now denote the set 
$\mathcal T \setminus (\mathcal T \cap \mathcal U)$.
We will define matrices $C_1, \dots , C_R$ inductively. 
For $1\leq r \leq R$, given $C_{r-1}$ 
we will
define  a path $\widetilde{C_t}$, $0\leq t \leq \nu$, such that 
$\widetilde{C_0}=C_{r-1}$ and 
such that the following hold for each $t$. 
\begin{enumerate} 
\item For $t\neq 0$ and $i=i_r$,  $e_i(\widetilde{C_t})^{h-1} \neq 0$ . 
\item 
If $i \in \mathcal T$ 
and $e_i(C_{r-1})^{h-1} \neq 0$, then 
$e_i(\widetilde{C_t})^{h-1} \neq 0$. 
\item 
If $i\in \mathcal J$, then $e_i\widetilde{C_t}=e_{i+1}$
\item 
$||\widetilde{C_t} -C_{r-1}||_{\textnormal{max}} < \mu - 
||C_{r-1}-N||_{\textnormal{max}}$ . 
\end{enumerate} 
(The conditions (3) and (4) keep the path in 
$\mathcal M^0_{\mu}$.)  
We then define the matrix $M$ of (iv) to be $C_R$. 

So, to define the path, suppose we are given $C_{r-1}$. 
For notational simplicity, we let $C$ denote $C_{r-1}$;  
$j$ denote $\xi(i_r)$; $k$ be the $k_i$ such that row 
$i_r$ is the top row of $N_{k_i}$; and let $i_r$ be $1$. 
 Because $e_i$ is 
in the image of $C$ if $1< i\leq k$, we have $j\notin \{ 1, \dots , k\}$. 
Given a scalar $t$, let $V_t$ be the $n\times n$ matrix such that 
\begin{align*} 
\textnormal{row } i \textnormal{ of } V_t \ &= \ e_i + te_jC^{i-1}\ , \quad 
1\leq i \leq k+1 \ , \\ 
&= \  e_i\ , \qquad \qquad \quad \quad \textnormal{otherwise} \ . 
\end{align*} 
We keep $t$ small enough that $V_t$ is invertible, 
and define 
\[
\widetilde C_t = V_tCV_t^{-1} \ . 
\]
Now we verify the induction conditions. 
The proof of (1) is a computation: 
\begin{align*} 
e_1(V_tCV_t^{-1})^{h-1} \ &= \ 
e_1V_t(C^{h-1}V_t^{-1}) 
= \ (e_1 + te_j)( C^{h-1}V_t^{-1}) \\ 
&= \ te_jC^{h-1}V_t^{-1} \neq 0\ . 
\end{align*} 
The proof for (2) is similar. If $i\in \mathcal T$ and 
$e_iC^{h-1}\neq 0$, then 
\begin{align*}  
e_i(V_tCV_t^{-1})^{h-1} \ &= \ 
e_iV_t(C^{h-1}V_t^{-1})  
= \ e_i(C^{h-1}V_t^{-1})  \neq 0\ . 
\end{align*} 
For (3), given $i\in \mathcal J$, we must show 
that $e_i\widetilde{C_t} = e_{i+1}$. 
We do this for two cases. 
If $1\leq i \leq k$, then 
\begin{align*} 
e_i(V_tCV_t^{-1}) \ &= \ 
 (e_iV_t)(CV_t^{-1}) \ = \ 
( e_i + te_jC^{i-1})(CV_t^{-1}) \\ 
\ &= \ 
 (e_{i+1} + te_jC^{i})V_t^{-1} 
= \ e_{i+1} \ . 
\end{align*} 
If $i\in \mathcal J\setminus \{1, \dots , k\}$, 
then $\{i,i+1\}\cap  \{ 1, \dots , k\}=\emptyset$; 
so, if  
$e_iC = e_{i+1}$, then  
\[
e_iV_tCV_t^{-1} = 
e_iCV_t^{-1} = 
e_{i+1}V_t^{-1} = 
e_{i+1} \ . 
\] 
It is clear that (4) holds if $\nu$ is sufficiently small. 
This completes the proof for Stage 2. 

{\bf Stage 3.} We begin with
 $C$ in $\mathcal M^0_{\mu}$, with $0< \mu < 1/49$, 
with $C$ (from Stage 2) the matrix $M$ 
satisfying condition (iv) of Stage 2.  
We produce a path in $\mathcal M^0_{\nu}$ 
from $C$ to a matrix $C_G$ whose first $\beta h$ rows have the 
form of $M(\epsilon )$, but with the $\epsilon $'s 
replaced perhaps by various positive numbers, 
with 
$\nu = (2n)^2\mu^{1/2^n}$.  Then, given  $0<\epsilon <1$, 
we will have 
$C_G\in \mathcal M^0_{\epsilon}$ if 
\[
\mu\  <\  \frac{n^{(1/2^n)} \epsilon^{2^n}}{4} \quad . 
\]

For the proof, we will inductively  produce a finite sequence of matrices 
$C_g$ and index sets $\mathcal S_g$,  
$0\leq g \leq G$, 
with $G <  \beta h$,    
beginning with $C_0=C$. 
Property (iv) from Stage 2 will 
be preserved at every step, because 
successive matrices 
will be  conjugate  by a conjugacy respecting the subspaces 
$\mathbb Re_i, i\in \mathcal T$. 

Given $M$ in $\mathcal M^{0}$, define 
the set $\mathcal S(M)$ to be the largest subset 
$\mathcal S$ 
of  $\{1,2,\dots , \beta h\}$ satisfying the following conditions: 
\begin{enumerate} 
\item[] (A1) 
If $i\in \mathcal S$, then 
\begin{align*} 
\textnormal{row } i \textnormal{ of } 
M  \ &= \ e_{i+1} \qquad \qquad   \qquad \qquad \qquad  \ \
 \quad \textnormal{ if } i\in \mathcal J \\ 
&=\ 0 \qquad \qquad \qquad \qquad \qquad  \quad  \quad 
\quad \textnormal{ if } h \textnormal{ divides } i \\ 
&= \ \textnormal{a positive multiple of } e_{i+1} \ , \quad \  
 \textnormal{ otherwise .} 
\end{align*}
\item[] (A2)  
If $0\leq r < \beta$ and $1 \leq j\leq h$ and 
$rh+ j\in \mathcal S$, \newline 
then 
$\mathcal S$ contains 
 $\{rh+ i: 1\leq i \leq j\}$. 
\end{enumerate} 
Note, $\mathcal S(M)$ contains $\mathcal J$. 
Also define 
\begin{align*} 
\mathcal R(M)\ &=\ \{ i: \ i\leq \beta h, \  i\notin \mathcal S (M),\ \  
i-1\in \mathcal S(M) \} \\ 
\mathcal P(M) \ &=\  \{(i,j): \ i\in \mathcal R(M) ,\  \  
j\notin \mathcal S(M)
\cup \mathcal R(M), \ \  M(i,j) \neq 0\}\ . 
\end{align*}   
Let $\mathcal S_g=\mathcal S(C_g)$, 
$\mathcal R_g=\mathcal R(C_g)$, 
$\mathcal P_g=\mathcal P(C_g)$. 
Continuing Example \ref{blockexample}, 
with $\mathcal S_{g}=\{1,2,3,6,7\}$ we would see 
the matrix   
$C_{g}$ having the following form: 
\begin{align*}
\setcounter{MaxMatrixCols}{14} 
\begin{pmatrix} 
0&\mathbf 1 &0&0&0&& 0&0&0&0&0&& 0&0 \\
0&0&\mathbf 1 &0&0&& 0&0&0&0&0&& 0&0 \\ 
0&0&0&\boldsymbol\epsilon_1   &0&& 0&0&0&0&0&& 0&0 \\ 
\bu &\bu &\bu &\bu &\bu   && \bu &\bu &\bu &\bu &\bu && \bu &\bu  \\ 
\bu &\bu &\bu &\bu &\bu && \bu &\bu &\bu &\bu &\bu && \bu &\bu  \\ 
 & & & & & & & & & & &  & & \\
0&0&0&0&0&  & 0&\mathbf 1 &0&0&0            & &0&0 \\
0&0&0&0&0&  & 0&0&\boldsymbol\epsilon_2   &0&0& &0&0 \\ 
\bu &\bu &\bu &\bu &\bu &  & \bu &\bu &\bu &\bu   &\bu & &\bu &\bu  \\ 
\bu &\bu &\bu &\bu &\bu &  & \bu &\bu &\bu &\bu &\bu   & &\bu &\bu  \\ 
\bu &\bu &\bu &\bu &\bu &  & \bu &\bu &\bu &\bu &\bu                      & &\bu &\bu  \\ 
 & & & & & & & & & & &  & & \\
\bu &\bu &\bu &\bu &\bu && \bu &\bu &\bu &\bu &\bu && \bu &\bu  \\ 
\bu &\bu &\bu &\bu &\bu && \bu &\bu &\bu &\bu &\bu && \bu &\bu  
\end{pmatrix} 
\end{align*} 
In this example, 
$\mathcal R_{g}=\{4,8\}$; 
$(i,j)\in \mathcal P_{g}$ iff 
$i\in \{4,8\}$ and 
$j\in \{5,9,10,11,12\}$.

We will arrange by induction that the 
following hold for $g\geq 1$.  
\begin{enumerate} 
\item[] (B1) 
If $\# \mathcal S_{g-1}\neq \beta h$, then $\mathcal S_{g-1}$
is properly contained in  $\mathcal S_{g}$. 
\item[] (B2)  
If  $C_{g-1}\in \mathcal M^0_{\mu}$, 
with $\mu <1$, 
then there is a path 
 in $\mathcal M^0_{2n\sqrt{ \mu}}$
from $C_{g-1}$ to 
$C_{g}$.  
\end{enumerate} 
Given all this, we  define 
 $G$ to be the index $g$ at which $\mathcal S_g=
\{1,2,\dots , \beta h\}$.  

Now, suppose we are given $C_{g-1}$
 and $\mathcal S_{g-1}$ with 
$\#\mathcal S_{g-1} < \beta h$ 
(i.e.,  $\mathcal R_{g-1}$ is nonempty). 
We will show  $\mathcal P_{g-1}$ is 
nonempty. 
Pick $i\in \mathcal R_{g-1}$. 

If $i$ is divisible by $h$, then 
(by property A2) let $t$ in 
$\mathcal T$ be such that row $i$ of $C_{g-1}$ 
(which is $e_iC_{g-1}$) is a positive 
multiple of $e_tC^h$, and therefore is zero. 
Since $i-1\in \mathcal S_{g-1}$, it follows 
that $i\in \mathcal S_{g-1}$, a contradiction. 
So $i$ is not divisible by $h$. Let $k$ be 
the positive integer  in $[1,h-1]$ 
such that $e_iC_{g-1} = e_tC_{g-1}^k $ 

Now suppose  $C_{g-1}(i,j)\neq 0$ implies 
$j\in \mathcal S_{g-1} \cup 
\mathcal R_{g-1}$. 
Then $e_iC_{g-1}$ is 
a linear combination of the vectors 
$e_{\tau}C^j$ such that $\tau\in \mathcal T$ 
and $0\leq k< h$. This is a contradiction, 
because the set 
 $\{e_tC_{g-1}^j: t\in \mathcal T, 0\leq j < h\}$ 
is linearly independent, by property (iv).
 Therefore there is a $j$ 
such that $(i,j)\in \mathcal P_{g-1}$.

From here, we will handle the inductive transition from $g-1$ to $g$
in three steps,  given $C_{g-1}$ and $\mathcal S_{g-1}$ with 
$\#\mathcal S_{g-1} < \beta h$.  By a {\it signed transposition 
matrix for indices i,j } we mean a matrix $Q$ which is equal 
to the permutation matrix $P$ for the transposition 
exchanging 
$i$ and $j$, except that one of the entries 
$Q(i,j)$ or $Q(j,i)$ is $-1$. 

STEP 1. 
Given $C_{g-1}\in \mathcal M^0_{\mu}$, 
we produce 
 an index $i$ in $\mathcal R_{g-1}$, and a matrix $Q$ 
which is either $I$ or is 
 a signed transposition matrix 
for indices outside $\mathcal R_{g-1}  \cup \mathcal S_{g-1}$, 
 such that  the following hold
for the matrix $\overline C= Q^{-1}C_{g-1}Q$: 
\begin{enumerate} 
\item[] (D1)
$(i,i+1)\in \mathcal P(\overline C)$ .  
\item[] (D2)
$\overline C(i,i+1)= \max \{|\overline C(i',j')|: (i',j')
\in \mathcal P(\overline C)\}$   . 
\item[] (D3) 
There is  a path in  
$\mathcal M^0_{\sqrt{\mu}}$ from
$C_{g-1}$ to $\overline C$ .
\item[] (D4) 
$\mathcal S(\overline C) = \mathcal S_{g-1}$, and 
$||N-\overline C||_{\textnormal{max}}= 
||N-C_{g-1}||_{\textnormal{max}} < \mu $ .  
\end{enumerate}


STEP 2. 
For the matrix $\overline C$ produced in Step 1, 
defining $ \alpha=||N-\overline C||_{\textnormal{max}}$,  
we produce a matrix $C'$ and $(i,i+1)\in \mathcal P_{g-1}$ 
 such that the following 
hold. 
\begin{enumerate} 
\item[] (C1) 
$||N-C'||_{\textnormal{max}}= |C'(i,i+1)| =
\sqrt{\alpha } $ .
\item[] (C2)
If $r\in \mathcal S_{g-1}$, then $|C'(i,r)|<\mu$ . 
\item[] (C3) 
There is a path in $\mathcal M^0_{\sqrt{\mu}}$ from $\overline C$ to $C'$ . 
\end{enumerate}


STEP 3. Given $C'$ from Step 2, we produce a path in 
$\mathcal M^0_{2n\sqrt{\mu}}$ from $C'$ to the desired matrix $C_g$.

PROOF  FOR STEP 1. 
\\ 
Choose $(i,j)$ from the nonempty set
$\mathcal P_{g-1}$ such that  
\[|C_{g-1}(i,j)|= \max \{|C_{g-1}(i',j')|: (i',j')\in \mathcal P_{g-1}\}\ .
\]     
There are two cases. 

CASE 1: $j\neq i+1$.  Both $j$ and $i+1$ are outside 
 $\mathcal R_{g-1} \cup \mathcal S_{g-1}$. 
Let $Q$ denote the $n\times n$ signed transposition 
matrix for indices 
$i+1$ and $j$ such that 
\begin{align*} 
Q(j,i+1) &  = -1 \ , \quad 
\textnormal{ if } C_{g-1}(i,j) < 0 \\ 
Q(i+1,j) &  = -1 \ , \quad 
\textnormal{ if } C_{g-1}(i,j) > 0 \ . 
\end{align*} 
Set $\overline C=Q^{-1}C_{g-1}Q$. 
We have $Q\in \textnormal{SL}(n,\mathbb R)$, and for a path 
from $C_{g-1}$ to $\overline C$ we can use 
$U_t^{-1}C_{g-1}U_t$, $0\leq t\leq 1$, with $(U_t)$ a path 
from $I$ to $Q$. 
E.g., for $Q=
\begin{pmatrix} 
0 & 1 \\ -1 & 0 
\end{pmatrix}$, we may use (in the principal submatrix on 
coordinates $\{i+1,j\}$) 
\begin{align*} 
U_t &= 
\begin{pmatrix} 
1 & 0 
\\ 
-t & 1
\end{pmatrix}
\begin{pmatrix} 
1 & t 
\\ 
0 & 1
\end{pmatrix}
\begin{pmatrix} 
1 & 0 
\\ 
-t & 1
\end{pmatrix} \\
&= 
\begin{pmatrix} 
1-t^2 & t 
\\ 
-2t +t^3 & 1-t^2
\end{pmatrix} \ . 
\end{align*} 
Then for $0\leq t\leq 1$, 
\[
||U_t^{-1}C_{g-1}U_t||_{\textnormal{max}} \ 
\leq \ 
4
||U_t||_{\textnormal{max}} 
||C_{g-1}||_{\textnormal{max}} \ 
 < \ 
5 ||C_{g-1}||_{\textnormal{max}} 
\ . 
\]

CASE 2: $j=i+1$. If $C_{g-1}(i,i+1)>0$, then set $Q=I$ and 
$\overline C=C_{g-1}$. If $C_{g-1}(i,i+1)<0$, then 
let $W$ be the matrix 
in $\textnormal{SL}(n,\mathbb R)$ 
obtained from $I_n$ by 
multiplying rows $i+1$ and $n$ by $-1$, and 
define 
$\overline C=W^{-1}C_{g-1}W $. 
As in Case 1,  we may produce a path from 
$C_{g-1}$ to $\overline C$ by conjugating with a path 
$(W_t)$, $0\leq t\leq 1$, from $I$ to $W$.
One such path is given (on the principal submatrix 
on indices $\{i+1,n\}$) by 
\begin{align*} 
W_t \ 
&= 
\begin{pmatrix} 
1&0
\\ 
-2t&1
\end{pmatrix}
\begin{pmatrix} 
1&t
\\ 
0&1
\end{pmatrix}
\begin{pmatrix} 
1&0
\\ 
-2t&1
\end{pmatrix}
\begin{pmatrix} 
1&t
\\ 
0&1
\end{pmatrix}
\\
&=
\begin{pmatrix} 
1-2t^2 & 2t-2t^3 
\\ 
-4t +4t^3 & 1 - 6t^2 + 4t^4 
\end{pmatrix}
\end{align*} 
Then 
\[
||W_t^{-1}C_{g-1}W_t||_{\textnormal{max}} \ 
\leq \ 
4
||W_t||_{\textnormal{max}} 
||C_{g-1}||_{\textnormal{max}} \ 
 < \ 
7 ||C_{g-1}||_{\textnormal{max}} 
\ . 
\] 
In both cases, (D1) and (D2) hold,  
and the path from $C_{g-1}$ to $\overline C$ 
is contained in 
$\mathcal M^0_{7\mu}$, and consequently in    
$\mathcal M^0_{\sqrt{\mu}}$, since     
$\mu < 1/49$. 
This completes the proof for Step 1. 

PROOF FOR STEP 2. 

Given $\overline C$ and $(i,i+1)\in \mathcal P(\overline C)$ 
from Step 1, define 
$\beta = 
 \max \{|\overline C(i,j)|: (i,j)\in \mathcal P(\overline C)\}$ . 
Define a path
$C'_t=D_t\overline{C}D_t^{-1}, 1\leq t \leq \sqrt{\alpha}/\beta$, 
 from $\overline C=C'_1$ to $C'=
 C'_{\sqrt{\alpha}/\beta}$,   
in which $D_t$ is the diagonal matrix defined by 
\begin{align*} 
D_t(i',i') \  & = \ t \ , \quad 
\textnormal{if } i'\in \mathcal S_{g-1}\cup \mathcal R_{g-1} \\ 
& = \  1 \ , \quad 
\textnormal{otherwise .} 
\end{align*} 
We have $|C_{g-1}(i,r)|<\mu$ 
for all $r$, 
because $i\notin \mathcal J$.  
 Therefore, $|\overline C(i,r)|<\mu$ 
for all $r$.  
So, if 
 $r\in \mathcal S_{g-1}\cup \mathcal R(\overline C)$, then 
$|C'_t(i,r)|\leq |\overline C(i,r)|< \mu$; this establishes 
(C2). 
If $i'\in 
\mathcal S_{g-1}$ and 
$j'\notin 
\mathcal S_{g-1}\cup \mathcal R(\overline C)$,
then  $C_{g-1}(i',j') =0$, and $\overline C(i',j')=0$. 
Therefore,  
$|C'_t(i',j')|> |\overline C(i',j')|$ is possible 
only if 
$i'\in \mathcal R_{g-1}$ 
and $j'\notin \mathcal S_{g-1}\cup \mathcal R(\overline C)$. 
It follows then from (D2) that 
$||C'_t-N||\leq \sqrt{\alpha} $ for all $t$, 
with (C1) holding for $C'=C'_1$. Because $\alpha < \mu$, 
$(C'_t)$ is a path in $\mathcal M^0_{\sqrt{\mu}}$, and 
then (C3) follows from (D3). 
This  
 finishes the argument for Step 2.  

PROOF FOR STEP 3. 

For a lighter notation, in this step we will write $C$ in place 
of $C'$ for the matrix satisfying (C1)-(C3) for a given 
$(i,i+1)$ from $ \mathcal P( C')$.  
For $0\leq t\leq 1$, we define an $n\times n$ matrix 
$V_t$ equal to $I$ outside row $i+1$. In that row, 
we define $V_t(i+1,i+1)=1$ and  
\[
V_t(i+1,r)  = -tC(i,r)/C(i,i+1) \ , \quad  
\textnormal{ if } r \neq i+1\ . 
\] 
Define the path $\widetilde{C_t}=V_t^{-1}CV_t$, $0\leq t \leq 1$. 
Then $C=\widetilde{C_0}$,   and we define 
$C_g=\widetilde{C_1}$. 

First we will check that $C_g$ satisfies condition (B1). 
The matrix  $V_1$ acts to add multiples of column $i+1 $ of $C $ to other columns so that 
row $i$ 
of $CV_1$ 
has exactly one 
nonzero entry, which 
 is at position $(i,i+1)$. 
Suppose 
 $i'\in \mathcal S_{g-1}$. If  $C(i',i+1) \neq 0$, 
then $i+1=i'+1$, which forces $i=i'\in \mathcal S_{g-1}$, 
contradicting $(i,i+1)\in \mathcal P$. Therefore  $C(i',i+1) = 0$. 
Therefore  
row $i'$ of $CV_t$ equals row $i'$ of $C$.  
 The matrix 
$V_t^{-1}CV_t$ can differ from $CV_t$ only in row $i+1$, which 
 is not in 
$\mathcal S_{g-1}$, since $(i,i+1)\in \mathcal P$.   Consequently,
 $ \mathcal S_g$ contains $ \mathcal S_{g-1}$ and also $\{i\}$.   Since 
 $i\notin \mathcal S_{g-1}$, 
the condition (B1) is satisfied.

Now we turn to (B2). We have 
\begin{align} \label{threeterms}
&\ ||N-V_t^{-1}CV_t ||_{\textnormal{max}}  \\ 
\notag 
\leq &\ 
||N-C ||_{\textnormal{max}} + 
||C-CV_t ||_{\textnormal{max}} + 
||CV_t -V_t^{-1}CV_t ||_{\textnormal{max}} 
\end{align}  
and we will bound the three terms on the 
right. 

We have 
$||N-C ||_{\textnormal{max}} < \sqrt{\mu}$,  
since $C\in \mathcal M^0_{\sqrt{\mu}}$ .

Because 
 $C-CV_t = C(I-V_t)$, 
 column $r$ of $C-CV_t$ is zero if $r=i+1$ 
and otherwise equals column $i+1$ of $C$ (whose entries 
are smaller in absolute value than $\sqrt{\mu}$, since $i\notin 
\mathcal S_{g-1}\supset \mathcal J$) multiplied by 
$-tC(i,r)/C(i,i+1)$ (which by (C1) has absolute value at most 1). Therefore 
$||C-CV_t ||_{\textnormal{max}} <  \sqrt{\mu}$ (and then 
 $||N-CV_t ||_{\textnormal{max}} <  2\sqrt{\mu}$).

The last term in (\ref{threeterms}) is the maximum 
 over $(i',j')$ of the absolute value of 
\[
(CV_t -V_t^{-1}CV_t)(i',j') \ 
= \ 
\Big((I -V_t^{-1})(CV_t)\Big)(i',j') \ . 
\]
This quantity is zero if $i'\neq i+1$, 
since row $i+1$ is the only nonzero row of 
$ (I -V_t^{-1})$. 
For $i'=i+1$, we have 
\begin{align*} 
\Big((I -V_t^{-1})(CV_t)\Big)(i+1,j') 
\ &= \ 
\sum_r
(I -V_t^{-1})(i+1,r)(CV_t)(r,j') \\
& = 
\sum_{ r:\,  r\neq i+1}
\Bigg( \frac{-tC(i,r)}{C(i,i+1)}\Bigg)
(CV_t)(r,j') \ . 
\end{align*}
Now we bound the terms in the last sum by two cases. 
\\
CASE 1: $r\in \mathcal J$. \\
Then $(CV_t)(r,j')=V_t(r+1,j')$  and also $r\neq i$. 
 If $j'\neq r+1$, then 
$(CV_t)(r,j')=V_t(r+1,j')=0$.
If $j'=r+1$, then 
 $(CV_t)(r,r+1)=V_t(r+1,r+1)=
1$. So, 
\begin{align*}
\Bigg| \frac{-tC(i,r)}{C(i,i+1)}\Bigg| 
(CV_t)(r,r+1) 
\ &= \
 \Bigg| \frac{-tC(i,r)}{C(i,i+1)}\Bigg|  \\ 
\ &= \
 \Bigg| \frac{-tC(i,r)}{\sqrt{\alpha}}\Bigg|  
\quad \textnormal{by (C1) }\\
\ &\ <  \frac{\mu}{\sqrt{\alpha}}  
\qquad \qquad  \textnormal{by (C2)} \\
\ &\ <
\frac{\mu}{\sqrt{\mu}} = \sqrt{\mu }
\quad \textnormal{by (C3) . } 
\end{align*} 
CASE 2:  $r\notin \mathcal J$. \\  
Then $N(r,j')=0$ and 
\[
\Bigg| \frac{-tC(i,r)}{C(i,i+1)}\Bigg| 
|(CV_t)(r,j')|  \ 
\leq \ 
(1)||N-CV_t||_{\textnormal{max}} \ < \ 
2\sqrt{\mu} \ . 
\]

Using the estimates above to bound the third term of 
(\ref{threeterms}) by $(n-1)2\sqrt{\mu}$ , and then substituting bounds 
into (\ref{threeterms}), we get 
\begin{align*} 
||N-V_t^{-1}CV_t ||_{\textnormal{max}}  
< \sqrt{\mu }
+ \sqrt{\mu }
+
 {(n-1)2\sqrt{\mu}}  = 2n\sqrt{\mu} \ .
\end{align*} 
This finishes the proof for Step 3, and for  
 Stage 3. 

{\bf Stage 4.} 


%

Given $0<\epsilon < \mu$ and 
a matrix $C$  in $\mathcal M^0_{\mu}$ 
with the first $\beta h$ rows agreeing with those of 
$M(\epsilon)$ (except that the epsilons are allowed to 
be different positive numbers), we produce a path in 
$\mathcal M^0_{\mu}$ from $C$ to $M(\epsilon)$. 

The $n\times n$ matrix $C$ has a block form 
$C= 
\left(\begin{smallmatrix} X & 0 \\ Y & Z 
\end{smallmatrix} \right)$, in which $X$ is $\beta h\times \beta h$. 

STEP 1: Defining 
$C_t= 
\left(\begin{smallmatrix} X & 0 \\ tY & Z 
\end{smallmatrix} \right)$, $1\geq t\geq 0$, we 
show $(C_t)$ is a path of conjugate matrices. 
Conjugacy is clear for $0<t\leq 1$, where 
\[
C_t= 
\begin{pmatrix} 
(1/t)I & 0 \\ 0 & I 
\end{pmatrix} 
\begin{pmatrix} 
X & 0 \\ Y & Z 
\end{pmatrix} 
\begin{pmatrix} 
tI & 0 \\ 0 & I 
\end{pmatrix} \ . 
\] 
For $1\leq i \leq h$, let 
$\mathcal Z_i$ denote the set of indices from 
$\{1,\dots ,\beta h\}$ (which indexes rows and 
columns of $X$) which are congruent to $i$ 
mod $h$. 
To check conjugacy of $C_0$ and $C$, 
we first note there is a conjugacy of the form 
\[
C'= 
\begin{pmatrix} 
X & 0 \\ Y' & Z 
\end{pmatrix} 
=
\begin{pmatrix} 
I & 0 \\ L & I 
\end{pmatrix} 
\begin{pmatrix} 
X & 0 \\ Y & Z 
\end{pmatrix} 
\begin{pmatrix} 
I & 0 \\ -L & I 
\end{pmatrix} \  
\] 
such that column $i$ of $Y'$ is zero if $i\notin \mathcal Z_1$. 
This result from a composition of conjugations 
arising from  elementary
row and column operations as follows. In decreasing order for
$i=h-1,h-2,\dots , 1$:  for each $j\in \mathcal Z_i$, 
and for each $i'$ in 
$[\beta h +1,  n]$ such that position 
$(i',j+1)$ of the current matrix 
has  a nonzero entry, 
 we add a multiple of row $j$ to row $i'$, and 
then subtract the same multiple of 
column $i'$ from column $j$. 

We will show $Y'$ must be zero. For this, 
 consider row vectors in the form $(u,v)$, where 
$u$ has $\beta h$ entries and $v$ has $n-\beta h$ 
entries. 
For $1\leq t \leq h$, 
let $W_t$ be the subspace of $\mathbb R^{\beta h}$ 
spanned by the vectors $e_j$ such that  $j\in \mathcal Z_i$ 
and $1\leq i \leq t$. Let $W_0$ be the trivial 
space $\{0\}$. 
Given $(0,v)$ and $k\geq 1$, 
define  
 $(u^{(k)},v^{(k)})
=(0,v)C^k$.  
We claim for $1\leq k\leq h$ that 
\begin{enumerate} 
\item 
 $u^{(k)}\in  W_k$. 
\item 
If $u^{(1)}\neq 0$, then 
$u^{(k)}\in 
 W_k\setminus W_{k-1}$. 
\end{enumerate} 
The claim is clear for $k=1$, because column 
$j$ of $Y'$ is zero if $j\notin \mathcal Z_1$.  
Suppose $1\leq k<h$,  and 
the claim  holds for $k$. 
 Then 
\begin{align*} 
(0,v)C^{k+1}&= (u^{(k)},v^{(k)})C=
(u^{(k)},0)C + (0,v^{(k)})C \\ 
&= (u^{(k)}X,0) + 
(u'',v'')
\end{align*}  
for some $u''\in W_1$.
The  claim then holds for $k+1$ 
 because $X$ maps 
$ W_k\setminus W_{k-1}$ injectively into 
$ W_{k+1}\setminus W_{k}$ (since $k<h$) 
and $W_k$ contains $W_1$. 

If $Y'\neq 0$, then there must be some $v$ for which 
$ (0,v)C= (u^{(1)},v^{(1)})$ with 
$u^{(1)}\neq 0$. It follows from the claim that 
$(0,v)C^h$ is nonzero. That contradicts 
$C^h=0$. So, $Y'=0$, and $C_0$ is conjugate to $C$. 

STEP 2. 
We begin with 
$C= 
\left(\begin{smallmatrix} X & 0 \\ 0 & Z 
\end{smallmatrix} \right)$ 
in $\mathcal M^0_{\mu}$ and 
define a conjugacy in 
 $\mathcal M^0_{\mu}$ 
 to 
$M(\epsilon)$. 
Clearly, after applying a path 
in $\mathcal M^0_{\mu}$, we can suppose that 
$X$ exactly equals the corresponding 
$\beta h \times \beta h$
block in  
 $M(\epsilon )$, and write 
 $M(\epsilon )$ as 
$ 
\left(\begin{smallmatrix} X & 0 \\ 0 & J  
\end{smallmatrix} \right)$, where 
$J$ is the matrix  of 
Definition \ref{Jdefinition}.
  Let $m=n-\beta h$. 

The conjugacy of $C$ and $M(\epsilon )$  
and the fact that $J$ has a Jordan 
zero block imply that there is a matrix 
$U \in \textnormal{SL}(m, \mathbb R)$ such that    
$U^{-1}ZU=J$.  Then there is a path 
$(U_t)$ in  
$ \textnormal{SL}(m, \mathbb R)$ from $I$ to 
$U$, giving a path of conjugate matrices 
from $Z$ to $J$, 
$Z_t= U_t^{-1}ZU_t$. Given $\delta >0$, 
we can follow a path 
$\left(\begin{smallmatrix} X & 0 \\ 0 & sZ 
\end{smallmatrix} \right)$, $1\geq s\geq \delta$,  
with a path 
$\left(\begin{smallmatrix} X & 0 \\ 0 & sZ_t 
\end{smallmatrix} \right)$, $1\geq s\geq \delta$,  
and then a path 
$\left(\begin{smallmatrix} X & 0 \\ 0 & sJ 
\end{smallmatrix} \right)$, $\delta \leq s\leq 1$. 
This gives a path of conjugate matrices from 
$C$ to $M(\epsilon )$, and with  $\delta $ 
small enough,  the path is contained in  
$\mathcal M^0_{\mu}$.   
  This completes the proof of 
Stage 4, and the lemma. 

\end{proof}

\end{document}